\documentclass[10pt]{amsart}

\usepackage{amsmath}
\usepackage{amsthm}
\usepackage{amsfonts}
\usepackage{amsxtra}
\usepackage[all]{xy}
\usepackage{verbatim}
\usepackage{color}
\usepackage{amssymb}
\usepackage{mathrsfs}
\usepackage{hyperref}
\usepackage{tikz-cd}
\usepackage{mathrsfs}
\usepackage[shortlabels]{enumitem}
\usepackage{appendix}

\theoremstyle{plain}
\newtheorem{thm}{Theorem}[section]
\newtheorem{Cor}[thm]{Corollary}
\newtheorem{Lem}[thm]{Lemma}
\newtheorem{Prop}[thm]{Proposition}

\newtheorem{Thm}[thm]{Theorem}
\newtheorem*{Thm*}{Theorem}
\newtheorem*{Cor*}{Corollary}
\newtheorem*{Lem*}{Lemma}
\newtheorem*{Prop*}{Proposition}
\newtheorem*{Claim*}{Claim}

\theoremstyle{definition}
\newtheorem{Def}[thm]{Definition}
\newtheorem{Constr}[thm]{Construction}

\theoremstyle{remark}
\newtheorem{Rem}[thm]{Remark}
\newtheorem{Ex}[thm]{Example}
\newtheorem{Obs}[thm]{Observation}
\newtheorem*{Rem*}{Remark}
\newtheorem*{Ex*}{Example}
\newtheorem*{Def*}{Defintion}

\newcommand{\Set}{\textsf{\textup{Set}}}
\newcommand{\Fin}{\textsf{\textup{Fin}}}
\newcommand{\Top}{\textsf{\textup{Top}}}
\newcommand{\Topu}{\textsf{\textup{\underline{\smash{Top}}}}}
\newcommand{\Cat}{\textsf{\textup{Cat}}}
\newcommand{\Sp}{\textsf{\textup{Sp}}}
\newcommand{\CW}{\textsf{\textup{CW}}}
\newcommand{\Sub}{\textsf{\textup{Sub}}}

\newcommand{\Fun}{\textup{Fun}}
\newcommand{\Map}{\textup{Map}}
\newcommand{\Sing}{\textup{Sing}}
\newcommand{\id}{\textup{id}}
\newcommand{\Id}{\textup{Id}}

\newcommand{\res}{\textup{res}}
\newcommand{\ind}{\textup{ind}}
\newcommand{\coind}{\textup{coind}}
\newcommand{\ob}{\textup{ob}}
\newcommand{\mor}{\textup{mor}}
\newcommand{\im}{\textup{im}}
\newcommand{\pr}{\textup{pr}}
\newcommand{\Uni}{\textsf{\textup{Uni}}}
\newcommand{\Tr}{\textsf{\textup{Tr}}}
\newcommand{\Ind}{\textsf{\textup{Ind}}}
\newcommand{\Fam}{\textsf{\textup{Fam}}}
\newcommand{\Aa}{\mathcal{A}}
\newcommand{\CC}{\mathcal{C}}
\newcommand{\EE}{\mathcal{E}}
\newcommand{\FF}{\mathcal{F}}
\newcommand{\II}{\mathcal{I}}
\newcommand{\JJ}{\mathcal{J}}
\newcommand{\KK}{\mathcal{K}}

\newcommand{\OO}{\mathcal{O}}
\newcommand{\PP}{\mathcal{P}}
\newcommand{\TT}{\mathcal{T}}
\newcommand{\UU}{\mathcal{U}}
\newcommand{\VV}{\mathcal{V}}
\newcommand{\A}{\mathscr{A}}
\newcommand{\B}{\mathscr{B}}
\newcommand{\C}{\mathscr{C}}
\newcommand{\D}{\mathscr{D}}
\newcommand{\E}{\mathscr{E}}
\newcommand{\F}{\mathscr{F}}
\newcommand{\I}{\mathscr{I}}
\newcommand{\V}{\mathscr{V}}

\newcommand{\ff}{\mathbb{F}}
\newcommand{\lL}{\mathbb{L}}
\newcommand{\pp}{\mathbb{P}}
\newcommand{\rr}{\mathbb{R}}
\newcommand{\sS}{\mathbb{S}}
\newcommand{\uu}{\mathbb{U}}
\newcommand{\zz}{\mathbb{Z}}

\let\oldmarginpar\marginpar
\renewcommand\marginpar[1]{\-\oldmarginpar[\raggedleft\footnotesize #1]%
{\raggedright\footnotesize #1}}

\hypersetup{
    colorlinks=true,
    linkcolor=blue,
    filecolor=blue,      
    urlcolor=blue,
    citecolor=blue,
    }

\title{Segal models for equivariant incomplete infinite loop spaces}

\author{Tjark Bantelmann}

\begin{document}

\begin{abstract}
We model equivariant infinite loop spaces indexed on incomplete universes via suitable equivariant analogs of $\Gamma$-spaces. The choice of universe dictates a transfer system which in turn dictates the Segal condition on equivariant $\Gamma$-spaces. Equivariant $\Gamma$-spaces themselves come in different but equivalent guises interpolating between categories $\Gamma$ as defined by Segal and $\Gamma_G$ as defined by Shimakawa. The main application is the construction of Segal $K$-theory of normed permutative categories.
\end{abstract}

\maketitle

\tableofcontents


\section{Introduction}

An infinite loop space machine is a functor taking some input object to a connective $\Omega$-spectrum. The input of the Segal machine as originally defined by Segal \cite{Seg74} is a functor from the category of finite based sets to the category of based spaces also called a \textit{$\Gamma$-space} which fulfils a \textit{specialness condition}, sometimes also called \textit{Segal condition}. Throughout this paper, $G$ will denote a finite group. In $G$-equivariant homotopy theory there are many different choices of output categories of infinite loop space machines. Naturally one must modify the input categories and the specialness condition to obtain the correct type of output. Shimakawa \cite{Shi89} defines an infinite loop space machine where the input objects are \textit{$\Gamma_G$-$G$-spaces}. These are functors from the category of finite based $G$-sets, denoted by $\Gamma_G$, to the category of based $G$-spaces which are enriched in the category of based $G$-spaces together with a modified Segal condition. The output spectra of the machine defined by Shimakawa are positive connective genuine $\Omega$-$G$-spectra. There is also a naive variant taking \textit{$\Gamma$-$G$-spaces} to naive $\Omega$-$G$-spectra. Besides naive and genuine $G$-spectra there are incomplete variants of $G$-spectra sitting in between. These incomplete variants are indexed on incomplete $G$-universes. We aim to develop equivariant incomplete infinite loop space machines which map variants of equivariant incomplete $\Gamma$-spaces to connective positive $\Omega$-$G$-spectra indexed on incomplete $G$-universes. We will mainly follow the detailed modernisation of the naive and genuine equivariant infinite loop space machines described by May, Merling and Osorno in \cite{MMO25}.

\textbf{Outline.} To understand these incomplete infinite loop space machines we need to understand the relationship between finite $G$-sets appearing in the input and representations appearing in the output. After discussing some preliminaries on $G$-spaces and $G$-spectra in Section \ref{sec:preliminaries} this relationship will be studied in Section \ref{sec:isotropy}.

Each $G$-universe $U$ determines a full subcategory $\II_U(G)$ of the category of finite based $G$-sets consisting of those finite $G$-sets $T$ such that there is a $G$-embedding $T\hookrightarrow U$. These are called \textit{admissible} $G$-sets. Actually, this defines a \textit{$G$-indexing system} $\II_U$ in the sense of \cite{BH15}. Indexing systems are collections of full subcategories $\II(H)$ of the category of finite $H$-sets for all $H\subset G$ which fulfil a certain list of axioms. For example, the subcategories $\II(H)$ and $\II(G)$ relate via restriction and induction. As shown by Blumberg and Hill \cite{BH15} the collections $\II_U(H)$ consisting of those finite $H$-sets $S$ such that there is an $H$-embedding $S\hookrightarrow \res_H^G U$ form an indexing system. Universes also define isotropy families $\FF_U=\{G_u\mid u\in U\}$ where $G_u$ is the stabilizer subgroup. Isotropy families are a space level concept, whereas indexing systems arise in stable homotopy theory. Note that not all indexing systems arise this way, there are indexing systems which do not arise as admissible sets of universes (\cite{Rub21a}, \S 3 and Example 4.9). There are also families which do not arise as admissible sets of universes as we will show in Section \ref{sec:isotropy-families}. The problem is even more nuanced. All indexing systems which arise through admissible sets of $G$-universes are \textit{disk-like} (\cite{Rub21a}, Lem.A.7). A disk-like indexing system is an indexing system which is determined by the $G$-level. This means that $\II(H)$ is obtained from $\II(G)$ by restriction. Not all indexing systems are disk-like and not every disk-like indexing system arises as the admissible sets of a universe (see \cite{Rub21a}, Theorem 4.15). Every disk-like indexing system dictates a family, but not every family arises from a disk-like indexing system. But this is not really a problem. We will call an indexing system \textit{compatible} with a $G$-universe $U$ if the isotropy family of $U$ agrees with the family $\FF_\II=\{H\mid G/H\in \II(G)\}$. All of our results will work for such compatible pairs of indexing systems and universes.

With this in mind we can define the input and output of our infinite loop space machines. For each disk-like indexing system $\II$ we define a full subcategory $\Gamma_\II\subset \Gamma_G$ consisting of those finite based $G$-sets $T\cong S_+$ such that $S\in\II(G)$. This definition only depends on the $G$-level of the indexing system hence the disk-like assumption. For $\II=\TT$ the trivial indexing system $\Gamma_\TT$ recovers Segal's \cite{Seg74} original category $\Gamma$ with trivial $G$-action\footnote{Noting that Segal's $\Gamma^\textup{op}$ is what we call $\Gamma$.} and setting $\II=\CC$ the complete indexing system we recover Shimakawa's \cite{Shi89} $\Gamma_G$. A $\Gamma_\II$-$G$-space is defined to be a functor from $\Gamma_\II$ to the category of based $G$-spaces which is enriched in based $G$-spaces. We will discuss these in Section \ref{sec:equivariant-Gamma}. The output will be orthogonal $G$-spectra indexed on universes as introduced by Mandell and May \cite{MM02}. An orthogonal $G$-spectrum $E$ is a \textit{positive $\Omega$-$G$-spectrum} if its adjoint structure maps $E(V)\to \Omega^W E(V\oplus W)$ are weak $G$-equivalences for all finite dimensional $G$-representations such that $V^G\ne 0$ and it is called \textit{connective} if all its underlying fixed point spectra are connective. Denote the category of connective positive orthogonal $\Omega$-$G$-spectra by $\Omega\Sp_{\textup{pc}}^{G,U}$. For the output spectra to be positive $\Omega$-$G$-spectra we need an additional condition on our input $\Gamma_\II$-$G$-spaces, namely that they are \textit{$\II$-special}. The idea is to homotopically invert a selection of so called \textit{Segal maps}. This specialness notion generalises the notion of special $\Gamma_G$-$G$-spaces as defined by Shimakawa \cite{Shi89} when setting $\II=\CC$ and will be discussed in detail in Sections \ref{sec:specialness}, \ref{sec:transfer-maps-Gamma} and \ref{sec:comparison}. Denote the category of $\II$-special $\Gamma_\II$-$G$-spaces by $\Gamma_\II[G\Topu_{\textup{wp}}]^{\II\textup{-spc}}$.

\begin{Def}
	Let $\II$ be a disk-like indexing system and let $U$ be a compatible $G$-universe.
	An \textit{equivariant incomplete Segal infinite loop $G$-space machine} is a functor
	\[\sS_\II^{G,U}\colon \Gamma_\II[G\Topu_{\textup{wp}}]^{\II\textup{-spc}}\to \Omega\Sp_{\textup{pc}}^{G,U}\]
	together with a natural map of $G$-spaces $\nu \colon X_1 \to \sS_\II^{G,U}(S^0)$ such that the composite with the adjoint structure map of the positive $\Omega$-$G$ spectrum 
	\[\sS_\II^{G,U}X(S^0)\to \Omega^V\sS_\II^{G,U}X(S^V)\]
	is a group completion for all $V$ in $U$ such that $V^G\ne 0$.
\end{Def}

In the complete and genuine case this definition of infinite loop $G$-space machine agrees with \cite{MMO25}, Definition 0.6, up to changing the input category to an equivalent category which allows for the comparison to the operadic machine. We will omit a comparison between the incomplete operadic and Segal machine in this paper. One has to deal with the fact that the inputs of the operadic machines depend on universe data and of our Segal machines depend on indexing data. We intend to return to this issue in a subsequent paper.

Having established these definitions we can state our first main theorem.

\begin{Thm}[Theorem A]
	Let $\II$ be a disk-like indexing system and let $U$ be a compatible $G$-universe. Then there exists an equivariant incomplete Segal infinite loop $G$-space machine
	\[\sS^{G,U}_\II\colon \Gamma_\II[G\Top_*]^{\II\textup{-spc}} \to \Omega\Sp_{\textup{pc}}^{G,U}.\]
\end{Thm}

To proof this we shall first construct in Section \ref{sec:construction-Segal-machine} a functor taking arbitrary $\Gamma_\II$-$G$-spaces to orthogonal $G$-spectra
\[\hat{\sS}_\II^{G,U}\colon \Gamma_\II[G\Topu_*]\to \Sp^{G,U}.\]
The functor $\hat{\sS}_\II^{G,U}$ is given by a bar construction, then prolonging to $G$-CW complexes with isotropy in $\FF_U$ and then restricting to representation spheres. We then sketch to show in Section \ref{sec:proof-strategy-theorem-A} that after restriction to $\II$-special $\Gamma_\II$-$G$-spaces this functor gives the desired Segal machine. To do so we need to show the delooping property and the group completion property. The proof strategy of the delooping property follows the genuine case originally described by Shimakawa \cite{Shi89}, though we will follow the modernised and more detailed version of the proof given by May, Merling and Osorno in \cite{MMO25}. The key difference between the genuine and the incomplete version is that most definitions and constructions are done with respect to \textit{a selection} of finite $G$-sets and \textit{a selection} of $G$-CW complexes. The main difficulty then is to keep track that all these constructions stay inside the prescribed selections. Closure properties of these selections are discussed in Section \ref{sec:finite-G-sets} and \ref{sec:F-G-CW-cxs}. Hence, we will mainly describe how to modify the genuine proofs so that they work in the incomplete setting and refer to the appendix for more detailed proofs. In Section \ref{sec:group-completion} we will show the group completion property. It only depends on the underlying naive $G$-spectrum so that we can deduce it from well known results.

Any infinite loop space machine should be able to model suspension $G$-spectra. We will show in Section \ref{sec:suspension} that given a based $G$-space $A$ one can define an input $\Gamma_\II$-$G$-space $^\bullet A$ such that there is a weak $G$-equivalence
\[\Sigma^\infty_U A \simeq \hat\sS_\II^{G,U}\text{}^\bullet A\]
where $\Sigma^\infty_U A$ denotes the suspension $G$-spectrum of $A$ indexed on $U$. We use $\hat\sS_\II^{G,U}$ instead of the proper Segal machine since suspension spectra are almost never positive $\Omega$-spectra. Hence, the input cannot be expected to be $\II$-special. The idea of this comparison follows the genuine case as described by Guillou, May, Merling and Osorno in \cite{GMMO19b}. In Section \ref{sec:Eilenberg-Mac-Lane} we will briefly talk about Eilenberg-MacLane $G$-spectra. We are able to construct genuine Eilenberg-MacLane $G$-spectra of abelian $G$-groups as in \cite{GMMO19b}. One shortcoming of our Segal machine is that it is not able to model incomplete Eilenberg-MacLane $G$-spectra of $\II$-incomplete Mackey functors. The problem is that the target category of our equivariant $\Gamma$-spaces, the category of based $G$-spaces, is not able to host incomplete monoid data. For this one is required to use coefficient systems. We intend to return to this issue in future work.

The main application of our infinite loop space machine is equivariant $K$-theory of $\II$-normed permutative categories, \textit{$\II$-NPCs} for short. These were originally defined by Rubin in \cite{Rub25}. We will describe an internal analog of $\II$-NPCs in Section \ref{sec:I-NPCs}. Following the non-equivariant blueprint described in \cite{May78}, given an $\II$-NPC internal to based $G$-spaces $\A$, we will define a $\Gamma_\II$-$G$-space $B\overline{\A}$ in Section \ref{sec:Gamma_I-G-cats} and \ref{sec:construction-K-theory}. We will show in Section \ref{sec:proof-K-theory} that this gives rise to the Segal $K$-theory of $\A$. This is the content of the second main theorem of this paper.

\begin{Thm}[Theorem B]
	Let $G$ be a finite group, $\II$ a disk-like $G$-indexing system and let $\A$ be an $\II$-normed permutative $G$-category internal to the category of based $G$-spaces. For any compatible choice of $G$-universe $U$ the positive connective $\Omega$-$G$-spectrum
	\[K_\II^{G,U}\A := \sS^{G,U}_\II B\overline\A\]
	gives rise to the Segal $K$-theory of the $\II$-normed permutative $G$-category $\A$.
\end{Thm}


\section{Preliminaries}\label{sec:preliminaries}

\subsection{Preliminaries on \texorpdfstring{$G$}{G}-spaces}

In the following we will introduce some notations and recall basic properties of $G$-spaces, see for instance \cite{Blu17} Chapter 1 or \cite{May96} Chapter 1. Denote by $\Top$ the category of compactly generated weak Hausdorff spaces, by $\Fin$ the category of finite sets and by $\Top_*$ and $\Fin_*$ the based analogs. For $G$ a finite group denote by $G\Top$ respectively $G\Top_*$ and the category of compactly generated weak Hausdorff (left) $G$-spaces respectively based (left) $G$-spaces with $G$-maps respectively based $G$-maps between them. The full subcategories of finite based $G$-sets are denoted by $G\Fin$ respectively $G\Fin_*$. The categories $G\Top$ and $G\Top_*$ are complete and cocomplete. Denote the subspace of all $G$-maps by $\Map_G(X,Y)$ and the subspace of all based $G$-maps by $F_G(X,Y)$.

We will be primarily interested in the self enriched category $G\Topu_*$ where the based mapping $G$-spaces $F(X,Y)$ consist of all based maps $f\colon X\to Y$ between based $G$-spaces $X,Y$ with conjugation $G$-action on morphisms. Denote this $G$-action by $g.f(x):=gf(g^{-1}x)$. This space is based at the unique map factoring through the base point $X\to *\to Y$. Analogously, in the unbased case denote the space of all unbased maps with conjugation $G$-action by $\Map(X,Y)$. The category $G\Topu_*$ is closed symmetric monoidal with respect to the smash product and $G\Topu$ with respect to the cartesian product which implies that
\[F(X\land Y,Z)\cong F(X,F(Y,Z)) \text{ and } \Map(X\land Y,Z)\cong \Map(X,\Map(Y,Z)).\]
Any based $G$-space has an underlying unbased $G$-space by forgetting the base point. Hence $G\Top_*$ is also symmetric monoidal with respect to the cartesian product, but this structure is not closed symmetric monoidal. The unit of the smash product is $S^0$ and the unit of the cartesian product is the one point set $*$. Hence $F(*,X)\cong *$ and $F(S^0,X)\cong X$.

Any space is a $G$-space with trivial $G$-action. Considering the finite group $G$ as a discrete space it is both a left and right $G$-space with group $G$-action given by the group operation. In addition, $G$-orbits $G/H$ are discrete $G$-spaces with action on cosets. Hence, these are also objects in $G\Fin$. We can adjoin a disjoint base point to any unbased set to make it based. This functor $(-)_+\colon G\Top\to G\Top_*$ is adjoint to the forgetful functor mentioned above. Given a (based) $G$-space $X$ and a subgroup $H\subset G$ its fixed point subspace $X^H$ naturally becomes a $W_GH$-space where $W_GH=N_GH/H$ is the Weyl group. Usually we will forget this residual action and consider $X^H$ as just a space.

Given a sequence of subgroups $K\subset H\subset L\subset G$ and an $H$-space $X$ we can forget the $H$-action to a $K$-action and consider $X$ as a $K$-space which will be denoted by $\res^H_KX$. We can induce $X$ up to an $L$-space by setting
\[\ind_H^LX:=G\times_HX.\]
Here, $\times_H$ denotes the balanced product which is defined for right $H$-spaces $A$ and left $H$-spaces $B$ by the quotient of $A\times B$ where $(ah,b)$ is identified with $(a,hb)$ for all $h\in H$. In the case above $X$ is naturally a left $G$-space and $G$ is considered as a right $H$-space. The $L$-action of $\ind_H^L X$ then is given by $\ell[g,x]=[g\ell ^{-1},x]$ for $\ell\in L$ and $[g,x]\in G\times_HX$. Analogously, one defines $\ind_H^LX:=G_+\land_H X$ for $X$ a based $G$-space. Another way to obtain an $L$-space from $X$ is via the coinduced $L$-space
\[\coind_H^LX:=\Map(G,X)^H\]
with $L$-action given by $\ell f(g)=f(g\ell^{-1})$ for $\ell\in L$, $g\in G$, and $f\in \Map(G,X)^H$. Note that restriction, induction and coinduction naturally extend to functors which form an adjoint triple $\ind\vdash \res\vdash \coind$.

Let $\Delta$ be the simplicial category with objects denoted by $[n]$. Presheafs on $\Delta$ in some category $\C$ are called \textit{simplicial objects in $\C$} and this presheaf category will be denoted by $\textsf{s}\C$. For instance, \textit{simplicial $G$-spaces} are elements in the category $\textsf{s}G\Top=\Fun(\Delta^{\textup{op}},G\Top)$. If there is a covariant functor $\Delta_*\colon\Delta\to\C$ and if $\C$ is symmetric monoidal we denote the geometric realization of a simplicial object $X$ in $\C$ by
\[|X|=X\otimes_\Delta\Delta_*=\int^{[n]\in\Delta}X_n\otimes\Delta_n.\]
This extends to a functor $|-|\colon \textsf{s}\C\to \C$ if $\C$ is cocomplete as is the case for $\C=G\Top_*$.
	
A \textit{$G$-homotopy} is a homotopy $H\colon X\times I\to Y$ in the usual sense so that $H$ is a $G$-map where $G$ acts on $I$ trivially. In the based setting $H$ is a based $G$-map $X\land I_+\to Y$. This way one obtains a notion of $G$-homotopy equivalence of (based) $G$-spaces. Similarly one defines \textit{weak $G$-equivalences} of based $G$-maps $f\colon X\to Y$ by requiring that the induced maps $\pi_n(f^H)\colon\pi_n(X^H)\to \pi_n(Y^H)$ are isomorphisms for all subgroups $H\subset G$. Clearly, any $G$-homotopy equivalence is a weak $G$-homotopy equivalence.

In the following we will follow \cite{MMO25} \S 1.1 and \S 1.2. A \textit{Hopf $G$-space} is a based $G$-space $X$ together with a product $G$-map $\mu\colon X\times X\to X$, such that the base point $*$ of $X$ is a two-sided unit in the sense that  left or right multiplication by $*$ is a weak $G$-equivalence $X\to X$. Note that if $X$ is a Hopf $G$-space then $X^H$ is a $W_GH$-Hopf space and in particular a Hopf space in the classical sense for any subgroup $H\subset G$.

A $G$-map $i\colon A\to X$ fulfilling the \textit{$G$-homotopy extension property} is called a $G$-cofibration. Since we are working with compactly generated weak Hausdorff spaces any $G$-cofibration is a closed inclusion. If $i$ is a subspace inclusion then it is a $G$-cofibration if and only if it is a $G$-NDR pair. Using this criterion any $G$-cofibration is also an $H$-cofibration after restriction, and a $W_GH$-cofibration after passage to fixed points. Pushouts of $G$-maps along a $G$-cofibrations are again $G$-cofibrations. All of this can also be done in the based setting.

\begin{Def}
	A based $G$-space $X$ is called \textit{well pointed} if the base point inclusion $*\hookrightarrow X$ is a $G$-cofibration or equivalently if the pair $(X,*)$ is a $G$-NDR pair. We will denote the full subcategory of $G\Top_*$ of well pointed $G$-spaces by $G\Top_{\textup{wp}}$.
\end{Def}

\begin{Rem}\label{not:well-pointed-spaces}
	We will denote the full subcategory of $G\Topu_*$ of well pointed $G$-spaces by $G\Topu_{\textup{wp}}$. Note that the underline notation usually means that the category is cartesian closed, but this is not the case here. The based mapping space of well pointed spaces needs not to be well pointed itself (cf. \cite{MMO25}, Footnote 8).
\end{Rem}

We will frequently work with based $G$-spaces which arise as the geometric realization of simplicial based $G$-spaces. The following definition and statements are ambiguously stated for both based and unbased $G$-spaces.

\begin{Def}
	Let $X_*$ be a simplicial $G$-space with $G$-space of $n$-simplices $X_n$. The \textit{$n$-th latching space of $X$} is given by
	\[L_nX=\bigcup_{i=0}^{n-1}s_i(X_{n-1}).\]
	It is a $G$-space, and the inclusion $L_nX\to X_n$ is a $G$-map. We say that $X_*$ is \textit{Reedy cofibrant} if this inclusion is a $G$-cofibration for all $n\ge 0$.
\end{Def}

\begin{Lem}[\cite{MMO25}, Lemma 1.11]\label{lem:MMO-lem-1.11}
	A simplicial $G$-space $X_*$ is Reedy cofibrant if all degeneracy operators $s_i$ are $G$-cofibrations.
\end{Lem}

\begin{Thm}[\cite{MMO25}, Theorem 1.12]\label{thm:MMO-thm-1.12}
	Let $f_*:X_*\to Y_*$ be a map of Reedy cofibrant simplicial $G$-spaces such that each $f_n$ is weak $G$-equivalence. Then the realization $|f_*|\colon|X_*|\to |Y_*|$ is a weak $G$-equivalence.
\end{Thm}

\begin{Thm}[\cite{MMO25}, Lemma 1.13]\label{thm:MMO-thm-1.13}
	Let $f_*\colon X_*\to Y_*$ be a map of Reedy cofibrant simplicial $G$-spaces such that each $f_n$ is a $G$-cofibration. Then the realization $|f_*|\colon|X_*|\to |Y_*|$ is a $G$-cofibration.
\end{Thm}

\subsection{Preliminaries on diagram spaces}

One major nuance in the study of $\Gamma$-spaces both equivariantly and non-equivariantly is the relationship between based and unbased spaces. We will recite the results presented in \cite{MMO25}, \S 1.3. For $\V$ a symmetric monoidal category we will call a $\V$-enriched category a $\V$-category and similarly $\V$-enriched functors and $\V$-enriched natural transformations $\V$-functors and $\V$-natural transformations.

\begin{Rem}\label{rem:MMO-rem-1.16}
	For a $G\Top$-category $\C$ we can add a disjoint base point to the hom objects of $\C$ to form a $G\Top_*$-category $\C_+$. Given a $G\Top_*$-category $\D$ we can forget base points to regard $\D$ as a $G\Top$-category $\uu\D$. Via the adjunction $(-)_+\vdash \uu$, we can identify $G\Top$-functors $\C\to \uu\D$ and $G\Top_*$-functors $\C_+\to\D$
	\[G\Top(\C,\uu\D)\cong G\Top_*(\C_+,\D).\]
\end{Rem}

All variants of the Segal machine start with a category $\E$ a priori enriched over $G\Top$ such that $\E$ has a zero object. One then assigns the morphism $n\to 0\to m$ in $\E(n,m)$ to be the base point making $\E$ into a category enriched over $G\Top_*$ as composition descends to a map out of the smash product. We will be concerned with $G\Top_*$-functors out of $\E$.

\begin{Lem}[\cite{MMO25}, Lemma 1.17]\label{lem:MMO-lem.1.17}
	Let $\E$ be a $G\Top$-category with zero object $0$, which we consider as a $G\Top_*$-category as above. Then,
	\begin{enumerate}[(i)]
		\vspace{-1mm}\item for any $G\Top_*$-functor $X\colon \E\to G\Topu_*$, $X(0)=*$;
		\item conversely, any $G\Top$-functor $Y\colon \uu\E\to G\Topu$ such that $Y(0)=*$ can be considered as a $G\Top_*$-functor $\E\to G\Topu_*$
	\end{enumerate}
\end{Lem}

\begin{Rem}
	Let $\E$ be a $G\Top_*$-category with a zero object. A $G\Top$-functor $Y\colon \uu\E\to G\Topu_*$ is said to be \textit{reduced} if $Y(0)$ is a point. Thus the previous lemma can be interpreted as saying that there is a one-to-one correspondence between reduced $G\Top$-functors $\uu\E\to G\Topu$ and $G\Top_*$-functors $\E\to G\Topu_*$.
\end{Rem}

\begin{Def}\label{def:tensor-product-of-functors}
	Let $\E$ be a $G\Top_*$-category with zero object $0$ and let $X$ and $Y$ be covariant and contravariant $G\Top_*$-functors $\E\to G\Topu_*$. Then the \textit{tensor product of functors}\footnote{Here we follow the convention in \cite{MMO25} to write $Y\otimes_\E X$ instead of $X\otimes_\E Y$.} $Y\otimes_\E X$ is defined as the following coend
	\[\int^{a\in \E}Y(a)\land X(a)\]
	which is the coequalizer of the following diagram
	\[\bigvee_{a,b\in \E} Y(b)\land \E(a,b)\land X(a) \rightrightarrows \bigvee_{c\in \E} Y(c)\land X(c)\]
	where the arrows are given by evaluation maps. More precisely, a triple $(y,\phi,x)$ is mapped to the tuple $(Y(\phi)(y),x)\in Y(a)\land X(a)$ by the one map and $(y,X(\phi)(x))\in Y(b)\land Y(b)$ by the other map.
\end{Def}

\begin{Def}\label{rem:prolongation-existence}
	Let $\iota\colon \E\to\D$ be a $G\Top_*$-functor between $G\Top_*$-categories with small skeleta\footnote{This assumption will always be fulfilled in our applications.}. Consider the functor category $G\Top_*(\D,G\Topu_*)$. There is a natural inclusion functor given by precomposing the functor $\iota$
	\[\uu_\E^\D\colon G\Top_*(\D,G\Topu_*)\to G\Top_*(\E,G\Topu_*).\]
	By abstract non-sense (see \cite{MMSS01}, \S 3, \S 23) this functor has a left adjoint \textit{prolongation} functor $\pp_\E^\D$ given by a left Kan extension $\textup{Lan}_\iota X$. This has a description via the coend
	\[\textup{Lan}_\iota X(b)=\int^{a\in E}\D(a,b)\land X(a)\]
	so that we may write this as the level wise tensor product $\pp_\E^\F X(b)=\D(-,b)\otimes_\E X$.
\end{Def}

\subsection{Preliminaries on orthogonal \texorpdfstring{$G$}-spectra}

We will call a finite dimensional real inner product space $V$ together with a linear and isometric $G$-action a \textit{$G$-representation}. Examples include the $n$-dimensional trivial $G$-representation $\rr^n$ with trivial $G$-action and the group ring $\rr[G]$ with $G$-action on generators. For $G=C_n$ the cyclic group one obtains $2$-dimensional representations by rotations of the plane by roots of unity $\exp(2\pi ik/n)$. Denote these by $V_n(k)$. Note that $V_n(0)\cong \rr\oplus \rr$ and $V_n(n/2)\cong \sigma\oplus\sigma$ if $n$ is even and where $\sigma$ denotes the $1$-dimensional sign representation. 

A \textit{$G$-universe} $U$ is a countable direct sum of $G$-representations such that $U$ contains all trivial representations and if $V$ is an irreducible $G$-representation in $U$, then $U$ contains countably many copies of $V$ (cf. \cite{May96}, Definition 2.1). A $G$-universe $U$ is called \textit{trivial} if it only contains trivial representations. A $G$-universe $U$ is called \textit{complete} if, up to isomorphism, it contains all irreducible representations of $G$. An example of a trivial $G$-universe is $\rr^\infty$ and of an example of a complete $G$-universe is $\rr[G]^\infty$. A $G$-universe is \textit{incomplete} if it is not complete. A non-trivial example of an incomplete universe is $V_5(0)^\infty\oplus V_5(1)^\infty$ since the representation $V_5(2)$ is missing.

Two $G$-universes are called isomorphic if there exists a $G$-equivariant linear isometric isomorphism between them. Denote by $\Uni(G)$ the set of isomorphism classes $[U]$ of $G$-universes $U$ of $G$. This naturally forms a poset. Rubin \cite{Rub21a} showed that $\Uni(G)$ possesses more than just a poset structure but also the structure of a finite lattice which is isomorphic to an $n$-dimensional cube where $n$ is the number of non-trivial irreducible $G$-representations up to isomorphisms. 

Let $\I_U$ be the category whose objects are $G$-representations $V$ in some $G$-universe $U$ and whose morphisms are all linear isometric maps $V\to W$ which are not necessarily $G$-maps with conjugation $G$-action so that $\I_U$ is naturally $G\Top$-enriched. Let $\I_{U_+}$ be $\I_U$ with disjoint base points added to the hom-spaces making it $G\Top_*$-enriched.

The following definitions are due to \cite{MM02}, Chapter II.2.

\begin{Def}
	An \textit{$\I_U$-$G$-space} is a $G\Top$-functor $\I_{U}\to G\Topu_*$ or equivalently (by Remark \ref{rem:MMO-rem-1.16}) a $G\Top_*$-functor $\I_{U_+}\to G\Topu_*$.
\end{Def}

The principal example of an $\I_U$-$G$-space is the \textit{sphere $\I_U$-$G$-space} $S$ which on some $G$-representation $V$ in $U$ is given by $S(V)=S^V$. 

There is an \textit{external smash product} of $\I_{U_+}$-$G$-spaces $X,Y\colon\I_{U_+}\to G\Top_*$ is given by 
\[X \bar{\land} Y\colon \I_{U_+}\times\I_{U_+}\to G\Top_* \text{ via } (X\bar{\land}Y)(V,W):= X(V)\land Y(W).\]

\begin{Def}
	An \textit{orthogonal $G$-spectrum indexed on $U$} consists of an $\I_{U_+}$-$G$-space $E$, together with a $G\Top_*$-natural transformation $E\bar{\wedge}S\Rightarrow E\circ\oplus$ between $G\Top_*$-functors $\I_{U_+}\times\I_{U_+}\to G\Top_*$.
\end{Def}

Hence, the data of an orthogonal $G$-spectrum $E$ consists of based $G$-spaces $E(V)$ for any $G$-representation $V$ in $U$, morphism $G$-maps
\[\I_U(V,W)\to G\Top_*(E(V),E(W)),\]
and structure $G$-maps $\sigma\colon E(V)\land S^W\to S^{V\oplus W}$ natural in $V$ and $W$. We call an orthogonal $G$-spectrum $E$ \textit{naive}\footnote{This is sometimes called \textit{classical}, for example in \cite{MMO25}.} if it is indexed on a trivial universe and it is called \textit{genuine} if it is indexed on a complete $G$-universe.

\begin{Def}[\cite{MMO25}, Definition 1.21]
	Let $E$ be an orthogonal $G$-spectrum indexed on a universe $U$. Then $E$ is called an \textit{$\Omega$-$G$-spectrum} if the adjoint structure maps
	\[\tilde{\sigma}\colon E(V)\to \Omega^WE(V\oplus W)\]
	are weak $G$-equivalences for all $G$-representations $V,W$ in $U$ and $E$ is called a \textit{positive $\Omega$-$G$-spectrum} if this is the case for all $G$-representations $V,W$ such that $V^G\ne\emptyset$. The orthogonal $G$-spectrum $E$ is called \textit{connective} if the underlying fixed point spectra are connective in the classical sense for all subgroups $H\subset G$, that is, all homotopy groups vanish in negative degrees. Denote the category of orthogonal $G$-spectra indexed on a $G$-universe $U$ by $\Sp^{G,U}$, the category of orthogonal $\Omega$-$G$-spectra by $\Omega\Sp^{G,U}$, and positive $\Omega$-$G$-spectra by $\Omega\Sp_{\textup{p}}^{G,U}$. We will decorate these categories with a little $\textup{c}$ to denote connective spectra.
\end{Def}

\begin{Def}
	The $0$-th space of a positive connective orthogonal $G$-spectrum indexed on some $G$-universe $U$ is called an \textit{infinite loop $G$-space}.
\end{Def}

\subsection{Preliminaries on transfer data}

Indexing and transfer systems are used to encode which transfers an equivariant commutative monoid or an $N_\infty$-algebra encodes. Since we aim to model the latter they play an important role. Indexing systems were introduced by Blumberg and Hill in \cite{BH15} to study transfers of algebras over $N_\infty$-operads. We will use the axiomatic description given by Rubin in \cite{Rub21b}.

\begin{Def}
	A \textit{class of finite $G$-subgroup actions} is a class $\II$, equipped with a function $\II\to\Sub(G)$, such that the fibre over $H\subset G$ is a class of finite $H$-sets. Write $\II(H)$ for the fibre over $H$. Let $G$ be a finite group and let $\II\to\Sub(G)$ be a class of finite $G$-subgroup actions. Then $\II$ is said to be a \textit{$G$-indexing system} if satisfies the following axioms:
	\begin{enumerate}[(\text{I}1)]
		\item (trivial sets) For any subgroup $H\subset G$, the class $\II(H)$ contains all finite disjoint unions of trivial $H$-orbits.
		\item (isomorphism) For any subgroup $H\subset G$ and finite $H$-sets $S$ and $T$, if $S\in\II(H)$ and $S\cong T$, then $T\in\II(H)$.
		\item (restriction) For any pair of subgroups $K\subset H\subset G$ and finite $H$-set $T$, if $T\in\II(H)$, then $\res_K^HT\in\II(K)$.
		\item (conjugation) For any subgroup $H\subset G$, group element $a\in G$, and finite $H$-set $T$, if $T\in\II(H)$, then $c_aT\in \II(aHa^{-1})$. Here $c_a$ denotes the induced conjugation automorphism.
		\item (subobjects) For any subgroup $H\subset G$ and finite $H$-sets $S$ and $T$, if $T\in\II(H)$ and $S\subset T$, then $S\in\II(H)$.
		\item (coproducts) For any subgroup $H\subset G$ and finite $H$-sets $S$ and $T$, if $S\in\II(H)$ and $T\in\II(H)$, then $S\sqcup T\in\II(H)$.
		\item (self-induction) For any subgroups $K\subset H\subset G$ and finite $K$-set $T$, if $T\in\II(K)$ and $H/K\in\II(H)$, then $\ind_K^H T\in\II(H)$.
	\end{enumerate}
\end{Def}

A finite $H$-set $T$ in $\II(H)$ for some indexing system $\II$ is called \textit{admissible}. Examples of indexing systems are the \textit{trivial} indexing system $\TT$ with $\TT(H)=\Fin$ for all $H\subset G$ and the \textit{complete} indexing system $\CC$ with $\CC(H)=H\Fin$ for all $H\subset G$. Rubin \cite{Rub21a} showed that the poset of indexing systems $\Ind(G)$ admits the structure of a finite lattice with minimal element $\TT$ and maximal element $\CC$.

\begin{Rem}
	This definition is equivalent to the one given in \cite{BH15}. In particular axioms (I1)-(I7) imply an eighth axiom (\cite{Rub21b}, Remark 2.13)
	\begin{enumerate}
		\item[(I8)] (Cartesian product) For any subgroup $H\subset G$ and finite admissible $H$-sets $T,S\in\II(H)$ also $T\times S\in\II(H)$. 
	\end{enumerate}
\end{Rem}

Axioms (I5) and (I6) together imply that $\II(H)$ is the class of all finite coproducts of admissible $H$-orbits in $\II$. Thus $\II$ is determined by the orbits it contains, and since $G$ is finite there are only finitely many orbits and therefore finitely many indexing systems (\cite{Rub21b}, remark after Theorem 2.14). In particular, to study the lattice $\Ind(G)$ it suffices to study orbits. This motivates the following definition of transfer systems.

\begin{Def}[\cite{Rub21a}, Definition 3.4]
	A \textit{$G$-transfer system} is a partial order on $\Sub(G)$ which refines the subset relation and is closed under conjugation and restriction. Denote by $\Tr(G)$ the poset of all transfer systems ordered under refinement.
\end{Def}

Given a transfer system $\to$ being a refinement of the subset relation means that if $K\to H$ then $K\subset H$, being closed under conjugation means that if $K\to H$ and $g\in G$ then $gKg^{-1}\to gHg^{-1}$, and being closed under restriction means that if $K\to H$ and $L\subset H$ then $K\cap L\to L$. The requirement that $\to$ is a partial order implies reflexivity and transitivity.

\begin{Rem}
	Any binary relation $R$ that refines the subset relation can be completed to a transfer system via \textit{Rubin's algorithm} (\cite{Rub21a}, Construction A.1). Set $R_0=R$ and let $R_1$ be the closure of $R_0$ under conjugation. Let $R_2$ be the closure of $R_1$ under restriction and let $R_3$ be the closure of $R_2$ under reflexivity and transitivity. Then $R_3$ defines a transfer system. Given a set of transfers $\Aa$ let $\to_\Aa$ be the transfer system generated by $\Aa$.
\end{Rem}

\begin{Rem}
	Any indexing system $\II$ gives rise to a transfer system $\to_\II$ by setting $H\to K$ if and only if $H/K\in\II(H)$ (\cite{Rub21a}, Proposition 3.3). The converse construction also works. The idea is that finite $G$-sets split into their orbits. Given some transfer system $\to$ define an indexing system $\II_\to$ on the $H$-th level by
	\[\II_\to(H):=\left\{\text{finite $H$-sets $T$}\bigg| \  \parbox{8.7cm}{there exists $n\ge 0$ and $K_1,\dots,K_n\subset H$ such that $T\cong\coprod_{i=1}^n H/K_i$ and $K_i\to H$ for $1\le i\le n$} \right\}.\]
	(cf. \cite{Rub21a} proof of Proposition 3.6). Here empty coproducts are understood as empty sets. One can think of $\II_\to$ as the coproduct and isomorphism completion of the set of $H$-orbits $H/K$ such that $K\to H$.
	
	Indeed, the poset $\Tr(G)$ again admits the structure of a finite lattice which is now isomorphic to $\Ind(G)$ (\cite{Rub21a}, Theorem 3.7).
\end{Rem}

We will be interested in the following subposet of $\Tr(G)$.

\begin{Def}[\cite{BHK24}, Definition 3.36]\label{def:disk-like}
	A $G$-transfer system $\to$ is called \textit{disk-like}\footnote{Sometimes this is also called \textit{cosaturated} for instance in \cite{Bal25} \S 5.} if it is generated by a set of transfers $K\to G$ given a selection of subgroups $K\subset G$ using Rubin's algorithm above. 
\end{Def}

Disk-like transfer systems appear in nature through the following.

\begin{Def}
	Let $H\subset G$ be a subgroup. A finite $H$-set $T$ is \textit{admissible} for some $G$-universe $U$ if and only if $T$ embeds $H$-equivariantly into $\res_H^G U$.
\end{Def}

\begin{Rem}
	Admissible sets for universes form an indexing system $A(U)$.  The usual proof of this fact uses operad theory. Our definition of $A(U)$ is the same as $A(\KK(U))$, where $\KK(U)$ is the Steiner operad for the universe $U$ and $A(-)$ refers to admissible sets of $N_\infty$-operads. Then \cite{BH15}, \S 4 show that $A(\OO)$ defines a $G$-indexing system.
\end{Rem}

\begin{Def}\label{def:transfer-from-universe}
	Let $U$ be a $G$-universe. Define a transfer system $\to_U$ be setting $K\to_U H$ if and only if there is an $H$-embedding $H/K\hookrightarrow \res^G_H U$.
\end{Def}

This defines a function $\Uni(G)\to \Tr^{\textup{d}}(G)$ which maps a class of $G$-universes $[U]$ to $A(U)$. This function preserves the order, the maximal element, the minimal element, and joins, but it is not always order-reflecting, meet-preserving, or injective (\cite{Rub21a}, Proposition 2.10).

\begin{Lem}[\cite{Rub21a}, Lemma 4.1]\label{lem:universe->disk-like}
	Let $U$ be a $G$-universe. Then $\to_U$ as defined above is a disk-like transfer system.
\end{Lem}

\begin{Rem}
	Note that the converse of Lemma \ref{lem:universe->disk-like} does in general not hold. Not every disk-like indexing system is realised by the admissible sets of some $G$-universe $U$. For $G$ a non-cyclic finite abelian group \cite{Rub21a} Theorem 4.15 shows that the transfer system generated by $\{e\}\to G$ is not equal to some $\to_U$. Though, the converse does hold if the generating set of the disk-like indexing system consists of a set $H_1,\dots,H_n\subset G$ of $G$-cocyclic subgroups, i.e. normal subgroups so that $G/H$ is cyclic (\cite{Rub21a}, Theorem 4.11).
\end{Rem}

\begin{Prop}[\cite{Rub21a} Lemma A.7, \cite{BHK24} Lemma.3.37]\label{prop:disk-like-equiv-defs}
	A transfer system $\to$ is disk-like if and only if it fulfils the property that for subgroups $K\subset H$ it holds that $K\to H$ whenever there exists $K'\subset G$ such that $K'\to G$ and $H\cap K'=K$.
\end{Prop}

\begin{Rem}\label{rem:disk-like-sublattice}
	Note that since $\Tr(G)$ is a finite lattice with minimal element which is always disk-like, any transfer system $\to$ has a maximal sub-transfer system $\to^\textup{d}$ which is disk-like. The same holds for indexing systems. Here we denote the maximal disk-like sub-indexing system by $\II^\textup{d}$. We will denote the sub-lattice of disk-like $G$-transfer systems by $\Tr^{\textup{d}}(G)$ and the sub-lattice of disk-like $G$-indexing systems by $\Ind^\textup{d}(G)$.
\end{Rem}

\begin{Ex}
	The trivial and the complete transfer system are always disk-like. The trivial one is generated by the empty set of transfers to $G$ and the complete one by all transfers. The $C_{p^2}$-transfer system consisting of the transfer $C_1\to C_p$ and all reflexive relations is an example of a non-disk-like transfer. The orbit $C_p/C_1$ is not accessible by restriction of an $C_{p^2}$-orbit. Hence, this transfer system cannot be generated by transfers of the form $K\to C_{p^2}$.
\end{Ex}


\section{Isotropy of \texorpdfstring{$G$}{G}-spaces}\label{sec:isotropy}

\subsection{Isotropy families}\label{sec:isotropy-families}

Another incarnation of disk-like transfer appears as isotropy families of $G$-universes. Isotropy families of $G$-spaces gives another point of view for disk-like transfer data.

\begin{Def}
	A \textit{family} $\FF$ is a set of subgroups of $G$ which is closed under conjugation, that is, if $H\in \FF$ and $g\in G$ also $gHg^{-1}\in \FF$ (\cite{Die87}, \S I.6). A family is called \textit{based} if $G$ is contained in it.
\end{Def}

\begin{Ex}
	For any group $G$ we have the \textit{trivial} based family $\FF_t=\{G\}$ and the \textit{complete} based family $\FF_c=\{H\subset G\mid H$ is a subgroup of $G\}$.
\end{Ex}

The primary example of a based family is the isotropy of a $G$-space. Recall that given a $G$-space $X$, $G_x=\{g\in G\mid gx=x\}\subset G$ denotes the \textit{stabilizer subgroup} of a point $x\in X$.

\begin{Def}
	Let $X$ be a $G$-space. The \textit{isotropy family} $\Phi(X)$ of $X$ is defined to be the set of stabilizer subgroups of $G$ indexed on points in $X$, that is, 
	\[\Phi(X):=\{G_x\subset G\mid x\in X\}.\]
\end{Def}

\begin{Rem}
	Note that this indeed defines a family since $gG_xg^{-1}=G_{gx}$ and noting that $X$ is a $G$-space. If $X$ is a based $G$-space with $G$-fixed base point $x_0$ then $G_{x_0}=G$ implying $G\in \Phi(X)$. Hence isotropy families of based $G$-spaces are always based. Note that the definition of $\Phi(X)$ soley depends on the underlying $G$-set.
\end{Rem}

\begin{Prop}\label{prop:isotropy-homeomorphism}
	If $X$ and $Y$ are $G$-homeomorphic (based) $G$-spaces then $\Phi(X)=\Phi(Y)$.
\end{Prop}

\begin{proof}
	Let $f\colon X\to Y$ be a $G$-homeomorphism with inverse $f^{-1}\colon Y\to X$. Since $gf(x)=f(gx)$ it holds that $G_x\subset G_{f(x)}$ and similarly $G_{f(x)}\subset G_{f^{-1}(f(x))}=G_x$. Hence, $G_x=G_{f(x)}$ for any $x\in X$. The same proof shows that $G_y=G_{f^{-1}(y)}$ for any $y\in Y$. Hence, $\Phi(X)=\Phi(Y)$.
\end{proof}

Homotopy equivalences not necessarily preserve isotropy. Take for example the disk with flip $C_2$-action which has non-trivial isotropy. This contracts to a point with trivial isotropy.

\begin{Def}
	Let $\Fam_*(G)$ be the poset of all based families of $G$. Let $\FF,\EE\in\Fam_*(G)$. Define a meet operation $\FF\land \EE$ by intersection and a join operation $\FF\lor \EE$ by union. These are easily checked to be based families. Hence, $\Fam_*(G)$ admits the structure of a finite lattice.
\end{Def}

\begin{Ex}
	Consider $G=C_{p^2}$. Its subgroups are $C_{1},C_{p},C_{p^2}$. The Hasse diagram of the lattice $\Fam_*(C_{p^2})$ can be depicted as follows
	
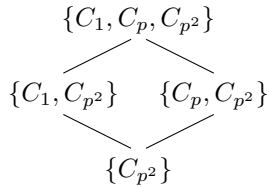
\begin{center}
		\begin{tikzpicture}
			\node at (1,1) {$\{C_p,C_{p^2}\}$};
			\node at (-1,1) {$\{C_1,C_{p^2}\}$};
			\node at (0,0) {$\{C_{p^2}\}$};
			\node at (0,2) {$\{C_1,C_p,C_{p^2}\}$};
			\draw (-0.1,0.3) -- (-1,0.75);
			\draw (0.1,0.3) -- (1,0.75);
			\draw (-0.1,1.75) -- (-1,1.3);
			\draw (0.1,1.75) -- (1,1.3);
		\end{tikzpicture}
	\end{center}
	This looks similar to the sublattice of disk-like transfer systems $\Tr^{\textup{d}}(C_{p^2})$ as depicted in \cite{Rub21a}, \S 3.2, Corollary 4.12. We now want to analyse the relationship between these two lattices.
\end{Ex}

\begin{Ex}\label{ex:families-generated-by-subsets}
	Any set $\Aa\subset\Sub(G)$ of subgroups of $G$ determines a based family $\langle \Aa\rangle_f$ by closing under conjugation and adding the base point if needed. If for example $G=\Sigma_3$ let $\Aa=\{\langle (12)\rangle \}$ be the singleton set of the subgroup generated by the permutation $(12)$. Its conjugate subgroups are generated by the other $2$-cycles $(23)$ and $(13)$. Hence,
	\[\langle \Aa \rangle_f = \{\Sigma_3,\langle (12)\rangle, \langle (23)\rangle, \langle (13)\rangle\}.\]
\end{Ex}

\begin{Rem}
	Any set $\Aa\subset\Sub(G)$ also determines a disk-like transfer system $\to_\Aa$ generated by transfers $K\to_\Aa G$ if and only if $K\in \Aa$ using Rubin's algorithm.
\end{Rem}

\begin{Prop}\label{prop:families-and-disk-like}
	There is an order, meet, and join preserving injective map
	\[\Phi\colon \Tr^{\textup{d}}(G)\to \Fam_*(G).\]
\end{Prop}

\begin{proof}
	Let $\to$ be a disk-like transfer system generated by transfers $H_i\to G$ for $i\in I$. Then $\FF_\to:=\{H_i\mid H_i\to G, i\in I\}$ is a based family since $\to$ is closed under conjugation and is reflexive. The poset structure on transfer systems and of based families is by definition compatible under this assignment so that the map is order preserving. Since $H_i \in \FF_\to$ if and only if $H_i\to G$ and since disk-like transfer systems are determined on transfers $H_i\to G$ the map is injective. This also implies that join and meet are preserved.
\end{proof}

Given a subset $\Aa\subset \Sub(G)$ we thus obtain another family $\langle \Aa\rangle_t$ generated by $\Aa$ given by $\Phi(\to_\Aa)$. Note that we could have also defined a function $\Tr(G)\to \textsf{Fam}_*(G)$, but this is evidently not an injective function.

\begin{Rem}
	Using Example \ref{ex:families-generated-by-subsets} and Example A.3 in \cite{Rub21a} we can directly see that $\Phi$ is in general not surjective. The family $\langle\Aa\rangle_f$ is not in the image of $\Phi$ since the transfer system generated by this family also contains the transfer $1\to \Sigma_3$ so that $\langle \Aa\rangle_t \ne \langle \Aa\rangle_f$.
\end{Rem}

We will mainly be interested in families in the image of $\Phi$ so we introduce the following terminology

\begin{Def}
	A based family in the image of the map $\Phi\colon \Tr^{\textup{d}}(G)\to \textsf{Fam}_*(G)$ defined in Proposition \ref{prop:families-and-disk-like} is called a \textit{transfer-like} family.
\end{Def}

We can now study isotropy families of constructions of based $G$-spaces.

\begin{Prop}\label{prop:properties-isotropy-families}
	Let $X,Y$ be $G$-spaces with isotropy family in a fixed family $\FF$, i.e $\Phi(X),\Phi(Y)\subset \FF$. Then,
	\begin{enumerate}[(1)]
		\item $\Phi(X\cup Y)=\Phi(X)\cup \Phi(Y)\subset \FF$ and in particular $\Phi(X\sqcup Y)=\Phi(X)\cup \Phi(Y)\subset \FF$,
		\item $\Phi(X\cap Y)=\Phi(X)\cap \Phi(Y)\subset \FF$,
		\item if $A\subset X$ is a $G$-subspace then $\Phi(A)\subset \Phi(X)\subset \FF$,
		\item $\Phi(X_+)=\Phi(X)\cup\{G\}$,
		\item let $A\subset Y$ be a $G$-subspace and let $f\colon A\to Y$ be a $G$-map, then $\Phi(X\cup_f Y)\subset \FF$ and in particular $\Phi(Y/A)\subset \FF$ and if $X,Y$ are based it holds that $\Phi(X\lor Y)\subset \FF$, and
		\item if $\FF$ is transfer-like then $\Phi(X\times Y)\subset \FF$ and if $X,Y$ are based then $\Phi(X\land Y)\subset \FF$.
	\end{enumerate}
\end{Prop}

\begin{proof}
	Part (1), (2), (3) and (4) are clear. To see (5) note that as sets $X\cup_f Y\cong X\sqcup Y\backslash A$. Here $A$ is a $G$-subspace. Hence, $Y\backslash A$ is again a $G$-space so that we may compute
	\[\Phi(X\cup_f Y)=\Phi(X)\cup \Phi(Y\backslash A) \subset \FF.\]
	using (1) and (3). If $X$ is based use $X\backslash A_-$ where $A_-$ is $A$ with its base point removed so that $X\backslash A_-$ is a based $G$-subspace of $X$ and repeat the argument in an analogous way. To see (6) note that
	\[\Phi(X\times Y)=\{G_{(x,y)}\mid (x,y)\in X\times Y\}= \{G_{(x,y)}\mid x\in X\land y\in Y\}\]
	and $g\in G_{(x,y)}$ if and only if $g(x,y)=(gx,gy)=(x,y)$ if and only if $gx=x$ and $gy=y$. Hence, $G_{(x,y)}=G_x\cap G_y$. Since $\FF$ is assumed to be transfer-like we know that there is a transfer system $\to$ such that $\FF_\to=\FF$. If $G_x\to G$ we know by closure under restriction that $G_x\cap G_y\to G_y$ and since also $G_y\to G$ and since $\to$ is transitive we know that $G_x\cap G_y\to G$. Hence, $G_x\cap G_y\in \FF$ so that $\Phi(X\times Y)=\{G_x\cap G_y\mid x\in X\land y\in Y\}\subset \FF$. The last part follows from the fact that
	\[\Phi(X\land Y)=\Phi((X\times Y)/(X\lor Y))\subset \FF\]
	using (5) and the previous argument.
\end{proof}

Using Proposition \ref{prop:families-and-disk-like} we can use transfer-like families, disk-like transfer systems and disk-like indexing systems interchangeably. One of the main reasons why we are interested in families is understanding the isotropy of $G$-representations.

\begin{Rem}
	To any $G$-representation $V$ we can assign a transfer system $\FF_V$ with $H\in \FF_V$ if and only if $G/H\hookrightarrow V$ embeds $G$-equivariantly. If $V^G\ne \emptyset$ then $\FF_V$ is based. In particular, we can define $\FF_U$ for $U$ a $G$-universe as for $V$ and note that this is the same as in Definition \ref{def:transfer-from-universe}. By Proposition \ref{lem:universe->disk-like} this a disk-like transfer system and hence a transfer like family. But note that $\FF_V$ does not need to be transfer-like. Note that
	\[\bigcup_{V\subset U \text{ $G$-rep}}\FF_V=\FF_U\]
	and $\FF_U$ is always based since $U$ contains the trivial representation.
\end{Rem}

\begin{Def}
	Let $\FF$ be a family. A universe $U$ is said to be \textit{compatible} with $\FF$ if $\FF_U=\FF$. Analogously, a $G$-representation $V$ is said to be \textit{compatible} with $\FF$ if $\FF_V\subset\FF$.
\end{Def}

\begin{Lem}
	Let $V$ be a $G$-representation. Then $\FF_V=\Phi(V)=\Phi(S^V)$.
\end{Lem}

\begin{proof}
	If $G/H\hookrightarrow V$ then any element $v$ in the image of the embedding is $H$-fixed so that $G_v\in \Phi(V)$. Conversely, given any $v\in V$ one can construct an embedding $G/G_v\to V$ by setting $gG_v\mapsto gv$. Hence, $\FF_V=\Phi(V)$. One point compactification adds a $G$-fixed point at infinity so that $G/G\hookrightarrow S^V$ which maps the single point in $G/G$ to the point at infinity in $S^V$. Since $0\in V$ is $G$-fixed $G\in \Phi(V)$ so that $\FF_V=\Phi(S^V)$.
\end{proof}

\begin{Cor}
	If $V\subset W$ is a finite dimensional $G$-subrepresentation then $\FF_V\subset\FF_W$.
\end{Cor}

\begin{Ex}
	For $\rho$ the regular representation of $G$ we know that $\FF_\rho=\FF_c$ since any orbit $G/H$ embeds into $\rho$. Similarly, for $\tau$ the trivial representation only trivial $G$-sets embed $G$-equivariantly so that $\FF_\tau=\FF_t$.
\end{Ex}

\subsection{Isotropy of finite \texorpdfstring{$G$}{G}-sets}\label{sec:finite-G-sets}

We now apply these isotropy consideration to finite $G$-sets.

\begin{Def}
	Let $\FF$ be a family. Let $G\Fin^\FF\subset G\Fin$ be the full subcategory consisting of those finite $G$-sets $T$ such that $\Phi(T)\subset\FF$. Analogously, define for a based family $\FF$ the full subcategory $G\Fin^\FF_*\subset G\Fin_*$. Let $\OO_G^\FF\subset\OO_G$ be the full subcategory of the orbit category consisting of those orbits $G/H$ so that $H\in\FF$.
\end{Def}

By Proposition \autoref{prop:families-and-disk-like} a based transfer-like family $\FF$ corresponds to a disk-like transfer system $\to_\FF$ and consequently an indexing system $\II_\FF$. \cite{Rub21a}, Proposition 3.6 showed that the full subcategory $\II_F(G)$ of $G\Fin$ is equivalent to the coproduct and isomorphism completion of $\OO_G^\FF$. 

Any finite $G$-set can be decomposed into a finite disjoint union of $G$-orbits. Since $\Phi(X\sqcup Y)=\Phi(X)\cup\Phi(Y)$ by Proposition \ref{prop:properties-isotropy-families} to compute the isotropy of finite $G$-sets it suffices to compute the isotropy of orbits. Let $H\in\FF$. To compute $\Phi(G/H)=\{G_{gH}|gH\in G/H\}$ note that $G_{gH}$ consists of those $h\in G$ such that $hgH=gH$, that is, $g^{-1}hgH=H$, that is, $g^{-1}hg\in H$, that is, $h\in gHg^{-1}$. Hence, $\Phi(G/H)=\{gHg^{-1}\mid g\in G\}\subset \FF$. In particular, $\OO_G^\FF\subset G\Fin^\FF$. Note that $G\Fin^\FF$ is isomorphism and coproduct complete so that $(\OO_G^\FF)^{\sqcup,\cong}\subset G\Fin^\FF$. The converse is shown analogously. Hence we have the following proposition.

\begin{Prop}
	Let $\FF$ be a based transfer-like family. Then
	\[\II_\FF(G) \simeq G\Fin^\FF \simeq (\OO_G^\FF)^{\sqcup,\cong}.\]
\end{Prop}

Using this we can collect properties of $G\Fin^\FF$ induced from indexing systems.

\begin{Cor}\label{cor:properties-finite-F-sets}
	Let $\FF$ be a based transfer-like family. The category $G\Fin^\FF$ contains all trivial $G$-sets, isomorphisms, is closed under subobjects, finite cartesian products and hence finite limits, and it is closed under finite coproducts.
\end{Cor}

Note that we require $\FF$ to be based so that trivial $G$-sets are contained in $G\Fin^\FF$. Closure under finite limits in $G\Fin^\FF$ follows from the assumption that $\FF$ is transfer-like. For more general families we cannot expect this.

\begin{Rem}
	We can repeat the story above for finite based $G$-sets. Let $(\OO_G^\FF)_+$ be the set of based orbits $G/H_+$ so that $H\in\FF$ and denote by $(\OO_G^\FF)_+^{\lor,\cong}$ the coproduct and isomorphism completion with respect to the wedge product. Denote by $\II_\FF(G)_+$ the category of finite based $G$-sets $T$ such that $T\cong S_+$ with $S\in\II_\FF(G)$. Then,
	\[\II_\FF(G)_+\simeq G\Fin^\FF_*\simeq (\OO_G^\FF)^\lor_+\]
	and the categories above contain all trivial based $G$-sets, isomorphisms, are closed under subobjects, finite smash products and hence finite limits, and they are closed under finite wedges.
\end{Rem}

\begin{Rem}
	Similarly to $G\Fin^\FF$ and $G\Fin^\FF_*$ we may define
	\[G\Top^\FF \text{ and }G\Top_*^\FF\]
	as those full subcategories of (based) $G$-spaces such that $\Phi(X)\subset \FF$ and $G\Set^\FF$ and $G\Set^\FF_*$ as categories of (based) sets with isotropy in $\FF$.
\end{Rem}

We now want to understand what happens when we restrict the categories $G\Fin^\FF$ to subgroup actions. For this we need the following description of restrictions of orbits.

\begin{Rem}\label{rem:double-coset-formula}
	Let $H,K\subset G$ be subgroups. Then we can decompose the restriction of $G/K$ to $H$ via the \textit{double coset formula} (cf. proof of \cite{Rub21a}, Proposition 3.6)
	\[\res_H^G G/K \cong \coprod_{a\in H\backslash G/K} H/(H\cap aKa^{-1})\]
	We thus always have an inclusion $i\colon H/(H\cap K)\hookrightarrow \res_H^G G/K$. If $|H\backslash G/K|=1$, then $\res_H^G G/K\cong H/(H\cap K)$.
\end{Rem}

This motivates why we are interested in disk-like transfer systems.

\begin{Def}
	Let $\FF$ be a based transfer-like family for a group $G$ and let $H\subset G$ be a subgroup. Define $\FF_H$ to be the family of $H$-subgroups
	\[\FF_H:=\{K\subset H \mid \text{ there is }K'\in \FF\text{ such that }K'\cap H=K\}.\]
\end{Def}

\begin{Prop}\label{prop:restriction-finite-G-sets}
	It holds that $\II_\FF(H) = H\Fin^{\FF_H}$
\end{Prop}

\begin{proof}
	Let $T\in H\Fin^{\FF_H}$ with orbit decomposition $T\cong H/K_1\sqcup\dots\sqcup H/K_n$. Then $K_i\in \FF_H$ for all $1\le i\le n$ which by definition of $\FF_H$ implies that there exists $K_i'\in \FF$ such that $K_i'\cap H=K_i$. By the double coset formula and the fact that $K_i'\cap H=K_i$ it holds that $H/K_i\hookrightarrow \res_H^G G/K_i'$. Let $T':= \bigsqcup_{i=1}^n G/K_i'$. By assumption $T'\in \II_\FF(G)$ and thus $\res_H^G T'\in \II_\FF(H)$ using closure under restriction (I3). Closure under subobjects (I5) implies that $T\in \II_\FF(H)$. Hence, $H\Fin^{\FF_H} \hookrightarrow \II_\FF(H)$. 
	
	The converse follows from the disk-like property. Decompose $T\in \II_\FF(H)$ as above into $H/K_1\sqcup\dots\sqcup H/K_n$ and use the disk-like property of $\II$ to obtain finite $G$-sets $G/K_i'$ in $\II_\FF(G)$ with $K_i'\cap H=K_i$. Hence, $K_i\in \FF_H$ for $1\le i\le n$ so that $T\in H\Fin^{\FF_H}$.
\end{proof}

The Segal machines in this paper use a skeleton of $G\Fin^\FF$ or rather $G\Fin^\FF_*$. Write
\[\textbf{n}:=\{1,\dots,n\} \text{ and }\textbf{n}_+:=\{0,1\dots,n\}\]
for the finite set of $n$ elements and the finite based set with $0$ as a base point. Any finite (based) set is isomorphic to some $\textbf{n}$ (resp. $\textbf{n}_+$). For $\alpha\colon G\to\Sigma_n$ a group homomorphism define finite (based) $G$-sets $\textbf{n}^\alpha$ and $\textbf{n}_+^\alpha$ with action
\[g.i=\alpha(g)(i)\]
where $\alpha(g)(0)=0$ in the based case. We assume that $\textbf{0}_+=\{0\}$. Any finite $G$-set $T$ is isomorphic to a finite $G$-set of the form $\textbf{n}^\alpha$ for some homomorphism $\alpha\colon G\to\Sigma_n$ (\cite{MMO25}, Convention 1.4). Similarly, any finite based $G$-set is of the form $\textbf{n}_+^\alpha$ for some $\alpha\colon G\to\Sigma_n$.

Let $H\in\FF$ for some based family $\FF$ so that the orbit $G/H$ is in $G\Fin^\FF$. Choose coset representatives $r_1,\dots,r_n$ for $n=[G:H]$. The $G$-action on $G/H$ determines a homomorphism $\rho\colon G\to\Sigma_n$ via
\[g.r_iH=r_{\rho(g)(i)}H.\]
Hence, there is an isomorphism $G/H\cong \textbf{n}^\rho$ via $r_iH\mapsto i$ and $G/H_+\cong \textbf{n}_+^\rho$. The choice of ordering of coset representatives amounts to applying a permutation.

The following proposition is a collection of standard facts on finite $G$-sets. The isotropy statements follow directly from the axioms of indexing systems and Corollary \ref{cor:properties-finite-F-sets}.

\begin{Prop}\label{prop:properties-isotropy}
	Let $K\subset H\subset G$ be subgroups and let $\alpha\colon H\to\Sigma_n$ and $\beta\colon H\to\Sigma_m$ be group homomorphisms. Let $\FF$ be a based transfer-like family and assume that $\textbf{n}^\alpha,\textbf{m}^\beta \in H\Fin^\FF$.
	\begin{enumerate}
		\item $\res_K^H\textup{\textbf{n}}^\alpha\cong\textup{\textbf{n}}^{\res_K^H(\alpha)} \in K\Fin^\FF$ and $\res_K^H\textup{\textbf{n}}^\alpha_+\cong\textup{\textbf{n}}_+^{\res_K^H(\alpha)}\in K\Fin^\FF_*$,
		\item $\ind_H^G\textup{\textbf{n}}^\alpha \in G\Fin^\FF$ and $\ind_H^G\textup{\textbf{n}}^\alpha_+ \in G\Fin^\FF_*$,
		\item $\textup{\textbf{n}}^\alpha \sqcup \textup{\textbf{m}}^\beta \cong (\textup{\textbf{n+m}})^{\alpha\sqcup \beta} \in H\Fin^\FF$ and $\textup{\textbf{n}}_+^\alpha \lor \textup{\textbf{m}}_+^\beta \cong (\textup{\textbf{n+m}})_+^{\alpha\lor \beta}\in H\Fin^\FF_*$,
		\item $\textup{\textbf{n}}^\alpha \times \textbf{m}^\beta \cong (\textup{\textbf{nm}})^{\alpha\times \beta}\in H\Fin^\FF$ and $\textup{\textbf{n}}_+^\alpha \land \textup{\textbf{m}}_+^\beta \cong (\textup{\textbf{nm}}_+)^{\alpha\land \beta}\in H\Fin^\FF_*$, and 
		\item if $\textup{\textbf{n}}^\alpha\subset \textup{\textbf{m}}^\beta$, then $n<m$ and $\im(\alpha)\subset\im(\beta)$ and vice versa and if $\textup{\textbf{n}}^\alpha_+\subset \textup{\textbf{m}}^\beta_+$, then $n<m$ and $\im(\alpha)\subset\im(\beta)$ and vice versa. If only $\textup{\textbf{m}}^\beta\in H\Fin^\FF$ is assumed it also follows that $\textup{\textbf{n}}^\alpha\in H\Fin^\FF$ and similarly in the based case.
	\end{enumerate}
\end{Prop}

\subsection{\texorpdfstring{$G$}{G}-CW complexes with fixed isotropy}\label{sec:F-G-CW-cxs}

We now want to define a category of $G$-CW complexes with isotropy in a fixed based family $\FF$. We will use based $G$-CW complexes in the sense of \cite{May96}, \S I.3. We will follow the exposition of unbased $G$-CW complexes in \cite{Die87}, \S II.1 and adapt for based spaces and isotropy.

\begin{Def}
	A collection of subgroups $(H_i)_{i\in I}$ is called \textit{compatible with $\FF$} if $I$ is finite and if $H_i\in\FF$ for all $i\in I$.
	
	Let $n\ge 0$ and let $A$ be a based $G$-space. Given a collection $(H_j)_{j\in J}$ of subgroups of $G$ compatible with $\FF$ and a collection of based $G$-maps
	\[\varphi_j\colon (G/H_j)_+\land S^{n-1}\to A\]
	for $j\in J$ we consider pushouts of based $G$-spaces
	\begin{center}
		\begin{tikzcd}
			\bigvee_{j\in J}(G/H_j)_+\land S^{n-1} \arrow[d, "\varphi"'] \arrow[r, hook] & \bigvee_{j\in J}(G/H_j)_+ \land D^{n} \arrow[d, "\phi"] \\
			A \arrow[r, "i"']                                                & X                    
		\end{tikzcd}
	\end{center}
	Here,
	\[\varphi|_{(G/H_j)_+\land S^{n-1}}=\varphi_j \text{ and } \phi|_{(G/H_j)_+\land S^{n-1}}=\phi_j.\]
	Note that $G\Top_*$ is cocomplete and therefore such pushouts always exist. We say that $X$ is obtained from $A$ by \textit{attaching the collection of $\FF$-compatible $n$-cells $((G/H_j)_+\land D^n)_{j\in J}$ of type $(G/H_j)_{j\in J}$}. Here $D^n$ and $S^n$ are assumed to be based on the south pole.
\end{Def}

\begin{Rem}
	Let $A$ be based $G$-space and assume that $X$ is obtained from $A$ by attaching a family of $\FF$-compatible $n$-cells $((G/H_j)_+\land D^n)_{j\in J}$ of type $(G/H_j)_{j\in J}$. If $A$ is Hausdorff then $X$ is also Hausdorff (\cite{Die87}, II.1.4). We will consider the map $i\colon A\to X$ to be an inclusion. Then there is a homeomorphism
	\[X\backslash A \cong \bigvee_{j\in J}(G/H_j)_+\land \mathring{D}^n.\]
\end{Rem}

\begin{Prop}\label{prop:isotropy-G-CW-attaching}
	Let $A$ be a space with isotropy $\Phi(A)\subset \FF$ and suppose that $X$ is obtained from $A$ by attaching $n$-cells compatible with $\FF$. Then $\Phi(X)\subset \FF$.
\end{Prop}

\begin{proof}
	Since $X=X\backslash A\cup A$ we can compute $\Phi(X)=\Phi(A)\cup \Phi(X\backslash A)$. By assumption $\Phi(A)\subset \FF$ and using the remark above and Proposition \ref{prop:properties-isotropy-families} we compute that
	\[\Phi(X\backslash A) = \Phi(\bigvee_{j\in J}(G/H_j)_+\land \mathring{D}^n) = \bigcup_{j\in J}\Phi((G/H_j)_+)\subset \FF.\qedhere\]
\end{proof}

\begin{Def}[\cite{Die87}, \S II.1]
	Suppose $(X,A)$ is a pair of based $G$-spaces with $A$ being a Hausdorff space. A \textit{based equivariant CW-decomposition} of $(X,A)$ consists of a filtration $(X_n)_{n\in\zz}$ of $X$ such that the following holds
	\begin{enumerate}[(i)]
		\item $A\subset X_0$; $A=X_n$ for $n<0$; $X=\bigcup_{n\in\zz} X_n$.
		\item For each $n\ge 0$, the space $X_n$ is obtained from $X_{n-1}$ by attaching equivariant $n$-cells.
		\item $X$ carries the colimit topology with respect to $(X_n)$, i.e. $B\subset X$ is closed if and only if, for all $n$, $B\cap X_n$ is closed in $X_n$.
	\end{enumerate}
	
	If $(X_n)$ is an equivariant CW decomposition of $(X,A)$ such that both $A$ and $X$ have isotropy in a fixed family $\FF$ then $(X,A)$ is called a \textit{relative based $\FF$-$G$-CW complex}. If $A=*$ then $(X,*)=X$ is called a \textit{based $\FF$-$G$-CW complex}.
\end{Def}

Hence by Proposition \ref{prop:isotropy-G-CW-attaching}, $\FF$-$G$-CW complexes are $G$-spaces with isotropy in $\FF$. One can analogously in parallel to $G$-CW complexes according terminology for $\FF$-$G$-CW complexes. For example subcomplexes, $n$-cells or $n$-skeleta. For $\II=\CC$ the complete indexing system, note that $\CC$-$G$-CW complexes are just $G$-CW complexes and for $\II=\TT$ the trivial indexing system, a $\TT$-$G$-CW complex is a classical CW complex with trivial $G$-action.

\begin{Rem}
	Note that $\FF$-$G$-CW complexes are in particular $G$-CW complexes. Hence many properties of $\FF$-$G$-CW complexes follow from the more general case. For example (see \cite{Die87}, Proposition 1.6) $\FF$-$G$-CW complexes $X$ are a Hausdorff spaces such that $X_n$ is closed in $X$ and $C$ is closed in $X$ if and only if, for each cell $E$ of $(X,A)$, the subset $C\cap E$ is closed in $E$.
\end{Rem}

\begin{Def}
	Denote by $G\CW^\FF$ the full subcategory of $G\Top_*^\FF$ or more generally $G\Top_*$ of finite based $\FF$-$G$-CW complexes and denote by $G\underline{\CW}^\FF$ the respective full subcategory of $G\Topu_*$ with mapping $G$-spaces consisting of all maps with conjugation $G$-action.
\end{Def}

\begin{Rem}
	 If $\II$ is a (disk-like) indexing system, then this determines a family $\FF_\II$. We will call $\FF_\II$-$G$-CW complexes \textit{$\II$-$G$-CW complexes} and decorate the corresponding categories with an $\II$ instead of $\FF_\II$, i.e. $G\CW_*^\II$.
\end{Rem}

\begin{Prop}
	If $(X,A)$ is a $G$-subcomplex of a relative based $\FF$-$G$-CW complex $(Y,B)$, then $(X,A)$ is a based relative $\FF$-$G$-CW complex with filtration $(X_n)_{n\in \zz}$ given by $X_n=X\cap Y_n$. Moreover, $X$ is closed in $Y$.
\end{Prop}

\begin{proof}
	This follows immediately from \cite{Die87}, Proposition 1.7 since $\Phi(X_n)=\Phi(X)\cap \Phi(Y_n)\subset \FF$ using Proposition \ref{prop:properties-isotropy-families}.
\end{proof}

Following \cite{Die87}, Example II.1.8 for $(X,A)$ a relative based $\FF$-$G$-CW complex $(X,X_n)$ and $(X_n,A)$ are also relative based $\FF$-$G$-CW complexes and $(X_n,A)$ is a subcomplex of $(X,A)$. In this sense, skeleta are $G$-subcomplexes.

\begin{Prop}
	Let $X$ be a relative based $\FF$-$G$-CW complex and $A$ a $G$-subcomplex. Then $X/A$ is a relative based $\FF$-$G$-CW complex with $n$-skeleton $X_n/A_n$.
\end{Prop}

\begin{proof}
	This follows immediately from \cite{Die87}, Proposition 1.11 using Proposition \ref{prop:properties-isotropy-families} and the fact that $A_n\subset X$ is a based $G$-subspace to compute that $\Phi((X/A)_n)=\Phi(X_n/A_n)\subset \FF$.
\end{proof}

\begin{Prop}[\cite{Die87}, Exercise II.1.17 (3)]
	Let $(X,A)$ be a relative based $\FF$-$G$-CW complex. Then $A\to X$ is a $G$-cofibration. In particular, any based $\FF$-$G$-CW complex is well pointed.
\end{Prop}

The following proposition collects some properties of $\FF$-$G$-CW complexes. The proof is the same as in the $G$-CW case where one uses Proposition \ref{prop:isotropy-G-CW-attaching} and Proposition \ref{prop:properties-isotropy-families} to deduce the corresponding statements about isotropy. All of the results without the isotropy and based assumption can be found in \cite{Die87}, \S II.1.

\begin{Prop}[Constructions of $\FF$-$G$-CW complexes]\label{prop:constr-F-G-CW}
	Let $\FF$ be a transfer-like family and let $X,Y$ be $\FF$-$G$-CW complexes. Let $H\subset G$ be a subgroup.
	\begin{enumerate}
		\item Let $Z$ be a based $\FF_H$-$H$-CW complex. If $H\in\FF$ then $X:=\ind_H^G X=G_+\land_H Y$ is a based $\FF$-$G$-CW complex with $n$-skeleton $X_n=G_+\land_H Y_n$. If $E$ is an open resp. closed $n$-cell of type $H/K$ of $Y$ then $G_+\land_H E$ is an open resp. closed $n$-cell of type $G/K$ of $X$.
		\item The $W_GH$-space $X^H$ is a based $W_GH$-CW complex with $n$-skeleton $X_n^H$.
		\item The $H$-space $\res_H^G X$ is a based $\FF_H$-$H$-CW complex with the same skeleton.
		\item The smash product $X\land Y$ is again an $\FF$-$G$-CW complex.
		\item The wedge sum $X\lor Y$ is again a finite based $\FF$-$G$-CW complex.
	\end{enumerate}
\end{Prop}

We want to close this paragraph and this chapter by providing the motivating example for $\FF$-$G$-CW complexes, namely representation spheres.

\begin{Thm}[\cite{Ill78}]
	Let $M$ be a smooth $G$-manifold. Then there exists an equivariant simplicial complex $K$ and a smooth triangulation $t\colon |K|\to M$. If $M$ is compact then $K$ can be chosen to be finite. Here $|K|$ denotes the underlying $G$-space of the simplicial complex $K$.
\end{Thm}

\begin{Rem}
	In the case that an equivariant triangulation of $M$ exists, $M$ admits the structure of a $G$-CW complex (\cite{Mat71}, Proposition 4.4). Since $t$ is a $G$-map the isotropy of $K$ and $M$ agree (see Proposition \ref{prop:isotropy-homeomorphism}) and therefore the $G$-CW complex structure consists of cells of type $\Phi(M)$.
\end{Rem}

Using this we can provide the main class of examples.

\begin{Cor}
	Let $V$ be a $G$-representation compatible with a family $\FF$. Then $S^V$ admits the structure of a finite based $\FF$-$G$-CW complex.
\end{Cor}

\begin{proof}
	Note that $\Phi(S^V)$ is a based family with $\Phi(V)=\Phi(S^V)$. Since $V$ is finite dimensional, $S^V$ is compact so that it admits a finite triangulation. By the above remark we can therefore construct a $G$-homotopy equivalent $G$-CW complex with $G$-cells in $\Phi(S^V)$.
\end{proof}

\begin{Ex}
	Consider the two non-isomorphic $2$-dimensional $C_5$-representations $V_5(1)$ and $V_5(2)$. They have both complete isotropy families. The corresponding $G$-CW structures are both obtained by gluing a single free $(C_5/C_1)_+\land D^1$-cell along two $C_5$-fixed $0$-cells and then gluing a single $2$-cell of type $(C_5/C_1)_+\land D^2$ (cf. \cite{Blu17}, Example 1.2.4). The difference between the two representations spheres comes from different gluing maps. To visualise the gluing happening consider the following schematic views of the equator
	
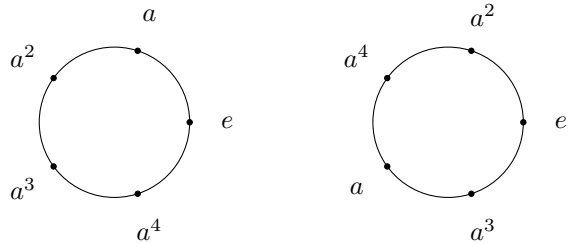
\begin{center}
		\begin{tikzpicture}
 			\draw (0,0) circle (1cm);
 			\filldraw (0:1cm) circle (1pt);
 			\filldraw (72:1cm) circle (1pt);
 			\filldraw (144:1cm) circle (1pt);
 			\filldraw (216:1cm) circle (1pt);
 			\filldraw (288:1cm) circle (1pt);
 			\node at (0:1.5cm) {$e$};
 			\node at (72:1.5cm) {$a$};
 			\node at (144:1.5cm) {$a^2$};
 			\node at (216:1.5cm) {$a^3$};
 			\node at (288:1.5cm) {$a^4$};
 		\end{tikzpicture}
 		\hspace{1cm}
		\begin{tikzpicture}
 			\draw (0,0) circle (1cm);
 			\filldraw (0:1cm) circle (1pt);
 			\filldraw (72:1cm) circle (1pt);
 			\filldraw (144:1cm) circle (1pt);
 			\filldraw (216:1cm) circle (1pt);
 			\filldraw (288:1cm) circle (1pt);
 			\node at (0:1.5cm) {$e$};
 			\node at (72:1.5cm) {$a^2$};
 			\node at (144:1.5cm) {$a^4$};
 			\node at (216:1.5cm) {$a$};
 			\node at (288:1.5cm) {$a^3$};
 		\end{tikzpicture}
	\end{center}
	Think of either identifying $(gC_5,t)$ with $(agC_5,t)$ or $(a^2gC_5,t)$ where $a$ is the generator of $C_5$.
\end{Ex}


\section{Equivariant incomplete \texorpdfstring{$\Gamma$}{Gamma}-spaces}\label{sec:Gamma-spaces}

We will first introduce several equivariant generalisations of $\Gamma$-spaces which were originally and non-equivariantly introduced by Segal \cite{Seg74} and later generalised by Shimakawa \cite{Shi89} in Section \ref{sec:equivariant-Gamma}. As it turns out all equivariant generalisations of $\Gamma$-spaces discussed here are equivalent as we will show in Section \ref{sec:comparison}. The key distinguishing properties are different notions of \textit{specialness} on equivariant $\Gamma$-spaces which we will discuss in \ref{sec:specialness}. In Section \ref{sec:transfer-maps-Gamma} we will discuss transfer maps on $\Gamma$-spaces.

\subsection{Equivariant generalisations of \texorpdfstring{$\Gamma$}{Gamma}-spaces}\label{sec:equivariant-Gamma}

\subsubsection{\texorpdfstring{$\Gamma$}{Gamma}-\texorpdfstring{$G$}{G}-spaces}\label{sec:Gamma-G-spaces}

The first equivariant generalisation of $\Gamma$-spaces is a `naive' generalisation of Segal's original definition \cite{Seg74}. A more detailed exposition is given in \cite{MMO25}, \S 2.1.

Let $\Gamma$ be the category of finite based sets $\textbf{n}_+=\{0,\dots,n\}$ based at $0$ and based maps between them\footnote{This definition of $\Gamma$ is equivalent to Segal's $\Gamma^{\text{op}}$ in \cite{Seg74} and May's $\F$ in various publications, for example \cite{May78} or \cite{MMO25}. Our convention is also used in \cite{Shi91}}. Consider $\Gamma$ as a $G\Top_*$-enriched category with trivial $G$-action and discrete topology everywhere and with base points of mapping spaces being the unique maps factoring through $\textbf{0}_+$. Let $\Sigma\subset\Gamma$ be the wide subgroupoid with morphisms being all permutations $\sigma\colon \textbf{n}_+\to\textbf{n}_+$ and $\Sigma(\textbf{m}_+,\textbf{n}_+)=\emptyset$ for $m\ne n$.

\begin{Def}[\cite{MMO25}, Definition 2.10]
	A $\Gamma$-$G$-space is a $G\Top_*$-functor $X\colon \Gamma\to G\Topu_*$ such that the map $\phi_*\colon X_m\to X_n$ is a $(G\times\Sigma_\phi)$-cofibration for every injection $\phi\colon \textbf{m}_+\to\textbf{n}_+$ in $\Gamma$. Here $X_n$ denotes the $(G\times\Sigma_n)$-space $X(\textbf{n}_+)$. A map of $\Gamma$-$G$-spaces is a $G\Top_*$-natural transformation. Denote the category of $\Gamma$-$G$-spaces by
	\[\Gamma[G\Top_*].\]
\end{Def}

\begin{Rem}[\cite{MMO25}, remark after Remark 2.11]
	Note that by Lemma \ref{lem:MMO-lem.1.17} $X_0$ is a point. The cofibrancy condition is in its entirety not used for the rest of this paper. It is only needed for homotopical control of the induced multiplication maps. But it is used for the following fact: A $\Gamma$-$G$-space naturally lands in the category $G\Top_{\textup{wp}}$ since the injection $\textbf{0}_+\to \textbf{n}_+$ induces a $G$-cofibration $*=X_0\to X_n$ making $X_n$ well pointed. Hence we may equivalently denote the category of $\Gamma$-$G$-spaces by
	\[\Gamma[G\Topu_{\textup{wp}}].\]
	Beware that $G\underline{\smash{\Top}}_{\textup{wp}}$ denotes the category of well pointed $G$-spaces enriched in $G\Top_*$ and not $G\Top_{\textup{wp}}$ since $G\underline{\smash{\Top}}_{\textup{wp}}$ is not cartesian closed (Remark \ref{not:well-pointed-spaces}).
\end{Rem}

Note that assuming the cofibrancy condition does not result in any loss of generality as shown in \cite{MMO25} via a process called \textit{bearding functors}. 

\begin{Prop}[\cite{MMO25}, Proposition 2.12]
	There is a functor 
	\[W\colon \Fun(\Gamma,G\Top_*)\to \Fun(\Gamma,G\Top_*)\]
	with a natural transformation $\pi\colon W\Rightarrow\Id$ such that, for any functor $X\colon \Gamma\to G\Top_*$, $WX$ is in $\Gamma[G\Top_{\textup{wp}}]$ and $\pi\colon WX\to X$ is a levelwise $G$-homotopy equivalence.
\end{Prop}

\begin{Def}[\cite{Seg74}, \S1, \cite{MMO25}, Definition 2.7]\label{def:delta-and-phi}
	Let $n\ge 0$. Define $\hat\varphi\colon \textbf{n}_+\to \textbf{1}_+$ to be the based maps given by $\hat\varphi(i)=1$ for all $1\le i\le n$ and $\hat\varphi(0)=0$. For each $1\le i\le n$ define based maps $\hat\delta_i\colon \textbf{n}_+\to \textbf{1}_+$ which map $i$ to $1$ and $j\ne i$ to $0$.
	
	Let $X$ be a $\Gamma$-$G$-space. The induced maps $\varphi:=X(\hat\varphi)\colon X_n\to X_1$ are called \textit{multiplication maps} and the maps $\delta\colon X_n\to X_1^n$ which on their $i$-th component are given by the induced maps $X(\hat\delta_i)\colon X_n\to X_1$ are called \textit{Segal maps}.
\end{Def}

The Segal maps naturally appear as the units of an adjunction.

\begin{Prop}[\cite{MMO25}, Lemma 2.8, Remark 2.9, Remark 2.13]
	There is an adjunction
	\begin{center}
		\begin{tikzcd}
			G\Top_{\textup{wp}} \arrow[rr, "\rr"', shift right] &  & {\Gamma[G\Top_{\textup{wp}}]} \arrow[ll, "\lL"', shift right]
		\end{tikzcd}
	\end{center}
	where $\rr$ maps a $G$-space $X$ to the $\Gamma$-$G$-space $\textbf{n}_+\mapsto X^n$ and $\lL$ maps a $\Gamma$-$G$-space $Y$ to $Y_1$. The unit is given by level-wise Segal maps.
\end{Prop}

\begin{Rem}[\cite{MMO25}, Remark 2.9]
	Starting in the non-equivariant case, Segal \cite{Seg74} did not require the functor $X\colon\Gamma\to G\Top_*$ to be a $\Top_*$-functor and instead required that $X_0$ is contractible and to call $X$ reduced if it is a point. The main point is that without enrichment, $\lL$ is not left adjoint to $\rr$. As mentioned, by Lemma \ref{lem:MMO-lem.1.17} the requirement that $X$ is a $\Top_*$-functor or in our case a $G\Top_*$-functor forces $X_0$ to be a point.
\end{Rem}

\subsubsection{\texorpdfstring{$\Gamma_G$}{Gamma_G}-\texorpdfstring{$G$}{G}-spaces}\label{sec:Gamma_G-G-spaces}

\begin{Def}[\cite{Shi89}, \S1; \cite{MMO25}, Definition 2.32, Convention 1.4]
	Let $\Gamma_G\subset G\Topu_*$ be the full subcategory of finite $G$-sets of the form $\textbf{n}^\alpha_+$ for $\alpha\colon G\to\Sigma_n$ a group homomorphism. These are finite $G$-sets with underlying set $\textbf{n}_+$ and $G$-action specified by $gi:= \alpha(g)(i)$ for each $g\in G,i\in\textbf{n}$. Here $G$ acts on the basepoint $0$ trivially.
\end{Def}

Note that given the trivial homomorphism $\varepsilon_n\colon G\to\Sigma_n$ with $\varepsilon_n(g)=1$ for all $g\in G$ one recovers finite based sets with trivial $G$-action via $\textbf{n}^{\varepsilon_n}_+=\textbf{n}_+$. Hence there is a natural inclusion $\Gamma\to \Gamma_G$ and we identify $\Gamma\subset\Gamma_G$.

\begin{Rem}[\cite{MMO25}, Definition 2.32 and remark after] $ $
	\begin{enumerate}
		\item Note that $\Gamma_G$ is itself enriched in $G\Top_*$ where the mapping space $\Gamma_G(\textbf{m}^\beta_+,\textbf{n}_+^\alpha)$ carries the discrete topology and its base point is the unique map factoring through $\textbf{0}_+$. 
		\item Since finite $G$-sets are disjoint unions of $G$-orbits they are $G$-CW complexes with skeletal dimension $0$. We thus obtain an additional inclusion
		\[\Gamma_G\to G\underline{\CW}_*^\CC=G\underline{\CW}_*.\]
	\end{enumerate}
\end{Rem}

The following definition is originally due to Shimakawa \cite{Shi89} but we use the modern formulation following \cite{MMO25}.

\begin{Def}[\cite{MMO25}, Definition 2.33]
	A $\Gamma_G$-$G$-space is a $G\Top_*$-functor $\Gamma_G\to G\Topu_*$ whose restriction to $\Gamma$ is a $\Gamma$-$G$-space. Morphisms are $G\Top_*$-natural transformations. We denote the category of $\Gamma_G$-$G$-spaces by
	\[\Gamma_G[G\Topu_*].\]
\end{Def}

To define specialness conditions on $\Gamma_G$-$G$-spaces we need to define Segal maps for all finite based $G$-sets. The multiplication maps $\varphi\colon X(\textbf{n}_+^\alpha)\to X(\textbf{1}_+)$ are analogously defined via induced maps $\hat\varphi\colon \textbf{n}_+^\alpha\to \textbf{1}_+$.

\begin{Def}[\cite{MMO25}, Definition 1.3]
	Let $X$ be a $G\times\Sigma_n$-space and let $\alpha\colon G\to\Sigma_n$ be a group homomorphism. Define $X^\alpha$ to be the $G$-space with same underlying space and \textit{$\alpha$-twisted} $G$-action given by
	\[g._\alpha x:=(g,\alpha(g)).x\]
\end{Def}

Writing $\Lambda_\alpha$ for the graph of the homomorphism $\alpha$ there exists a homeomorphism of fixed point spaces $(X^\alpha)^G\cong X^{\Lambda_\alpha}$.

\begin{Rem}
	For $X$ a based $G$-space $X^n$ is a $G\times\Sigma_n$-space with $\Sigma_n$-action given by permutation of factors. In this case we can further spell out the $\alpha$-twisted action via
	\[g._\alpha (x_1,\dots,x_n)=(gx_{\alpha^{-1}(g)(1)},\dots,gx_{\alpha^{-1}(g)(n)}).\]
	The resulting space $(X^n)^\alpha$ is canonically $G$-homeomorphic to the mapping space $F(\textbf{n}_+^\alpha,X)$ after spelling out the $G$-actions (cf. \cite{MMO25}, Example 1.5). 
\end{Rem}

\begin{Constr}[\cite{Shi89} Definition 1.3; \cite{MMO25}, Definition 2.34 and remark before]
	Let $\textbf{n}_+^\alpha$ be a finite based $G$-set and let $X\colon \Gamma_G\to G\Topu_*$ be a $\Gamma_G$-$G$-space. For $\textbf{n}_+^{\varepsilon_n}$ a trivial $G$-set we write $X(\textbf{n}_+^{\varepsilon_n})=X_n$ for short. Denote with $(X_1^n)^\alpha$ the space $X_1^n$ with twisted $G$-action
	\[g.(x_1,\dots,x_n)=(gx_{\alpha(g)^{-1}(1)},\dots,gx_{\alpha(g)^{-1}(n)}).\]
	Note that by definition of the $G$-action on $(X_1^n)^\alpha$ this space can be identified with the based $G$-space $F(\textbf{n}^\alpha_+,X_1)$ with conjugation $G$-action. Define a based $G$-map $\varepsilon\colon \textbf{n}_+^\alpha\land \textbf{n}_+^\alpha\to \textbf{1}_+=S^0$ by via the Kronecker delta function, i.e. by setting $\varepsilon(i,j)=1$ if $i=j$ and $\varepsilon(i,j)=0$ if $i\ne j$. At this point we heavily use that $G\Topu_*$ is closed with respect to the smash product. The $G$-map $\varepsilon$ has an adjoint $G$-map $\textbf{n}_+^\alpha\to \Gamma_G(\textbf{n}_+^\alpha,\textbf{1}_+)$. The $\Gamma_G$-$G$-space $X$ induces a $G$-map $\Gamma_G(\textbf{n}_+^\alpha,\textbf{1}_+)\to \Gamma_G(X(\textbf{n}_+^\alpha),X_1)$. Composing these two and taking adjoints we obtain a $G$-map 
	\[\partial\colon \textbf{n}_+^\alpha\land X(\textbf{n}_+^\alpha)\to X_1.\]
	This map is given by $\partial_{\textbf{n}_+^\alpha}(j,y)=X(\delta_j)(y)$ where $\delta_j\colon \textbf{n}^\alpha_+\to\textbf{1}_+$ is induced by the $j$-th projection map. We define the \textit{Segal map} to be the adjoint of $\partial$
	\[\delta\colon X(\textbf{n}_+^\alpha)\to \Gamma_G(\textbf{n}_+^\alpha,X_1) = F(\textbf{n}_+^\alpha,X_1)\cong (X_1^n)^\alpha.\]
\end{Constr}

\subsubsection{\texorpdfstring{$\Gamma_\II$}{Gamma_I}-\texorpdfstring{$G$}{G}-spaces}\label{sec:Gamma_I-G-spaces}

Having recalled the existing two equivariant generalisations of $\Gamma$-spaces we can now generalise this to the incomplete setting. 

\begin{Def}
	Let $\II$ be a disk-like indexing system. Let $\Gamma_\II\subset \Gamma_G$ be the full subcategory of those finite $G$-sets $\textbf{n}^\alpha_+\in\Gamma_G$ such that $\textbf{n}^\alpha\in\II(G)$.
\end{Def}

\begin{Rem}
	Let $\II$ be a disk-like indexing system.
	\begin{enumerate}[(1)]
		\item By definition of indexing systems (I1) all trivial sets are contained in $\II(G)$, i.e. disjoint unions of trivial $G$-orbits $G/G\cong *$. Hence there are natural inclusions of full subcategories
		 \[\Gamma\to \Gamma_\II.\]
		By definition we also have inclusions
		\[\Gamma_\II\to \Gamma_G\text.\]
		\item Note that for $\TT$ the trivial and $\CC$ the complete indexing system, $\Gamma_\TT=\Gamma$ and $\Gamma_\CC=\Gamma_G$. We will usually continue to write $\Gamma$ and $\Gamma_G$ instead of $\Gamma_\TT$ and $\Gamma_\CC$.
		\item If we do not assume that $\II$ is disk-like it can happen that for two different non-disk-like indexing systems $\II,\JJ$ one has equal categories $\Gamma_\II$ and $\Gamma_\JJ$. This is because only $\II(G)$ is used in the definition of $\Gamma_\II$ and the lower parts of the indexing system $\II(H)$ are omitted. For $G=C_{p^2}$ consider the two transfer systems
		\begin{center}
			\begin{tikzpicture}[scale = 0.7]
				\filldraw (0,0) circle (1pt);
				\filldraw (0,0.5) circle (1pt);
				\filldraw (0,-0.5) circle (1pt);
				\draw (0,0) -- (0,-0.5);
				\node at (-0.7,0) {$\II=$};
				\node at (1.3,0) {and};
			\end{tikzpicture}
			\begin{tikzpicture}[scale = 0.7]
				\filldraw (0,0) circle (1pt);
				\filldraw (0,0.5) circle (1pt);
				\filldraw (0,-0.5) circle (1pt);
				\node at (-0.7,0) {$\JJ=$};
			\end{tikzpicture}
		\end{center} 
		The two corresponding indexing systems have equal $G$-level $\II(G)=\JJ(G)$ and therefore $\Gamma_\II=\Gamma_\JJ$. In this case $\II$ is not-disk like and $\JJ$ is the maximal disk-like indexing system contained in $\II$. 
		
		In general, for $\II$ an indexing system, let $\II^\textup{d}$ be the maximal disk-like indexing system contained in $\II$. Then $\II(G)=\II^\textup{d}(G)$ since otherwise $\II^\textup{d}$ would not be maximal disk-like and thus $\Gamma_\II=\Gamma_{\II^\textup{d}}$. Note that $\II^\textup{d}$ always exists since $\Ind(G)$ is finite and has a minimal disk-like indexing system, the trivial indexing system.
	\end{enumerate}
\end{Rem}

Recall the notations $\Tr^\textup{d}(G)$ for the lattice of disk-like transfer systems and $\Ind^\textup{d}(G)$ the lattice of disk-like indexing systems (see Remark \ref{rem:disk-like-sublattice}). By definition it holds that $\Tr^\textup{d}(G)\cong \Ind^\textup{d}(G)$.

\begin{Lem}\label{lem:inclusions-of-Pi_I}
	Any inclusion $\II\subset\JJ$ of disk-like transfer systems induces an inclusion $\Gamma_\II\to\Gamma_\JJ$. Hence, the categories $\Gamma_\II$ form a poset isomorphic to $\textup{\Ind}^\textup{d}(G)$ or $\textup{\Tr}^\textup{d}(G)$.
\end{Lem}

\begin{proof}
	Since $\II\subset\JJ$ also $\II(G)\subset\JJ(G)$. Hence, if $\textbf{n}^\alpha\in \II(G)$ also $\textbf{n}^\alpha\in \JJ(G)$. Since $\Gamma_\II$ and $\Gamma_\JJ$ are full subcategories we obtain inclusions $\Gamma_\II\to \Gamma_\JJ$.
\end{proof}

\begin{Def}\label{def:Gamma_I-G-spaces}
	A $\Gamma_\II$-$G$-space is a $G\Top_*$-functor $\Gamma_\II\to G\Topu_*$ whose restriction to $\Gamma$ is a $\Gamma$-$G$-space. Morphisms are respective $G\Top_*$-natural transformations. Denote the corresponding category by
	\[\Gamma_\II[G\Topu_*].\]
\end{Def}

\begin{Rem}
	We will show that (see Remark \ref{rem:Gamma_I-lands-in-well-pointed}) any $\Gamma_\II$-$G$-space $X$ is level-wise well pointed. Hence, we could also write
	\[\Gamma_\II[G\Topu_{\textup{wp}}].\]
\end{Rem}

The second reason why we use disk-like transfer systems is that restrictions are well-behaved. 

\begin{Rem}\label{rem:restriction-Gamma_I-G-spaces-to-H}
	Note that we can restrict the $G\Top_*$-category $\Gamma_\II$ to an $H\Top_*$-category with objects being those finite $H$-sets $T$ such that $T\cong \res_H^G T'$ for some $T'\in \II(G)$. Since $\II$ is assumed to be disk-like this is equivalent to requiring that $T\in \II(H)$ using Remark \ref{rem:double-coset-formula} and Proposition \ref{prop:restriction-finite-G-sets}. This uses indexing system axioms (I2) through (I6) as well as the disk-like property. Hence, any $\Gamma_\II$-$G$-space has an underlying $\Gamma_\II$-$H$-space $\res_H^G X$.
	
	Let $X\colon \Gamma_\II\to G\Topu_*$ be a $\Gamma_\II$-$G$-space. Its evaluation on some finite $H$-set $T$ will also be denoted by $X(T)$ implicitly using the underlying $\Gamma_\II$-$H$-space structure of $\res_H^G X$.
\end{Rem}

\subsection{Specialness conditions}\label{sec:specialness}

We now want to generalise the three known specialness notions of \cite{Shi89}, \cite{Shi91} and \cite{MMO25} to allow for arbitrary disk-like indexing data. We define two classes of specialness conditions which we will later show to be equivalent. The key distinguishing factor is the indexing system data.

\subsubsection{\texorpdfstring{$\ff_\bullet^\II$}{F_bullet^I}-specialness}

We will begin by defining a lattice of specialness conditions on $\Gamma$-$G$-spaces and then extend to $\Gamma_\II$-$G$-spaces.

For a $\Gamma$-$G$-space $X$, note that $\Sigma_n$ acts on $X_n$ since $\Sigma\subset \Gamma$. Additionally, $\Sigma_n$ also acts on $X_1^n$ via permuting the entries in the product. Hence, the Segal maps are $G$-maps of $(G\times\Sigma_n)$-spaces. For $\alpha\colon H\to \Sigma_n$ a group homomorphism we write $\Lambda_\alpha$ for its graph and $\textbf{n}^\alpha$ for its associated finite $H$-set.

\begin{Def}
	Let $\II$ be an indexing system. Define
	\[\ff_n^\II:=\{\Lambda_\alpha\subset G\times\Sigma_n \mid \textbf{n}^\alpha\in \II(H)\}.\]
	A $\Gamma$-$G$-space $X$ is called \textit{$\ff_\bullet^\II$-special} if for each $n\ge 0$ and $\Lambda_\alpha\in \ff_n^\II$ the Segal maps $\delta\colon X_n\to X_1^n$ are weak $\Lambda_\alpha$-equivalences. A map $f\colon X\to Y$ of $\Gamma$-$G$-spaces is called \textit{$\ff_\bullet^\II$-level equivalence} if for each $n\ge 0$ and $\Lambda_\alpha\in\ff_\bullet^\II$ the map $f_n\colon X_n\to Y_n$ is a weak $\Lambda_\alpha$-equivalence.
\end{Def}

Note that this definition is valid for arbitrary indexing systems, but we will later restrict to disk-like indexing systems. One can define $\ff_\bullet^\II$-specialness even more generally at a loss of coherency.

\begin{Rem}\label{rem:more-general-Gamma-G-spaces}
	We technically do not need $\II$ to be an indexing system, but in order to obtain a homotopy coherent set of transfers we do. When defining the Segal machine it will become apparent why we need indexing systems. But for any set of exponents $\EE$ in the sense of \cite{Rub25}, Definition 2.3, that is, a collection $\EE=(T_i)_{i\in I}$ of finite $H_i$-sets, we can define $\ff_\bullet^\EE$-special $\Gamma$-$G$-spaces by requiring weak homotopy equivalences of Segal maps for all $T_i\in \EE$.
\end{Rem}

\begin{Rem}\label{rem:edge-cases-special-and-homotopy-assumption} $ $
	\begin{enumerate}
		\item We recover the two notions of specialness in \cite{MMO25}. For $\II=\TT$, the trivial indexing system, this recovers the \textit{naive} or \textit{classic} notion of `special' $\Gamma$-$G$-space and for $\II=\CC$, the complete indexing system, we recover the \textit{genuine} notion of `$\ff_\bullet$-special' $\Gamma$-$G$-space also used in \cite{Shi91}.
		\item We may assume without loss of generality that for any $\Lambda_\alpha\in\ff_\bullet^\II$, for $\II$ some indexing system, any weak $\Lambda_\alpha$-homotopy equivalence $\delta$ is a $\Lambda_\alpha$-homotopy equivalence. This is true since one can do cofibrant replacements of the spaces of an $\ff_\bullet^\II$-special $\Gamma$-$G$-space (cf. \cite{MMO25}, \S 8.2).
	\end{enumerate}	
\end{Rem}

More generally we can define $\ff_\bullet^\II$-specialness for arbitrary $\Gamma_\JJ$-$G$-spaces using only the underlying $\Gamma$-$G$-space.

\begin{Def}
	Let $\II,\JJ$ be indexing systems. A $\Gamma_\JJ$-$G$-space $X$ is called \textit{$\ff_\bullet^\II$-special} if for any group homomorphism $\alpha\colon H\to\Sigma_n$ such that $\textbf{n}^\alpha\in\II(H)$ the associated Segal map $\delta\colon X_n\to X_1^n$ is a weak $\Lambda_\alpha$-equivalence. A map of $\Gamma_\JJ$-$G$-spaces $f\colon X\to Y$ is called \textit{$\ff_\bullet^\II$-level $G$-equivalence} if for any such $\alpha$ the maps $f\colon X_n\to Y_n$ are weak $\Lambda_\alpha$-equivalences. Denote the full subcategory of $\ff_\bullet^\II$-special $\Gamma_\JJ$-spaces by
	\[\Gamma_\II[G\Topu_*]^{\ff_\bullet^\JJ\text{-spc}}.\]
\end{Def}

In the next section we will justify these definitions by constructing transfer maps.

\subsubsection{\texorpdfstring{$\II$}{I}-specialness}

We will again begin by defining a lattice of specialness conditions this time on $\Gamma_G$-$G$-spaces and then extend to $\Gamma_\II$-$G$-spaces.

\begin{Def}
	Let $\II$ be an indexing system. A $\Gamma_G$-$G$-space $X$ is called \textit{$\II$-special} if the maps $\delta\colon X(\textbf{n}_+^\alpha)\to (X_1^n)^\alpha$ are weak homotopy equivalences for all finite based $G$-sets $\textbf{n}^\alpha\in\II(G)$. A map $f\colon X\to Y$ of $\Gamma_G$-$G$-spaces is called \textit{$\II$-level $G$-equivalence} if for any such $\textbf{n}_+^\alpha$ the maps $f(\textbf{n}_+^\alpha)\colon X(\textbf{n}_+^\alpha)\to Y(\textbf{n}_+^\alpha)$ are weak $G$-equivalences.
\end{Def}

\begin{Rem}
	The complete indexing system $\CC$ is disk-like. A $\CC$-special $\Gamma_G$-$G$-space is the same thing as a `special' $\Gamma_G$-$G$-space as defined in \cite{Shi89}, Definition 1.3 or \cite{MMO25}, Definition 2.35. 
	
	The trivial indexing system $\TT$ is also disk-like. A $\TT$-special $\Gamma_G$-$G$-space requires that the Segal maps are weak $G$-equivalences only for trivial $G$-sets $\textbf{n}^{\varepsilon_n}_+$.
\end{Rem}

We can now generalise this to underlying $\Gamma_\II$-$G$-spaces.

\begin{Def}
	Let $\II\subset\JJ$ be indexing systems. A $\Gamma_\JJ$-$G$-space $X$ is called \textit{$\II$-special} if for any finite $G$-set $\textbf{n}^\alpha\in\II(G)$ the associated Segal map $\delta\colon X(\textbf{n}^\alpha_+)\to (X_1^n)^\alpha$ is a weak $G$-equivalence. A map of $\Gamma_\JJ$-$G$-spaces $f\colon X\to Y$ is called \textit{$\II$-level $G$-equivalence} if for any finite $G$-set $\textbf{n}^\alpha\in\II(G)$ the map $f\colon X(\textbf{n}^\alpha_+)\to Y(\textbf{n}^\alpha_+)$ is a weak $G$-equivalence. Denote the full subcategory by
	\[\Gamma_\II[G\Topu_*]^{\JJ\text{-spc}}.\]
\end{Def}

\subsection{Transfer maps on equivariant \texorpdfstring{$\Gamma$}{Gamma}-spaces}\label{sec:transfer-maps-Gamma}

Note that the homotopy commutative coherent monoid structure of any special $\Gamma_\II$-$G$-space is fully contained on the underlying classically special $\Gamma$-$G$-space. In particular, by definition it should be clear that $\ff_\bullet^\TT$-specialness and $\TT$-specialness are equivalent and provide the necessary multiplication maps. The additional structure encoded by indexing data gives rise to transfer maps.

Let $X$ be a based $G$-space and let $K\subset H\subset G$ be subgroups. There are natural \textit{restriction} maps $r_K^H\colon X^H\to X^K$ given by inclusions of fixed points. Maps $t_K^H\colon X^K\to X^H$ in the other direction are called \textit{internal transfer maps}.

\begin{Def}
	Let $\textbf{n}_+^\alpha$ be a finite based $H$-set. An \textit{external $\textbf{n}_+^\alpha$-transfer} on $X$ is an $H$-map $(X^n)^\alpha\to X$ or equivalently an $H$-map $F(\textbf{n}_+^\alpha,X)\to X$.
\end{Def}

\begin{Rem}
	The existence of non-trivial transfer maps depends on the extra structure of $X$. \cite{BH15} show that if $X$ is an algebra over an $N_\infty$-operad then there exist external transfer maps for any finite $G$-set in the indexing system of admissible sets of the operad. We will next show that special equivariant $\Gamma$-spaces also give rise to external transfer maps depending on choices of indexing systems.
\end{Rem}

Any external transfer map gives rise to an internal transfer map in the following way.

\begin{Lem}\label{lem:external->internal-transfer}
	Let $X$ be a $G$-space with external transfer map $\overline{t}_{H/K}$ for $K\subset H$ a subgroup and $H/K$ the finite $H$-set given by the orbit space. Then there exists an internal transfer map $t_K^H\colon X^K\to X^H$ on $X$.
\end{Lem}

\begin{proof}
	Choose coset representatives $r_1,\dots,r_n$ of $H/K$ with $\rho\colon H\to \Sigma_n$ is the permutation representation given by $hr_i=h_{\rho(h)(i)}$. Then $t_K^H(x):=\overline{t}_{H/K}(r_1x,\dots,r_n x)$ defines a map $X^K\to X^H$ since for any $h\in H$.
	\begin{align*}
		h\overline{t}_{H/K}(r_1x,\dots,r_n x) 
		&= \overline{t}_{H/K}(h._\rho(r_1x,\dots,r_n x)) \\
		&= \overline{t}_{H/K}(hr_{\rho(h)^{-1}(1)}x,\dots,hr_{\rho(h)^{-1}(n)} x) \\
		&= \overline{t}_{H/K}(r_{\rho(h)^{-1}\rho(h)(1)}x,\dots,r_{\rho(h)^{-1}\rho(h)(n)} x) \\
		&= \overline{t}_{H/K}(r_1x,\dots,r_n x)\qedhere
	\end{align*}
\end{proof}

We are not aware that there is any way to obtain an external transfer map from an internal one. The external transfer map is an $H$-map between $H$-spaces and the internal one is just a map between fixed point spaces with residual Weyl group actions.

We can construct external transfer maps on special equivariant $\Gamma$-spaces.

\begin{Constr}\label{constr:transfer-I-special-Gamma-J}
	Let $\II\subset\JJ$ be disk-like indexing systems and let $X\colon \Gamma_\JJ\to G\Topu_*$ be an $\II$-special $\Gamma_\JJ$-$G$-space. Without loss of generality assume that all required Segal maps are $G$-homotopy equivalences and, just a weak homotopy equivalences, using Remark \ref{rem:edge-cases-special-and-homotopy-assumption}. Let $T$ be a finite $G$-set in $\II$ and choose a homotopy inverse $\delta^{-1}$ to the Segal map $\delta\colon X(T_+)\to X_1^{\times T}$. Define an external transfer $G$-map $\overline{t}_T\colon X_1^{\times T}\to X_1$ by
	\begin{center}
		\begin{tikzcd}
		X(1)^{\times T} \arrow[r, "\delta^{-1}"] \arrow[rd, "\overline{t}_T"', dotted] & X(T_+) \arrow[d, "\varphi"] \\& X(1)                     
	\end{tikzcd}
	\end{center}
	Hence, $\overline{t}_T$ is a $G$-map implying that it is an external transfer map. Via Lemma \ref{lem:external->internal-transfer} we also obtain internal transfer maps $t_H^G\colon X^H\to X^G$ for all $G/H\in\II(G)$. 
	
	To also construct external transfer maps for finite $H$-sets we need to use the disk-like property of $\JJ$. Let $S$ be a finite $H$-set in $\II(H)$. Using the disk-like property and the double coset formula (Remark \ref{rem:double-coset-formula}, Proposition \ref{prop:restriction-finite-G-sets}) there is a finite $G$-set $S'$ and an inclusion $S\to \res_H^G S'$. This induces an inclusion $i\colon X(S)\to S(\res_H^G S')$ on the underlying restricted $\Gamma_\JJ$-$H$-space. Note that for any finite $G$-set it holds that $\res_H^GX^{\times T}\cong (\res_H^G X)^{\times \res_H^G T}$ since
	\[\res_H^G X^{\times T} \cong \res_H^G F(T_+,X)\cong F(\res_H^G T,\res_H^G X)\cong (\res_H^G X)^{\times \res_H^G T}.\]
	Construct an external $S_+$-transfer map $t_{S_+}\colon \res_H^G X_1^{\times S}\to \res_H^G X_1$ by setting
	\begin{center}
		\begin{tikzcd}
			\res_H^G X_1^{\times S} \arrow[r, "i_*"] \arrow[rrd] & \res_H^G X_1^{\times \res_H^GS'} \arrow[r, "\cong"] & \res_H^G (X_1^{\times S'}) \arrow[d, "\res_H^G \overline{t}_{S'}"] \\  &                                                     & \res_H^GX_1                                                       
		\end{tikzcd}
	\end{center}
	Note that the external transfer map $\overline{t}_{S'}$ only exists since $\II$ is disk-like and $\II\subset\JJ$ is disk-like. Using Lemma \ref{lem:external->internal-transfer} again we also obtain internal transfer maps $t_K^H\colon X_1^K\to X_1^H$.
\end{Constr}

\begin{Rem}\label{rem:induction-does-not-work}
	One might try to construct transfer maps for finite $H$-sets not by using the disk-like property but instead using induction. There also exists an inclusion $i'\colon S\to \res_H^G\ind_H^G S$ but beware that $\ind_H^G S'$ is in $\II(G)$ if and only if additionally $G/H\in \II(G)$. This is not always the case for disk-like indexing systems. Take for example the following disk-like $G=C_{pq}$-transfer system of (cf. \cite{Rub21a}, \S 2.3, Example 4.13)
	
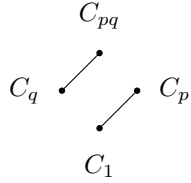
\begin{center}
		\begin{tikzpicture}
			\filldraw (0.5,0) circle (1pt);
			\filldraw (-0.5,0) circle (1pt);
			\filldraw (0,0.5) circle (1pt);
			\filldraw (0,-0.5) circle (1pt);
			\draw (-0.5,0) -- (0,0.5);
			\draw (0.5,0) -- (0,-0.5);
			\node at (1,0) {$C_p$};
			\node at (-1,0) {$C_q$};
			\node at (0,1) {$C_{pq}$};
			\node at (0,-1) {$C_1$};
		\end{tikzpicture}
	\end{center}
	To induce the $C_p$-set $C_p/C_1$ up to a $C_{pq}$-set we need the relation $C_p\to C_{pq}$, but this is not part of the indexing system.
\end{Rem}

Non-equivariantly $\pi_0(X_1)$, for $X$ a special $\Gamma$-space, is a commutative monoid. Equivalently we see that $\underline{\pi}_0(X_1)$ is an $\II$-incomplete semi Mackey functor in the sense of \cite{BH21}, Definition 7.21.

\begin{Prop}\label{prop:Mackey-functor-I-special}
	Let $\II\subset\JJ$ be disk-like indexing systems and let $X\colon \Gamma_\JJ\to G\Topu_{\textup{wp}}$ be an $\II$-special $\Gamma_\JJ$-$G$-space. Then $\underline{\pi}_0(X_1)$ given by
	\[\underline{\pi}_0(X_1)(G/H)=\pi_0(X_1^H).\]
	is an $\II$-incomplete semi Mackey functor with evident restriction maps given by inclusion of fixed points and transfer maps induced by the $t_K^H$ from Construction \ref{constr:transfer-I-special-Gamma-J}.
\end{Prop}

We can do a similar construction for $\ff_\bullet^\II$-special $\Gamma_\JJ$-$G$-spaces.

\begin{Constr}\label{constr:transfer-F_bullet-I-special-Gamma-J}
	Let $\II$ and $\JJ$ be disk-like indexing systems. Let $X$ be an $\ff_\bullet^\II$-special $\Gamma_\JJ$-$G$-space such that for $X$ the Segal maps $\delta\colon X_n\to X_1^n$ are $\Lambda_\alpha$-homotopy equivalences for all $n\ge 0$ and $\alpha\colon H\to \Sigma_n$ such that $\textbf{n}^\alpha\in \II(H)$. By definition 
	\[\delta^{\Lambda_{\alpha}}\colon X_n^{\Lambda_{\alpha}}\to (X_1^n)^{\Lambda_{\alpha}}\] is a homotopy equivalence which is equivalent to a homotopy equivalence 
	\[(\delta^\alpha)^H\colon (X_n^\alpha)^H\to ((X_1^n)^\alpha)^H\] which is equivalent to $\delta\colon X_n^\alpha\to (X_1^n)^\alpha$ being an $H$-homotopy equivalence. Choosing a homotopy inverse we obtain a map $(X_1^n)^\alpha\to X_n^\alpha$. Now, this is not quite enough to define an external transfer map. If $\II\subset\JJ$ the the $\Gamma_\JJ$-$G$-space admits a multiplication $G$-map $X(\textbf{n}_+^\alpha)\to X_1$ but there is a priory no reason that a $G$-map $X_n^\alpha\to X_1$ exists.
	
	In the next section (Proposition \ref{prop:equivalence-G-actions-on-Gamma-I-space}) we will demonstrate that there is a $G$-isomorphism $X_n^\alpha\cong X(\textbf{n}_+^\alpha)$ so that there are transfer maps
	\begin{center}
		\begin{tikzcd}
			(X_1^n)^\alpha \arrow[r, "\delta^{-1}"] \arrow[rrd, "\overline{t}_{\textbf{n}_+^\alpha}"'] & X_n^\alpha \arrow[r, "\cong"] & X(\textbf{n}_+^\alpha) \arrow[d, "\varphi"] \\ & & X_1                                        
		\end{tikzcd}
	\end{center}
	One can analogously extend this construction to finite $H$-sets so that we obtain the following fact.
\end{Constr}

\begin{Prop}\label{prop:Mackey-functor-F_bullet^I-special}
	Let $\II\subset\JJ$ be disk-like indexing systems and let $X\colon \Gamma_\JJ\to G\Topu_{\textup{wp}}$ be an $\ff_\bullet^\II$-special $\Gamma_\JJ$-$G$-space. Then $\underline{\pi}_0(X_1)$ given by
	\[\underline{\pi}_0(X_1)(G/H)=\pi_0^H(X_1).\]
	is an $\II$-incomplete semi Mackey functor with evident restriction maps given by inclusion of fixed points and transfer maps induced by the $t_K^H$ from Construction \ref{constr:transfer-F_bullet-I-special-Gamma-J}.
\end{Prop}

In the next section we will show that $\ff_\bullet^\II$ and $\II$-special $\Gamma_\JJ$-$G$-spaces are equivalent. It follows from the fourth comparison theorem (Theorem \ref{thm:4th-comparison-theorem}) that the two Mackey functors constructed in Proposition \ref{prop:Mackey-functor-I-special} and \ref{prop:Mackey-functor-F_bullet^I-special} are equivalent.

\begin{Rem}
	In \cite{BH15} external transfer maps are constructed from operadic data. Since we have not established a comparison between incomplete operadic and Segal machines we cannot directly compare these transfer maps. Using Construction \ref{constr:transfer-F_bullet-I-special-Gamma-J} and \cite{Sch18}, Proposition B.52 and Proposition B.54 one observes that the above defined transfer maps after applying prolongation $\PP_\II$ as defined in Section \ref{sec:construction-Segal-machine} agree with the transfer maps of $G$-spectra (cf. \cite{Sch23}, Definition 4.13). In particular, their existence is equivalent to the existence of the inverse of the Wirthm\"uller map.
\end{Rem}

\subsection{Comparison results}\label{sec:comparison}

In this section we wish to compare $\Gamma_\JJ$-$G$-spaces with $\Gamma_\II$-$G$-spaces for indexing systems $\II,\JJ$ as well as $\ff_\bullet^\II$-special and $\II$-special $\Gamma_\JJ$-spaces. The results of this section generalise results of \cite{Shi91} and \cite{MMO25}, Theorem 2.38 to the incomplete setting. We will assume that all indexing systems are disk-like.

First, we want to generalise the fact that $\Gamma$-$G$-spaces and $\Gamma_G$-$G$-spaces are equivalent (cf. \cite{Shi91}).

\begin{Lem}\label{lem:forgetful-functors}
	Let $\II\subset \JJ$ be (disk-like) indexing systems. There are forgetful functors
	\[\uu_\II^\JJ\colon\Gamma_\JJ[G\Topu_*] \to \Gamma_\II[G\Topu_*].\]
\end{Lem}

\begin{proof}
	For $\II\subset\JJ$ there is an inclusion $\Gamma_\II\to\Gamma_\JJ$ by Lemma \ref{lem:inclusions-of-Pi_I}. Therefore, a $G\Top_*$-functor $X\colon\Gamma_\JJ\to G\Topu_*$ induces a $G\Top_*$-functor $\uu_\II^\JJ X\colon\Gamma_\II\to G\Topu_*$ by precomposing with the inclusion. Since being a $\Gamma_\II$-$G$-space solely depends on the underlying functor $\Gamma\to G\Topu_*$, $\uu_\II^\JJ X$ is a $\Gamma_\II$-$G$-space. In particular, there is a commutative diagram
	\begin{center}
		\begin{tikzcd}
			\Gamma \arrow[r] \arrow[rd] & \Gamma_\JJ \arrow[r, "X"]                         & G\Topu_* \\
                         & \Gamma_\II \arrow[u] \arrow[ru, "\uu_\II^\JJ X"'] &          
		\end{tikzcd}
	\end{center}
\end{proof}

By abstract nonsense the functor $\uu_\II^\JJ$ has a left adjoint \textit{prolongation functor} $\pp_\II^\JJ$ in form of a left Kan extension (cf. \cite{MMSS01}, Proposition 3.2) since $\Gamma_\II$ and $\Gamma_\JJ$ are small. Whatever description of this left Kan extension given in Remark \ref{rem:prolongation-existence} suits the situation best will be denoted by $\pp_\II^\JJ X$. Levelwise this prolongation functor is described by the categorical tensor product $\pp_\II^\JJ X(T_+):= T^\bullet_+ \otimes_{\Gamma_\II}X$.

Since the property of being a $\Gamma_\II$-$G$-space solely depends on the underlying functor $\Gamma\to G\Topu_*$, the adjoint diagram of the diagram used in the proof of Lemma \ref{lem:forgetful-functors} shows that if $X$ is a $\Gamma_\II$-$G$-space, then $\pp_\II^\JJ X$ is a $\Gamma_\JJ$-$G$-space. This implies the following proposition.

\begin{Prop}
	Let $\II\subset \JJ$ be (disk-like) indexing systems. There are adjoint prolongation functors
	\[\pp_\II^\JJ\colon \Gamma_\II[G\Topu_*] \to \Gamma_\JJ[G\Topu_*].\]
\end{Prop}

\begin{Rem}
	Note that for $\II\subset\JJ\subset\KK$ sub-indexing systems we have that $\uu_\JJ^\KK\uu_\II^\JJ X = \uu_\II^\KK X$ for any $\Gamma_\KK$-$G$-space $X$. Since Kan extensions are unique up to isomorphisms we also see that $\pp_\II^\JJ\pp_\JJ^\KK Y\cong \pp_\II^\KK Y$ for any $\Gamma_\II$-$G$-space $Y$.
\end{Rem}

In the following theorem we will show that the adjunctions above form equivalences of categories. Indeed we will prove that already
\begin{center}
	\begin{tikzcd}
		{\Fun(\Gamma_\II,G\Topu_*)} \arrow[rr, "\pp_\II^\JJ", shift left] &  & {\Fun(\Gamma_\JJ,G\Topu_*)} \arrow[ll, "\uu_\II^\JJ", shift left]
	\end{tikzcd}
\end{center}
specifies an equivalence of categories and the corresponding specialisations to $\Gamma_\II$-$G$-spaces then follow immediately since the needed cofibrancy properties only rely on the underlying $\Gamma$-$G$-spaces. In the special case $\II=\TT$ and $\JJ=\CC$ this was originally proven by Shimakawa \cite{Shi91}. We will follow the modernised proof of Theorem 2.38 in \cite{MMO25}.

\begin{Thm}[1st comparison theorem]\label{thm:1st-comparison}
	Let $\II\subset\JJ$. The adjoint pairs of functors
	\begin{center}
		\begin{tikzcd}
			{\Gamma_\II[G\Topu_*]} \arrow[rr, "\pp_\II^\JJ", shift left] &  & {\Gamma_\JJ[G\Topu_*]} \arrow[ll, "\uu_\II^\JJ", shift left]
		\end{tikzcd}
	\end{center}
	specifies an equivalences of categories.
\end{Thm}

\begin{proof}
	The inclusion $\Gamma_\II\hookrightarrow\Gamma_\JJ$ is full and faithful which implies that the unit of the adjunction $\eta\colon X\to \uu_\II^\JJ\pp_\II^\JJ X$ is a natural isomorphism. To show that the counit is also a natural isomorphism we need to work a little harder. Define $\varepsilon\colon \pp_\II^\JJ\uu_\II^\JJ Y\to Y$ for a $G\Top_*$-functor $Y\colon \Gamma_\JJ \to G\Topu_*$ to be the map which evaluated at $T_+\in \Gamma_\JJ$ is of the form
	\[\varepsilon\colon \pp_\II^\JJ\uu_\II^\JJ Y(T_+)=\int^{S_+\in \Gamma_\II}\Gamma_\JJ(S_+,T_+)\land \uu_\II^\JJ Y(S_+) \to Y(T_+)\]
	and is explicitly given by $\varepsilon(\phi,y)=Y(\phi)(y)$ for $\phi\in \Gamma_\JJ(S_+,T_+)$ and $y\in \uu_\II^\JJ Y(S_+)$. The map $\varepsilon$ is $G$-equivariant since
	\[\varepsilon(g.\phi,gy)=Y(g.\phi)(gy)\overset{\text{($*$)}}{=} g.Y(\phi)(gy)=gY(\phi)(g^{-1}gy)=g\varepsilon(\phi,y)\]
	 using that $Y$ is a $G\Top_*$-functor and thus $g.Y(\phi)=Y(g.\phi)$ in ($*$). The map $\varepsilon$ is easily seen to be well defined and continuous using that $Y$ is a $G\Top_*$-functor.
	 
	 The $G$-map $\varepsilon$ has an inverse given by $\varepsilon^{-1}(y):=(\iota^{-1},Y(\iota)(y))$ for $y\in Y(T_+)$ and $\iota\in \Gamma_\JJ(T_+,|T|_+)$ the function whose underlying function on sets is the identity. Here we write $|S|$ for the underlying set with trivial $G$-action. Note that $|S_+|=|S|_+$. For this we need to use that $|S|_+\in\Gamma_\JJ$, i.e. $\JJ$ fulfils axiom (I1) of indexing systems. On the one hand we have that $\varepsilon(\varepsilon^{-1}(y))=\varepsilon(\iota^{-1},Y(\iota)(y))=Y(\iota^{-1})(Y(\iota)(y))=Y(\iota^{-1}\circ\iota)(y)=y$ and on the other hand $\varepsilon^{-1}(\varepsilon(\phi,y))=\varepsilon^{-1}(Y(\phi)(y))=(\iota^{-1},Y(\iota)(Y(\phi)(y)))$. Since $\iota\circ\phi\colon S_+\to |T|_+$ is a morphism in $\Gamma_\II$ we find in the quotient that
	 \[(\iota^{-1},Y(\iota\circ\phi)(y))\sim (\iota^{-1}\circ\iota\circ\phi,y)=(\phi,y).\]
	 Hence $\varepsilon^{-1}\circ \varepsilon \cong \Id$. Since $\varepsilon$ is a $G$-map the inverse is also a $G$-map. To conclude, note that $\varepsilon^{-1}$ is a continuous map since $Y$ is a $G\Top_*$-functor and the map is a map inside a quotient of a coproduct whose components are continuous.
	 
	 This shows that $\varepsilon$ is a $G\Top_*$-natural isomorphism and therefore the adjunction of functor categories is an equivalence of categories. The proposition follows, as previously mentioned, since the property of being a $\Gamma_\II$ or $\Gamma_\JJ$-$G$-space only depends on the underlying functor $\Gamma\to G\Topu$.
\end{proof}

Since any indexing system sits between the trivial and the complete one we find the following commutative diagram of equivalences of categories
\begin{center}
	\begin{tikzcd}
		{\Gamma[G\Topu_*]} \arrow[rd, "\pp_\TT^\II", shift left] \arrow[rr, "\pp^\TT_\CC", shift left] &                                                                                                     & {\Gamma_G[G\Topu_*]} \arrow[ld, "\uu^\CC_\II", shift left] \arrow[ll, "\uu^\TT_\CC", shift left] \\                                                                                                & {\Gamma_\II[G\Topu_*]} \arrow[ru, "\pp_\II^\CC", shift left] \arrow[lu, "\uu_\TT^\II", shift left] &                                                                                                  
	\end{tikzcd}
\end{center}

For any pair $\II,\JJ$ of indexing systems we defined that a $\Gamma_\JJ$-$G$-space $X$ is $\ff_\bullet^\II$-special if the Segal maps $\delta\colon X_n\to X_1^n$ are weak $\Lambda_\alpha$-equivalences for any $\alpha\colon H\to\Sigma_n$ with $\textbf{n}_+^\alpha\in \II(H)$. If in addition $\II\subset\JJ$, we called $X$ $\II$-special if the Segal maps $\delta\colon X(\textbf{n}_+^\alpha)\to (X_1^n)^\alpha$ are weak $G$-equivalences. The next two comparison theorems compare these two notions.

\begin{Thm}[2nd comparison theorem]
	Let $\II\subset\JJ\subset\KK$ be indexing systems and let $X$ be a $\Gamma_\JJ$-$G$-space. Then $X$ is an $\II$-special $\Gamma_\JJ$-$G$-space if and only if $\pp_\JJ^\KK X$ is an $\II$-special $\Gamma_\KK$-$G$-space. Conversely, let $Y$ be a $\Gamma_\KK$-$G$-space. Then $Y$ is an $\II$-special $\Gamma_\KK$-$G$-space if and only if $\uu_\JJ^\KK Y$ is an $\II$-special $\Gamma_\JJ$-$G$-space. In particular, there are equivalences of categories
	\[\pp_\JJ^\KK\colon \Gamma_\JJ[G\Topu_*]^{\II\textup{-spc}} \rightleftarrows \Gamma_\KK[G\Topu_*]^{\II\textup{-spc}}:\uu_\JJ^\KK.\]
\end{Thm}

\begin{proof}
	Let $X$ be an $\II$-special $\Gamma_\JJ$-$G$-space. We need to show that for all $\textbf{n}^\alpha\in\II(G)$ the Segal maps $\pp_\JJ^\KK X(\textbf{n}^\alpha_+)\to (\pp_\JJ^\KK X_1)^n)^\alpha$ are weak $G$-equivalences. Note that since $\textbf{n}_+^\alpha\in\Gamma_\JJ$ we have that
	\[\pp_\JJ^\KK X(\textbf{n}^\alpha_+) = \int^{\textbf{m}^\beta_+\in\Gamma_\JJ}\Gamma_\KK(\textbf{m}^\beta_+,\textbf{n}^\alpha_+)\land X(\textbf{m}^\beta_+) \cong X(\textbf{n}^\alpha_+)\]
	using the fact that $\Gamma_\KK(\textbf{m}^\beta_+,\textbf{n}^\alpha_+)= \Gamma_\JJ(\textbf{m}^\beta_+,\textbf{n}^\alpha_+)$ and therefore we may apply the Yoneda Lemma for coends (cf. \cite{Lor21}, Proposition 2.2.1/ Remark 4.3.5). The same holds for $X_1$. Hence, all we need to show is that the Segal map $X(\textbf{n}^\alpha_+)\to (X_1^n)^\alpha$ is a weak $G$-equivalence. This is the case if and only if $X$ is $\II$-special.
	
	The second part of the proposition follows directly.
\end{proof}

The same proof idea can be used to show the following proposition.

\begin{Thm}[3rd comparison theorem]
	Let $\II,\JJ\subset\KK$ be indexing systems and let $X$ be a $\Gamma_\JJ$-$G$-space. Then $X$ is an $\ff_\bullet^\II$-special $\Gamma_\JJ$-$G$-space if and only if $\pp_\JJ^\KK X$ is an $\ff_\bullet^\II$-special $\Gamma_\KK$-$G$-space. Conversely, let $Y$ be a $\Gamma_\KK$-$G$-space. Then $Y$ is an $\ff_\bullet^\II$-special $\Gamma_\KK$-$G$-space if and only if $\uu_\JJ^\KK Y$ is an $\ff_\bullet^\II$-special $\Gamma_\JJ$-$G$-space. In particular, there are equivalences of categories
	\[\pp_\JJ^\KK\colon \Gamma_\JJ[G\Topu_*]^{\ff_\bullet^\II\text{-spc}}\rightleftarrows \Gamma_\KK[G\Topu_*]^{\ff_\bullet^\II\text{-spc}}:\uu_\JJ^\KK.\]
\end{Thm}

Given disk-like $\II\subset \JJ$ we next wish to compare the two notions of $\II$-special and $\ff_\bullet^\II$-special $\Gamma_\JJ$-$G$-spaces. This is the content of the fourth comparison theorem. Note that using the second comparison theorem it is enough to compare $\II$-special and $\ff_\bullet^\II$-special $\Gamma_\II$-$G$-spaces. This comparison uses the same ideas as in \cite{Shi91} and \cite{MMO25}, Theorem 2.38.

\begin{Obs}[\cite{Shi91}]\label{obsv:acions-on-fin-G-sets}
	 Let $\textbf{n}_+^\alpha$ be a finite $G$-set in $\Gamma_\II$ and let $\textbf{n}_+$ be the underlying finite set with trivial $G$-action. Let $\iota\colon \textbf{n}_+^\alpha\to \textbf{n}$ be the canonical bijection. The map $\iota$ determines a homomorphism $\rho\colon G\to\Sigma_n$ such that
	\begin{center}
 		\begin{tikzcd}
			\textbf{n}_+^\alpha \arrow[r, "\iota"] \arrow[d, "g.-"'] & \textbf{n}_+ \arrow[d, "\rho(g)"] \\
			\textbf{n}_+^\alpha \arrow[r, "\iota"']                  & \textbf{n}_+                     
		\end{tikzcd}
	\end{center}
	commutes for all $g\in G$.
\end{Obs}

\begin{Prop}[\cite{Shi91}, Proposition 2]\label{prop:equivalence-G-actions-on-Gamma-I-space}
	Let $X$ be a $G\Top_*$-functor $\Gamma_\II\to G\Topu_*$ and let $\iota$ be as in Observation \ref{obsv:acions-on-fin-G-sets}. Then
	\[X(\iota)\colon X(\textbf{n}_+^\alpha)\to X(\textbf{n}_+)^\alpha\]
	is a $G$-homeomorphism where $X(\textbf{n}_+)^\alpha=X_n^\alpha$ denotes the $G$-space $X_n$ with usual $\alpha$-twisted $G$-action via $g._\alpha x=X(\alpha(g))(gx)$.
\end{Prop}

It follows that the the inverse of the counit of the adjunction between $\Gamma$ and $\Gamma_\II$-$G$-spaces $\varepsilon$ identifies $X(\textbf{n}_+^\alpha)$ with $X_n^\alpha$. In particular, there is a $G$-homeomorphism
\[\varepsilon^{-1}\colon X(\textbf{n}_+^\alpha)\to X_n^\alpha\]
coming from the above proposition and the universal property of the coend describing prolongation.

\begin{Lem}\label{lem:equivalence-segal-maps}
	Let $X$ be a $\Gamma_\II$-$G$-space. Then $X$ is $\II$special if and only if for every $\textbf{n}^\alpha\in\II(G)$ the induced Segal map
	\[\delta'\colon (X_n)^\alpha\to (X_1^n)^\alpha\]
	is a weak $G$-equivalence.
\end{Lem}

\begin{proof}
	By Proposition \ref{prop:equivalence-G-actions-on-Gamma-I-space} and naturality of $\varepsilon$ we have a commutative square
	\begin{center}
		\begin{tikzcd}
X(\textbf{n}_+^\alpha) \arrow[rr, "\varepsilon^{-1}"] \arrow[d, "\delta"'] &  & (X_n)^\alpha \arrow[d, "\delta'"] \\
(X_1^n)^\alpha \arrow[rr, "\varepsilon^{-1}"]                              &  & (X_1^n)^\alpha                   
\end{tikzcd}
	\end{center}
	Passing to $\pi_n$ we see that $\delta$ is a weak $G$-equivalence if and only if $\delta'$ is so.
\end{proof}

\begin{Prop}\label{prop:equivalence-mapping-spaces}
	Let $\alpha\colon G\to \Sigma_n$ be a group homomorphism with graph $\Lambda_\alpha$ and let $X$ and $Y$ be $G\times\Sigma_n$-spaces. Recall that $X^\alpha$ denotes the space $X$ with $G$-action given by $g\cdot_\alpha x:=(g,\alpha(g)).x$. Then there is a homeomorphism
	\[F(X^\alpha,Y^\alpha)^G\cong F_{\Lambda_\alpha}(X,Y).\]
\end{Prop}

\begin{proof}
	The underlying non-$G$ map is the identity. Let $f\colon X^\alpha\to Y^\alpha$ be a $G$-fixed map with conjugation $G$-action. Hence, $g.f(x)=g\cdot_\alpha f(g^{-1}\cdot_\alpha x)=f(x)$ and so $f(g^{-1}\cdot_\alpha x)=g^{-1}\cdot_\alpha f(x)$. In other words, $g\cdot_\alpha f(x)=f(g\cdot_\alpha x)$.
\end{proof}

Lemma \ref{lem:equivalence-segal-maps} and Proposition \ref{prop:equivalence-mapping-spaces} together imply that for any finite $G$-set $\textbf{n}^\alpha\in \II(G)$ the associated Segal map $\delta\colon X(\textbf{n}_+^\alpha)\to (X_1^n)^\alpha$ is a weak $G$-equivalence if and only if the Segal map $\delta\colon (X_n)^\alpha \to (X_1^n)^\alpha$ is a weak $G$-equivalence. This is equivalent to the map $\delta\colon X_n\to X_1^n$ being a weak $\Lambda_\alpha$-equivalence.

The tricky part is that if $\delta\colon X_n\to X_1^n$ is a weak $\Lambda_\beta$-equivalence for $\beta\colon H\to\Sigma_n$ for $H\subsetneq G$ to show the converse implications. The idea of Shimakawa \cite{Shi91} was that one can induce $\textbf{n}^\beta_+$ up to a $G$-set and then recover $\textbf{n}^\beta_+$ via restriction and inclusion. In the incomplete case we cannot generally expect that the induced $G$-set exists as mentioned in Remark \ref{rem:induction-does-not-work}. We can however use the disk-like property.

The final comparison theorem is essentially due to \cite{Shi91} and \cite{MMO25}, Lemma 2.18, but we need to use the disk-like property instead of induction.

\begin{Thm}[4th comparison theorem]\label{thm:4th-comparison-theorem}
	Let $\II$ be a disk-like indexing system and let $X$ be an $\Gamma_\II$-$G$-space. Then $X$ is an $\II$-special $\Gamma_\II$-$G$-space if and only if it is an $\ff_\bullet^\II$-special $\Gamma_\II$-$G$-space. In particular, there is an equivalence of categories
	\[\Gamma_\II[G\Topu_*]^{\II\text{-spc}}\rightleftarrows \Gamma_\II[G\Topu_*]^{\ff^\II_\bullet\text{-spc}}.\]
\end{Thm}

\begin{proof}
	By the previous discussion it is clear that if $X$ is $\ff_\bullet^\II$-special it is also $\II$-special. To show the converse let $\alpha\colon H\to \Sigma_n$ be a homomorphism so that $\textbf{n}^\alpha\in\II(H)$. Using the disk-like property, or rather Remark \ref{rem:double-coset-formula} and Proposition \ref{prop:restriction-finite-G-sets} there is a homomorphism $\beta\colon G\to\Sigma_m$ such that $\textbf{m}^\beta\in \II(G)$ and such that there is an inclusion of $H$-sets $\iota\colon \textbf{n}^\alpha \to \res_H^G\textbf{m}_+^\beta$. Without loss of generality we may assume that the inclusion $\iota$ maps $i\mapsto i$ for $0\le i\le n$. Let $k:=m-n$ and define
	\[\textbf{k}^\gamma_+:=\textbf{m}^\beta_+\backslash\textbf{n}^\alpha\]
	so that as $H$-sets, we can write $\res_H^G\textbf{m}^{\beta}_+\cong \res_H^G\textbf{k}^{\gamma}_+\lor \textbf{n}^\alpha_+$. Note that we also have a projection map $\smash{\pi\colon \res_H^G \textbf{m}^{\beta}_+ \to \textbf{n}^{\beta}_+}$ given by mapping $i\mapsto i$ for $0\le i\le n$ and $i\mapsto 0$ for $i>n$. We assume that $X$ is $\II$-special. Hence, $\delta\colon X(\textbf{m}_+^\beta)\to (X_1^m)^\beta$ is a weak $G$-equivalence. Restricting to $H$ yields an $H$-equivalence. We then have a retract diagram
	\begin{center}
		\begin{tikzcd}
			\res_H^GX(\textbf{n}_+^\alpha) \arrow[d, "\delta"] \arrow[rr, "X(\iota)"] &  & \res_H^G X(\textbf{m}_+^\beta) \arrow[d, "\delta"] \arrow[rr, "X(\pi)"] &  & \res_H^GX(\textbf{n}_+^\alpha) \arrow[d, "\delta"] \\
			((\res_H^GX_1)^n)^\alpha \arrow[rr]                                       &  & ((\res_H^GX_1)^m)^\beta \arrow[rr]                                      &  & ((\res_H^GX_1)^n)^\alpha                          
		\end{tikzcd}
	\end{center}
	where the bottom arrows are the evident inclusion and projection of twisted products and where $\res_H^G X(\textbf{n}_+^\alpha)$ denotes the evaluation of the finite $H$-set on the underlying $\Gamma_\II$-$H$-space. Since the middle $\delta$ is a weak $H$-equivalence and the outer $\delta$ maps exhibit the inner $\delta$ as a weak retract, using that equivalences are closed under retracts, the outer $\delta$ maps are also weak $H$-equivalences. Via Proposition \ref{prop:equivalence-mapping-spaces} this shows that the Segal map $\delta\colon \res_H^GX_n\to \res_H^G X_1^n$ is a weak $\Lambda_\alpha$-equivalence.
\end{proof}

\begin{Cor}
	A map $f\colon X\to Y$ of $\Gamma_\II$-$G$-spaces is an $\II$-level $G$-equivalence if and only if it is an $\ff_\bullet^\II$-level equivalence.
\end{Cor}

\begin{proof}
	First note that $f\colon X_n\to Y_n$ is a weak $\Lambda_\alpha$-equivalence in and only if $f\colon X_n^\alpha\to Y_n^\alpha$ is a weak $G$-equivalence if and only if $f(\textbf{n}_+^\alpha)\colon X(\textbf{n}_+^\alpha)\to Y(\textbf{n}_+^\alpha)$ is a weak $G$-equivalence. The rest of the proof is the same as in the previous proposition. Clearly any $\ff_\bullet^\II$-level equivalence induces an $\II$-level equivalence. For the converse do the same constructions of maps $\iota\colon \textbf{n}_+^\alpha\to \res_H^G \textbf{m}_+^\beta$ and $\pi\colon \res_H^G \textbf{m}_+^\beta\to \textbf{n}_+^\alpha$ as in the proof of Theorem \ref{thm:4th-comparison-theorem} for $\textbf{n}_+^\alpha$ a finite $H$-set in $\II(H)$. Then we have a retract diagram
	\begin{center}
		\begin{tikzcd}
			\res_H^GX(\textbf{n}_+^\alpha) \arrow[d, "\res_H^G f(\textbf{n}_+^\alpha)"'] \arrow[rr, "X(\iota)"] &  & \res_H^G X(\textbf{m}_+^\beta) \arrow[d, "\res_H^G f(\textbf{m}_+^\beta)"] \arrow[rr, "X(\pi)"] &  & \res_H^GX(\textbf{n}_+^\alpha) \arrow[d, "\res_H^G f(\textbf{n}_+^\alpha)"] \\
			\res_H^GY(\textbf{n}_+^\alpha) \arrow[rr]                                                           &  & \res_H^G Y(\textbf{m}_+^\beta) \arrow[rr]                                                       &  & \res_H^GY(\textbf{n}_+^\alpha)                                             
		\end{tikzcd}
	\end{center}
	where by assumption $\res_H^G f(\textbf{m}_+^\beta)$ is a weak $H$-equivalence so that also $\res_H^G f(\textbf{n}_+^\alpha)$ is a weak $H$-equivalence.
\end{proof}

\begin{Rem}[cf. \cite{MMO25}, Remark 2.39]\label{rem:Gamma_I-lands-in-well-pointed}
	As noted, for $Y$ a $\Gamma_\II$-$G$-space there is a $G$-homeomorphism $Y(\textbf{n}_+^\alpha)\cong Y_n^\alpha$. Since $\uu_\TT^\II Y$ is a $\Gamma$-$G$-space, the inclusion of the base-point $*\to Y_n$ is a $(G\times\Sigma_n)$-cofibration, and in particular $Y_n^\alpha$ is in $G\Top_{\textup{wp}}$. This implies $Y$ is indeed a functor $\Gamma_\II\to G\Topu_{\textup{wp}}$.
\end{Rem}


\section{Equivariant incomplete Segal machines}\label{sec:Segal-machines}

Fix a disk-like $G$-indexing system $\II$ together with a compatible $G$-universe $U$. This means that for any finite $H$-set $T$ in $\II$ there is a $G$-representation $V$ in $U$ so that $T$ embeds $H$-equivalently into $\res^G_H V$. For any such choice of universe and indexing system we will define an equivariant incomplete infinite loop space machine
\[\sS^{G,U}_\II\colon \Gamma_\II[G\Topu_*]^{\II\textup{-spc}}\longrightarrow \Omega\Sp_{\textup{cp}}^{G,U}.\]
Using the discussion in the previous chapter we can also replace $\II$-special $\Gamma_\II$-$G$-spaces with $\ff_\bullet^\II$-special $\Gamma$-$G$-spaces, but for the proof of the existence of $\sS$ we will explicitly need $\II$-special $\Gamma_\JJ$-$G$-spaces with $\II\subset \JJ$. For minimality we will use $\II$-special $\Gamma_\II$-$G$-spaces.

The construction of $\sS^{G,U}_\II$ and the proof idea of its existence is derived from the genuine case (i.e. the case $\II=\JJ=\CC$) given in \cite{MMO25} in \S 3 and \S 8. Their proof is a modernised version of the original proof given in \cite{Shi89} by Shimakawa. Knowing the proof in the genuine case the main difficulty in the incomplete case boils down to keeping track of equivariance and that constructions do not leave the boundaries of the indexing system. Therefore most of the proofs are deferred to the appendix with only brief instructions given how to alter the genuine proofs in the incomplete case.

\subsection{Constructing the Segal machines}\label{sec:construction-Segal-machine}

The goal of this section is to define a functor
\[\hat{\sS}_\II^{G,U}\colon\Gamma_\II[G\Topu_*]\longrightarrow  \Sp^{G,U}\]
which we will show in the next section to restricts to give the incomplete equivariant Segal machine $\sS^{G,U}_\II$. Define this functor by the concatenation of the following functors
\begin{center}
	\begin{tikzcd}
	{\Gamma_\II[G\Topu_*]} \arrow[r, "\bar{b}_\II"] \arrow[rr, "b_\II", bend right] & \Gamma_\II[G\Topu_* \arrow[r, "\PP_\II"] & {\Fun(\underline{G\CW}_*^\II,G\Topu_*)} \arrow[r, "R^\II_U"] & {\Sp^{G,U}}
	\end{tikzcd}
\end{center}
The functor $\bar{b}_\II$ is given by the geometric realization of the simplicial bar construction. The functor $\PP_\II$ is given by the prolongation functor which is the adjoint of the forgetful functor induced by the inclusion $\Gamma_\II\to G\underline{\CW}_*^\II$. The main functor of interest is $b_\II$, which we will show is naturally isomorphic to $\PP_\II\circ \bar b_\II$. The functor $R_U^\II$ is given by restricting a $G\Top_*$-functor $\smash{G\underline{\CW}_*^\II\to G\Topu_*}$ to $G$-representation spheres $\smash{S^V}$ for $V$ in $U$ a $G$-representation.

The bar construction was in the form used here originally introduced by May \cite{May75}, \S 7 and \S 12. A summary on the construction using the enriched terminology can be found in \cite{MMO25}, \S 3.1 and \S 3.2. 

\begin{Def}
	Let $\E$ be a $\V$-enriched category, $X\colon \E\to\underline\V$ a covariant, and $Y\colon \E\to \underline\V$ a contravariant functor. Denote by $B_*(Y,\E,X)$ the simplicial bar construction of $Y$ and $X$. Here $\underline\V$ denotes $\V$ enriched over itself with $\V$ cartesian closed. If $\V$ is cocomplete, denote by $B(Y,\E,X)$ the geometric realization of the simplicial object $B_*(Y,\E,X)$.
\end{Def}

We will mainly use this construction for $\V=G\Top_*$ with respect to the smash product and will occasionally need to use the construction for $\V=G\Top$ with respect to the cartesian product. Following the notation in \cite{MMO25} we will denote the bar construction in $G\Top_*$ by $B(Y,\E,X)$ and in $G\Top$ by $B^\times(Y,\E,X)$.

\begin{Rem}[\cite{MMO25}, Assumption 3.10]\label{rem:cofibrancy-assumptions-bar-construction}
	We will assume that the category $\E$ above has a zero object such that for all $a,b\in\E$, $\E(a,0)\cong *\cong \E(0,b)$, $\E$ is $G\Top_*$-enriched where the mapping $G$-space $\E(a,b)$ is based at the zero map $a\to 0\to b$, $\E(a,b)$ is well-pointed, and that the inclusion of the identity map into $\E(a,a)$ is a $G$-cofibration. If we furthermore assume that the functors $X,Y$ are $G\Top_*$-enriched such that $X(a)$ and $Y(a)$ are well-pointed for all $a\in\E$ this ensures that the bar constructions are given by geometric realizations of the Reedy cofibrant simplicial $G$-spaces in $G\Top_{\textup{wp}}$ (cf. \cite{MMO25}, Assumption 3.10).
	
	Note that for all indexing systems $\II$, the categories $\Gamma_\II$ fulfil the assumptions of $\E$. The zero object is $\textbf{0}_+=\{0\}$, the mapping spaces $\Gamma_\II(\textbf{n}_+^\alpha,\textbf{m}_+^\beta)$ consist of based maps with conjugation $G$-actions so that the second assumption above holds. The third assumption holds by Definition \ref{def:Gamma_I-G-spaces} and the $G$-cofibration condition is easily verified. Lastly note that $\Gamma_\II$-$G$-spaces fulfil the assumption on well-pointedness by Remark \ref{rem:Gamma_I-lands-in-well-pointed}.
\end{Rem}

\begin{Def}
	Define
	\[\bar b_\II\colon \Fun(\Gamma_\II,G\Topu_*)\to \Fun(\Gamma_\II,G\Topu_*)\]
	to be the functor which maps a functor $X\colon\Gamma_\II\to G\Topu_*$ to $B(\Gamma_\II,\Gamma_\II,X)\colon \Gamma_\II\to G\Topu_*$ which maps a finite based $G$-set $T_+$ in $\Gamma_\II$ to the $G$-space given by the bar construction $B(\Gamma_\II(-,T_+),\Gamma_\II,X)$, that is
	\[\bar b_\II X(T_+):=B(\Gamma_\II(-,T_+),\Gamma_\II,X).\]
\end{Def}

\begin{Prop}
	Let $X$ be a $\Gamma_\II$-$G$-space. Then $b_\II X$ is also a $\Gamma_\II$-$G$-space.
\end{Prop}

\begin{proof}
	We need to show that given an injection $\phi\colon S_+\to T_+$ the induced map $b_\II X(\phi)$ is a $(G\times\Sigma_\phi)$-cofibration. Consider $X$ as a trivial simplicial $G$-space $X_*$ given by $X_q=X$ for all $q\ge 0$ and with face and degeneracies given by identities. We can think of $\bar b_\II X(\phi)$ as the realization of the map
	\[\bar b_\II X(\phi)_*\colon B_*(\Gamma_\II(-,S_+),\Gamma_\II,X)\to B_*(\Gamma_\II(-,T_+),\Gamma_\II,X).\]
	To show that $b_\II X(\phi)$ is a $(G\times\Sigma_\phi)$-cofibration, using Theorem \ref{thm:MMO-thm-1.13}, it suffices to show that $\bar b_\II X(\phi)_*$ is a levelwise $(G\times\Sigma_\phi)$-cofibration under the assumption that the bar construction is a Reedy cofibrant simplicial $G$-space. This assumption holds by Remark \ref{rem:cofibrancy-assumptions-bar-construction}. Levelwise, the map $\bar b_\II X(\phi)_q$ is a wedge of maps of the form
	\[\Gamma_\II(a_q,\phi)\land \id\land \dots\land \id.\]
	It is not hard to check that the induced map $\Gamma_\II(a_q,\phi)\colon \Gamma_\II(a_q,S_+)\to \Gamma_\II(a_q,T_+)$ is a $(G\times\Sigma_\phi)$-cofibration between finite based $G\times\Sigma_{\phi}$-spaces.
\end{proof}

From now on, we will consider $\bar b_\II$ as an endofunctor
\[\bar b_\II\colon \Gamma_\II[G\Topu_*]\to \Gamma_\II[G\Topu_*].\]

The following Proposition generalises Proposition 3.17 and 3.20 in \cite{MMO25} and is directly implied by Proposition 3.13 in \cite{MMO25}.

\begin{Prop}\label{prop:bar-construction-approximation-of-X}
	Let $X$ be a $\Gamma_\II$-$G$-space. Then the map
	\[\varepsilon\colon B(\Gamma_\II,\Gamma_\II,X)\to X\]
	is an $\II$-level $G$-equivalence of $\Gamma_\II$-$G$-spaces. Hence, $X$ is $\II$-special if and only if $B(\Gamma_\II,\Gamma_\II,X)$ is $\II$-special.
\end{Prop}

\begin{Rem}
	The above proposition is the reason why we need to use $\Gamma_\II$-$G$-spaces and why using $\Gamma$-$G$-spaces does not suffice. For $\II=\TT$, the trivial indexing system, we can only expect $\TT$-level-equivalences from the above proposition. For $\JJ$ an indexing system with $\II\subset\JJ$ and $X$ an $\II$-special $\Gamma_\JJ$-$G$-space it also implies that $\varepsilon$ is a $\JJ$-level equivalence so that $B(\Gamma_\JJ,\Gamma_\JJ,X)$ is also $\II$-special. Hence, there is technically no need to consider $\Gamma_\II$-$G$-spaces and it is enough to consider $\II$-special $\Gamma_G$-$G$-spaces.
	
	 As observed in \cite{MMO25}, Warning 3.18, the map $\varepsilon$ is in general not an $\ff_\bullet^\CC$-level equivalence. Hence, one cannot expect $B(\Gamma_\CC,\Gamma_\CC,X)$ to be $\ff_\bullet^\CC$-special, even if $X$ is so. The same holds for $\II$ in place of $\CC$.
\end{Rem}

Recall that any finite based $G$-set in $\Gamma_\II$ is an $\II$-$G$-CW complex which is $0$-skeletal so that there is an inclusion
\[\Gamma_\II\to G\underline\CW_*^\II.\]
Indeed, for disk-like indexing systems $\II\subset\JJ$ we have natural inclusions $\Gamma_\II\to\Gamma_\JJ$ and $G\underline\CW_*^\II\to G\underline\CW_*^\JJ$ so that the following diagram commutes
\begin{center}
	\begin{tikzcd}
		\Gamma_\II \arrow[d] \arrow[r] & G\underline\CW_*^\II \arrow[d] \\
		\Gamma_\JJ \arrow[r]           & G\underline\CW_*^\JJ          
	\end{tikzcd}
\end{center}

\begin{Def}
	A \textit{$G\underline{\CW}_*^\II$-$G$-space} $Z$ is a $G\Top_*$-functor $Z\colon G\underline{\CW}_*^\II\to G\Topu_*$. Denote the category of $G\underline{\CW}^\II$-$G$-spaces by $G\underline{\CW}_*^\II[G\Topu_*]$.
\end{Def}

For disk-like indexing systems $\II\subset\JJ$ any $G\underline{\CW}_*^\JJ$-$G$-space $X$ has an underlying $G\underline{\CW}_*^\II$-$G$-space $\UU_\II^\JJ X$ by precomposition with the inclusion and it has similarly an underlying $\Gamma_\JJ$-$G$-space $\UU_\JJ X$. The commutative diagram above induces a commutative diagram
\begin{center}
	\begin{tikzcd}
		\Fun(\Gamma_\JJ,G\Topu_*) \arrow[d, "\uu_\II^\JJ"'] \arrow[r, "\UU_\JJ"] & \Fun(G\underline\CW_*^\JJ,G\Topu_*) \arrow[d, "\UU_\II^\JJ"] \\
		\Fun(\Gamma_\II,G\Topu_*) \arrow[r, "\UU_\II"']           & \Fun(G\underline\CW_*^\II,G\Topu_*)          
	\end{tikzcd}
\end{center}

By Remark \ref{rem:prolongation-existence} we have an adjoint prolongation functor $\PP_\II$ which is a left adjoint to $\UU_\II$. For a $G\Top_*$-functor $X\colon \Gamma_\II\to G\Topu_*$ define $\PP_\II X(A):= A^\bullet \otimes_{\Gamma_\II} X$ for $A^\bullet$ the representable contravariant functor $F(-,A)$ where $A$ is a based $\II$-$G$-CW complex. Similarly, any $G\underline{\CW}_*^\II$-$G$-space $Y$ can be prolonged to a $G\underline{\CW}_*^\JJ$-$G$-space $\PP_\II^\JJ X$. For this to exist note that the category of finite based $G$-CW complexes is skeletally small (see \cite{MMSS01}, Example 4.6). The adjoint diagram of the above implies the following proposition.

\begin{Prop}\label{prop:commutative-diagram-prolongation}
	For $\II\subset \JJ$, there is a commutative diagram
	\begin{center}
	\begin{tikzcd}
		\Fun(\Gamma_\II,G\Topu_*) \arrow[d, "\pp_\II^\JJ"'] \arrow[r, "\PP_\JJ"] & \Fun(\underline\CW_*^\II,G\Topu_*) \arrow[d, "\PP_\II^\JJ"] \\
		\Fun(\Gamma_\JJ,G\Topu_*) \arrow[r, "\PP_\JJ"']           & \Fun(\underline\CW_*^\II,G\Topu_*)          
	\end{tikzcd}
\end{center}
\end{Prop}

\begin{Def}
	Define a functor
	\[b_\II\colon \Gamma_\II[G\Topu_*]\to \Fun(G\underline\CW_*^\II,G\Topu_*)\]
	by mapping a $\Gamma_\II$-$G$-space $X$ to the $G\underline\CW_*^\II$-space $b_\II X$ which maps a finite $\II$-$G$-CW complex $A$ to the $G$-space $b_\II X(A)=B(A^\bullet,\Gamma_\II,X)$, where $A^\bullet$ is the representable contravariant functor $F(-,A)$.
\end{Def}

\begin{Prop}
	There is a natural isomorphism
	\[b_\II\cong \PP_\II\circ \bar b_\II.\]
\end{Prop}

\begin{proof}
	Let $X$ be a $\Gamma_\II$-$G$-space and let $A$ be a finite based $\II$-$G$-CW complex. There is a natural isomorphism
	\begin{align*}
		\PP_\II(\bar b_\II X)(A) = A^\bullet \otimes_{\Gamma_\II} B(\Gamma_\II,\Gamma_\II,X)
		\cong B(A^\bullet ,\Gamma_\II, X) = b_\II X(A)
	\end{align*}
	following from Lemma 3.7 in \cite{MMO25}.
\end{proof}

Given compatible $\II$ and $U$ we now explain how to obtain an orthogonal $G$-spectrum indexed on a $G$-universe $U$ from an \smash{$G\underline{\CW}_*^\II$-$G$-space}.

\begin{Constr}[\cite{MMO25}, Definition 2.26]\label{constr:the-functor-RUI}
	Let $X\colon G\underline\CW_*^\II\to G\Topu_*$ be a $G\underline\CW_*^\II$-$G$-space. Let $\VV(U)\subset G\underline\CW_*^\II$ be the full subcategory of representation spheres $S^V$ so that $V\subset U$ is a $G$-representation. Define a functor
	\[R^\II_U\colon \Fun(G\underline\CW_*^\II,G\Topu_*)\to \Sp^{G,U}\]
	as follows: On objects, let $R_U^\II X(V):= X(S^V)$. The morphism $G$-maps are induced by one point compactification and application of $X$
	\[\I_U(V,W)\xrightarrow{S^{(-)}}G\underline{\CW}_*^\II (S^V,S^W) \xrightarrow{X(-)} G\Topu_*(X(S^V),X(S^W)).\]
	This gives $R_U^\II X$ the structure of an $\I_U$-$G$-space. Since $X$ is a $G\underline\CW_*^\II$-$G$-space, for any $A,B$ in $G\underline\CW_*^\II$ the adjoint $G$-map $B\to F(A,A\land B)$ of the identity $\id_{A\land B}$ can be composed with $X$ to obtain a $G$-map $B\to F(X(A),X(A\land B)$ whose adjoint is the structure map
	\[X(A)\land B\to X(A\land B).\]
	The special cases $A=S^V$ and $B=S^W$ then give rise to the structure maps $\sigma$ of an orthogonal $G$-spectrum. Since this construction is given by juggling adjoint maps it is natural in all of its arguments giving rise to a $G\Top$-natural transformation
	\[R_U^\II X\land S\Rightarrow R_U^\II \circ \oplus\]
	making $R_U^\II X$ an orthogonal $G$-spectrum indexed on $U$.
\end{Constr}

\begin{Rem}
	The assumption of compatibility of $U$ and $\II$ is very fundamental. Otherwise there does not exist an inclusion $\VV(U)\subset G\underline{\CW}^\II_*$.
\end{Rem}

\begin{Thm}[Theorem A]
	For each indexing system $\II$ and compatible $G$-universe $U$, there exists a functor
	\[\sS^{G,U}_\II\colon \Gamma_\II[G\Topu_*]^{\II\textup{-spc}}\longrightarrow \Omega\Sp_{\textup{pc}}^{G,U}\]
	which maps $\II$-special $\Gamma_\II$-$G$-spaces to positive connected orthogonal $\Omega$-$G$-spectra indexed on $U$ so that the bottom structure map $\sS_\II^{G,U}X(S^0)\to \Omega^V\sS_\II^{G,U}X(S^V)$ is a group completion for all $V$ in $U$ with $V^G\ne 0$. This machine is given by restriction of $R_U^\II\circ b_\II$ to $\II$-special $\Gamma_\II$-$G$-spaces and therefore
	\[\sS_\II^{G,U}X(V):=B((S^V)^\bullet,\Gamma_\II,X).\]
\end{Thm}

\subsection{Sketch of proof of Theorem A}\label{sec:proof-strategy-theorem-A}

The proof strategy of Theorem A is similar to the genuine case described in \cite{Shi89}, \cite{MMO25}. We will give an overview on how to modify the genuine proof to the incomplete setting with detailed proofs if needed deferred to the appendix. We first introduce some terminology.

A based $G$-space $X$ is called \textit{$G$-connected} if $X^H$ is connected for any subgroup $H\subset G$. Note that for based spaces it always holds that $X^H\ne\emptyset$.

\begin{Def}[\cite{MMO25}, Definition 3.27]\label{def:G-connected}
	A $G\underline{\CW}_*^\II$-$G$-space $Z$ \textit{preserves connectivity} if $Z(A)$ is $G$-connected when $A$ is $G$-connected.
\end{Def}

For $f\colon A\to B$ a based $G$-map between based $G$-CW complexes define $Cf$ to be the mapping cone $B\cup_f(A\land I_+)$ where $I$ is the unit interval with trivial $G$-action.

\begin{Def}[\cite{MMO25}, Definition 3.26]\label{def:positive-linear}
	A $G\underline{\CW}_*^\II$-$G$-space $Z$ is called \textit{linear} if for any $G$-map $f\colon A\to B$ in $G\CW_*^\II$ with cofibre $i\colon B\to Cf$ the sequence
	\[Z(A)\xrightarrow{f_*}Z(B)\xrightarrow{i_*}Z(Cf)\]
	is a fibration sequence of based $G$-spaces. This means that the induced map $i_*\circ f_*$ is a weak $G$-equivalence. Furthermore, $Z$ is called \textit{positive linear} if this condition holds when $A$ is $G$-connected, but not necessarily in general.
\end{Def}

\begin{Rem}\label{rem:mapping-cone-is-I-G-CW}
	For this definition to be well defined we will need to show that $Cf$ is an $\II$-$G$-CW complex if $f$ is in $G\CW_*^\II$. It is evidently a $G$-CW complex. Since $A$ and $B$ are $\II$-$G$-CW complexes $\Phi(X),\Phi(Y)\subset\FF_\II$. By Proposition \ref{prop:properties-isotropy-families} we can compute that the isotropy of $Cf$ also lies in $\FF_\II$
	\[\Phi(Cf)=\Phi(B\cup_f(A\land I_+))\subset \Phi(B)\cup \Phi(A\land I_+)=\Phi(B)\cup \Phi(A)\subset \FF_\II.\]
	Hence, $Cf$ admits the structure of an $\II$-$G$-CW complex.
\end{Rem}

Recall that given a based $\II$-$G$-CW complex $A$, we denote by $A^\bullet$ the contravariant represented functor $F(-,A)\colon\Gamma_\II\to G\Topu_*$. The following theorem generalises Theorem 3.29 \cite{MMO25} to the incomplete setting.

\begin{Thm}[Positive linearity theorem]\label{thm:positive-linearity-theorem}
	Let $X$ be an $\II$-special $\Gamma_\II$-$G$-space. Then $b_\II X$ is positive linear and preserves connectivity.
\end{Thm}

We will need the following terminology generalising the terminology in \cite{MMO25}.

\begin{Def}
	A $G\underline{\CW}_*^\II$-$G$-space
	\begin{enumerate}
		\item \textit{commutes with (geometric) realization} if for any based simplicial $\II$-$G$-CW complex $A_*$ the natural $G$-map
		\[|Z(A_*)|\to Z(|A_*|)\]
		is a $G$-homeomorphism.
		\item \textit{preserves Reedy cofibrancy} if for any based simplicial $\II$-$G$-CW complex $A_*$, the simplicial $G$-space $Z(A_*)$ is Reedy cofibrant whenever $A_*$ is Reedy cofibrant.
		\item satisfies the \textit{wedge axiom} if for all finite based $\II$-$G$-CW complexes $A$ and $B$ in $G\CW_*^\II$ the natural map
		\[\pi\colon Z(A\lor B)\to Z(A)\times Z(B),\]
		induced by the canonical projection $G$-maps $\pi_A\colon A\lor B\to A$ and $\pi_B\colon A\lor B\to B$, is a weak $G$-equivalence.
	\end{enumerate}
\end{Def}

Let $X$ be a $\Gamma_\II$-$G$-space.

\begin{Lem}[The connectivity lemma]\label{lem:connectivity-lemma}
	 The $G\underline{\CW}_*^\II$-$G$-space $b_\II X$ preserves connectivity.
\end{Lem}

\begin{Lem}[Realization lemma]\label{lem:realization-lemma}
	The $G\underline{\CW}_*^\II$-$G$-space $b_\II X$ commutes with realization.
\end{Lem}

\begin{Lem}[Cofibrancy lemma]\label{lem:cofibrancy-lemma}
	The $G\underline{\CW}_*^\II$-$G$-space $b_\II X$ preserves Reedy cofibrancy.
\end{Lem}

The proofs of the previous three lemmas work completely analogously to the genuine setting. One just needs to replace $\Gamma_G$ by $\Gamma_\II$. In particular no assumption of specialness of $X$ is needed. Detailed proofs in the genuine case can be found in \cite{MMO25}, Lemma 8.9 and 8.15. For the next Lemma we will require that $X$ is $\II$-special. 

\begin{Lem}[Wedge lemma]\label{lem:wedge-lemma}
	Let $X$ be an $\II$-special $\Gamma_\II$-$G$-space. Then the associated $G\underline{\CW}_*^\II$-$G$-space $b_\II X$ that maps $A$ to $B(A^\bullet,\Gamma_\II,X)$ satisfies the wedge axiom.
\end{Lem}

The proof of the wedge lemma is more involved but again works essentially the same as in the genuine case. The proof uses extensively the fact that the indexing system $\II$ is closed under isomorphisms, coproducts and subobjects. It also requires an incomplete version of the \textit{invariance theorem} as stated in \cite{GMMO19b}, Theorem 2.6. We will give a statement and proof the the incomplete invariance theorem in Appendix \ref{app:invariance-theorem} and detailed proof of the wedge lemma in Appendix \ref{app:wedge-lemma}.

The positive linearity theorem then follows from the following more general theorem. It is a generalisation of Theorem 8.16 in \cite{MMO25}.

\begin{Thm}\label{thm:general-thm-implying-poslinthm}
	Let $Z$ be a $G\underline\CW_*^\II$-$G$-space that preserves connectivity, commutes with realization, preserves Reedy cofibrancy and satisfies the wedge axiom. Then $Z$ is positive linear.
\end{Thm}

The proof is the same as in the genuine case after replacing general $G$-CW complexes by $\II$-$G$-CW complexes. One then needs to note that mapping cones of maps in $G\CW_*^\II$ are again $\II$-$G$-CW complexes as remarked in Remark \ref{rem:mapping-cone-is-I-G-CW} and the fact that finite wedges of $\II$-$G$-CW are again such (see Proposition \ref{prop:constr-F-G-CW}).

\begin{Def}
	A $G\underline\CW_*^\II$-$G$-space $X$ is called \textit{$\II$-special} if it preserves connectivity and if its restriction to $\Gamma_\II$ is an $\II$-special $\Gamma_\II$-$G$-space.
\end{Def}

Of course we could also define $\II$-special $G\underline\CW_*^\JJ$-$G$-spaces for $\II\subset\JJ$ or more specifically one can define $\II$-special $G\underline{\CW}_*^\CC$-$G$-spaces. The following theorem is a generalisation to Lemma 3.31, \cite{MMO25}.

\begin{Thm}\label{thm:structure-theorem}
	Let $Z$ be a positively linear and $\II$-special $G\underline{\CW}^\II_*$-G-space. If $A$ is a finite based $G$-connected $\II$-$G$-CW complex, then $Z[A]=Z(A\land -)$ is a linear and $\II$-special $G\underline{\CW}^\II_*$-G-space.
\end{Thm}

The proof of this theorem will follow from the next two lemmas. To demonstrate the use of indexing system axioms we will give the proofs here. Again, they are generalisations of genuine proofs where we need to keep track of indexing data. The following Lemma is a generalisation of Lemma 8.24 and (2.22) in \cite{MMO25}.

\begin{Lem}\label{lem:thm-A-help-1}
	Let $Z$ be a positively linear and $\II$-special $G\underline{\CW}^\II_*$-G-space and let $\textbf{n}_+^\alpha$ be a finite $G$-set in $\Gamma_\II$. Denote by $Y[\textbf{n}_+^\alpha]$ the restriction of $Z[\textbf{n}_+^\alpha]$ to $\Gamma_\II$. Then $Y[\textbf{n}_+^\alpha]$ is an $\II$-special $\Gamma_\II$-$G$-space.
\end{Lem}

\begin{proof}
	We are given that $Y[\textbf{1}_+]$ is a special $\Gamma_\II$-$G$-space since $Y$ is a special $\Gamma_\II$-$G$-space and since for any finite $G$-set $\textbf{n}_+^\alpha\land \textbf{1}_+\cong \textbf{n}_+^\alpha$. Here we use that $\II$ is closed under isomorphisms (I2). For a finite $G$-set $\textbf{n}_+^\alpha$ as in the proposition and any other finite $G$-set $\smash{\textbf{m}_+^\beta}$ in $\Gamma_\II$ we have the following commutative diagram
	\begin{center}
		\begin{tikzcd}
{Y[\textbf{n}_+^\alpha](\textbf{m}_+^\beta)} \arrow[d, Rightarrow, no head] \arrow[r, "\delta"] & {Y[\textbf{n}_+^\alpha](\textbf{1})^\beta} \arrow[r, Rightarrow, no head] & Y(\textbf{n}_+^\alpha)^\beta \arrow[d, "\delta^\beta"] \\
Y(\textbf{n}_+^\alpha \land \textbf{m}_+^\beta) \arrow[r, "\delta"]                             & Y(\textbf{1})^{\alpha\land \beta} \arrow[r, Rightarrow, no head]          & (Y(\textbf{1})^\alpha)^\beta                                   
\end{tikzcd}
	\end{center}
	Here we use that $(\textbf{n}^\alpha\times \textbf{m}^\beta)_+\cong \textbf{n}^\alpha_+\land \textbf{m}^\beta_+$ and since $\textbf{n}^\alpha_+$ and $\textbf{m}^\beta_+$ are in $\Gamma_\II$ using that indexing systems are closed under products (I8) and isomorphisms (I2) we find that the bottom map $\delta$ is a weak $G$-equivalence by the assumption that $Y$ is a special $\Gamma_\II$-$G$-space. Also the right vertical $\delta$ is a weak $G$-equivalence by assumption and therefore also $\delta^\beta$ is. This implies that the top map $\delta$ is also a weak $G$-equivalence. This shows that the restriction to $\Gamma_\II$ of $Z[\textbf{n}_+^\alpha]$ is indeed $\II$-special.
	
	We also have to show that $Z[\textbf{n}_+^\alpha]$ preserves connectivity. But this follows immediately, since if $A$ is $G$-connected so is $\textbf{n}_+^\alpha\land A$ and $Z$ is by assumption $\II$-special and thus preserves connectivity.
\end{proof}

The following Lemma is a generalisation of Lemma 8.25 in \cite{MMO25}.

\begin{Lem}\label{lem:thm-A-help-2}
	Let $Z$ be a positively linear and $\II$-special $G\underline{\CW}^\II_*$-G-space and let $A$ be a finite based $G$-connected $\II$-$G$-CW complex. Denote by $Y[A]$ the restriction of $Z[A]$ to $\Gamma_\II$. Then $Y[A]$ is an $\II$-special $\Gamma_\II$-$G$-space.
\end{Lem}

\begin{proof}
	Let $S_s^1$ be the simplicial circle. Its $q$-simplices $(S^1_s)_q$ are finite based sets, hence finite based $G$-sets with trivial $G$-action. Note that $S_s^1$ is a simplicial $G$-set with isotropy in $\II$ using that indexing systems contain all trivial $G$-sets (I1). Using (I2) there is an isomorphism $(S^1_s)_q\cong \textbf{m}_+$. By Lemma \ref{lem:thm-A-help-1}, for each finite $G$-set $\textbf{n}_+^\alpha$ in $\Gamma_\II$ we have the simplicial special $\Gamma_\II$-$G$-space with $q$-simplices $Y[(S^1_s)_q\land \textbf{n}_+^\alpha]\cong Y[\textbf{m}_+\land \textbf{n}_+^\alpha]$. Its geometric realization gives an $\II$-special $\Gamma_\II$-$G$-space isomorphic to $Y[S^1\land \textbf{n}_+^\alpha]$.
	
	Let $A\to B\to C$ be a cofibre sequence in $G\underline{\CW}_*^\II$ so that $A$ is $G$-connected. Furthermore, let $\textbf{n}_+^\alpha$ be a finite based $G$-set in $\Gamma_\II$. Using that $Z$ is positive linear and noting that $Z[\textbf{n}_+^\alpha](A) = Z(\textbf{n}_+^\alpha\land A)\cong Z(A\land \textbf{n}_+^\alpha)$ we have a fibre sequence
	\begin{center}
		\begin{tikzcd}
			Z(A\land \textbf{n}_+^\alpha) \arrow[r] & Z(B\land \textbf{n}_+^\alpha) \arrow[r] & Z(C\land \textbf{n}_+^\alpha).
		\end{tikzcd}
	\end{center}
	We also have Segal $G$-maps $Z[A](\textbf{n}_+^\alpha)\to (Z[A](\textbf{1}_+)^n)^\alpha$. Applying $Z[\textbf{1}_+]$ and noting that $Z[\textbf{1}_+](A)=Z(A)$ we obtain another corresponding fibre sequence. This results in the following commutative diagram, which is a map of fibre sequences
	\begin{center}
		\begin{tikzcd}
			Z(A\land \textbf{n}_+^\alpha) \arrow[r] \arrow[d, "\delta"] & Z(B\land \textbf{n}_+^\alpha) \arrow[r] \arrow[d, "\delta"] & Z(C\land \textbf{n}_+^\alpha) \arrow[d, "\delta"] \\
			(Z(A)^n)^\alpha \arrow[r]                                   & (Z(B)^n)^\alpha \arrow[r]                                   & (Z(C)^n)^\alpha                                  
		\end{tikzcd}
	\end{center}
	Therefore, if $Y[A]$ and $Y[B]$ are special, then so is $Y[C]$. 
	
	Let $A$ be a finite based $\II$-$G$-CW complex. When $A$ is $G$-connected, we can replace $A$ by a $G$-homotopy equivalent based $\II$-$G$-CW complex whose $0$-skeleton is a point and whose attaching maps are based $G$-maps defined on $G$-spheres $G/H_+\land S^n$ for $n\ge 1$. Write $\textbf{n}_+^\alpha$ as the wedge of based orbits $G/H_+$ where each based orbit is represented by some $\textbf{k}^\rho_+$ with $\rho\colon G\to\Sigma_n$ a choice of representation of $G/H$ with $k=[G:H]$.
	
	The proof then follows by applying induction twice. First, induct on the dimension $n$. For $n=0$ consider the cofibre sequence $S^0\to D^1\to S^1$ and apply the argument above to deduce that for any $G/H$ in $\II(G)$ we have that $Y[G/H_+\land S^1]$ is $\II$-special. The same argument applies to cofibre sequences $S^{n-1}\to D^n\to S^n$ to deduce that $Y[G/H_+\land S^n]$ is $\II$-special. Secondly, induct on the number of cells in $A$.
\end{proof}

\begin{proof}[Proof of Theorem \ref{thm:structure-theorem}]
	This then is an immediate consequence of Lemma \ref{lem:thm-A-help-2} since $Z[A]$ is $\II$-special if and only if its restriction to $\Gamma_\II$ is $\II$-special.
\end{proof}

\begin{Thm}[Delooping theorem]\label{thm:delooping-theorem}
	Let $Z$ be a linear and $\II$-special $G\underline{\CW}_*^\II$-$G$-space. Then the adjoint structure map
	\[\tilde{\sigma}\colon Z(S^0)\to \Omega^VZ(S^V)\]
	is a weak $G$-equivalence.
\end{Thm}

The proof of this theorem works again similarly to the genuine case. One needs to extensively use that $S^V$ and $S(V)$ admit structures of finite $\II$-$G$-CW complexes. A key observation is that equivariant triangulations of $S(V)$ are indexed on finite $G$-sets with isotropy in $\FF_\II$. Here the choice of universe compatible with indexing data is integral. One then needs to repeatedly use that indexing systems are closed under isomorphisms, subobjects and coproducts and follow the proof of Theorem 8.27 and 8.32 in \cite{MMO25}. Since this is not entirely trivial we will give a complete proof in Appendix \ref{app:proof-delooping-theorem}. 

We can now complete the first part of the proof of Theorem A. The following theorem generalises Theorem 3.33 in \cite{MMO25} or rather Theorem B in \cite{Shi89}. The proof strategy is the same.

\begin{Thm}\label{thm:special-CW-space->-omega-G-spectrum}
	Let $Z$ be a positively linear $\II$-special $G\underline{\CW}_*^\II$-$G$-space. Then $R_U^\II(Z)$ is a positive $\Omega$-$G$-spectrum indexed on $U$.
\end{Thm}

\begin{proof}
	Let $V$ and $W$ be $G$-representations in $U$ so that $V^G\ne 0$. We must show that the adjoint structure map
	\[\tilde{\sigma}\colon Z(S^V)\to \Omega^WZ(S^{V\oplus W})=\Omega^WZ(S^V\land S^W)\]
	is a weak $G$-equivalence. Since representation spheres are $G$-connected as the one point compactification of an inner product $G$-space, $Z[S^V]$ is linear and $\II$-special by the structure theorem, Theorem \ref{thm:structure-theorem}. Setting $Z:=Z[S^V]$ and replacing $V$ with $W$ in the delooping theorem, Theorem \ref{thm:delooping-theorem}, shows that $\tilde{\sigma}$ is a weak $G$-equivalence. Hence $R_U^\II(Z)$ is a positive $\Omega$-$G$-spectrum.
\end{proof}

From the fact that the resulting spectrum of the Segal machine is a positive $\Omega$-$G$-spectrum by Theorem \ref{thm:special-CW-space->-omega-G-spectrum} and the connectivity lemma which we will show below (Lemma \ref{lem:connectivity-lemma}) it follows that all negative homotopy groups of $\sS_\II^{G,U}X$ for $\II$-special $X$ must vanish. Hence, the Segal machine produces by construction connective $G$-spectra.

It remains to show the group completion property.

\subsection{The group completion property.}\label{sec:group-completion}

Non-equivariantly, a Hopf space $Y$ is called \textit{group-like} if the monoid $\pi_0(Y)$ is a group. If $X$ is a special $\Gamma$-space, then $X_1$ is in general not group-like. But it is the case that the bottom structure map of the associated $\Omega$-spectrum obtained by prolonging to $CW$-complexes and then restricting to spheres is a \textit{group completion} (cf. \cite{Seg74}, \S 4). Group completions first came up and were studied extensively in the seventies. They were defined by \cite{DS76} but we will use the formulation given in \cite{MMO25}, Definition 1.6.

\begin{Def}
	A Hopf map $f\colon X\to Y$ is a \textit{group completion} if $Y$ is group-like, $\pi_0(f)\colon \pi_0(X)\to\pi_0(Y)$ is the Grothendieck group of the commutative monoid $\pi_0(X)$, and for every field of coefficients, $H_*(f)\colon H_*(X)\to H_*(Y)$ is the algebraic localization obtained by inverting the elements of the submonoid $\pi_0(X)$ of $H_*(X)$.
\end{Def}

As remarked in \cite{MMO25}, Remark 1.7, if $Y,Y'$ are group completions of $X$, then $Y$ and $Y'$ are weakly homotopy equivalent.

\begin{Def}[\cite{MMO25}, Definition 1.8]
	A Hopf $G$-space $X$ is \textit{group-like} if $X^H$ is group-like for any subgroup $H\subset G$. A Hopf $G$-map $f\colon X\to Y$ is a \textit{group completion} if for every subgroup $H\subset G$, the induced map $f^H\colon X^H\to Y^H$ on the underlying spaces is a group completion in the classical sense.
\end{Def}

The following remark allows us to reduce the group completion property of incomplete positive $\Omega$-$G$-spectra to the underlying naive $\Omega$-$G$-spectrum.

\begin{Rem}[\cite{MMO25}, Remark 1.23]\label{rem:group-completions-genuine-to-naive}
	If $V^G\ne 0$, we can write $V\cong\rr\oplus W$ and thus $S^V\cong S^1\land S^W$ and $\Omega^V\cong \Omega\Omega^W$. Then $\tilde{\sigma}\colon X_0\to \Omega^VX(V)$ factors as the composite
	\[X_0 \xrightarrow{\tilde{\sigma}}\Omega X_1 \xrightarrow{\Omega\tilde{\sigma}}\Omega\Omega^WX(R\oplus W)\cong \Omega^VX(V).\]
	If $X$ is a positive $\Omega$-$G$-spectrum, then the second arrow is a weak $G$-equivalence. Therefore, if $X_0\to \Omega X_1$ is a group completion, then so is $X_0\to \Omega^VX(S^V)$ for all $V$ such that $V^G\ne 0$. Hence, it is enough to compare the underlying classical positive $\Omega$-$G$-spectra obtained from our Segal machines.
\end{Rem}

\begin{Rem}\label{rem:prolongation-and-bar-construction}
	We need to compare bar constructions along the adjoint equivalences $(\pp_\II^\JJ,\uu_\II^\JJ)$. Using Lemma 3.7, \cite{MMO25} we have hat
	\[\pp_\II^\JJ B(\Gamma_\II,\Gamma_\II,X) = \Gamma_\JJ\otimes_{\Gamma_\II}B(\Gamma_\II,\Gamma_\II,X)\cong B(\Gamma_\JJ,\Gamma_\II,X).\]
	Therefore, the inclusion $\iota\colon \Gamma_\II\to\Gamma_\JJ$ induces a natural map of $\Gamma_\JJ$-$G$-spaces 
	\[\iota_*\colon \pp_\II^\JJ B(\Gamma_\II,\Gamma_\II,X)\to B(\Gamma_\JJ,\Gamma_\JJ,\pp_\II^\JJ X).\]
	We thus obtain commutative diagrams
	\begin{center}
		\begin{tikzcd}
			{\pp_\II^\JJ B(\Gamma_\II,\Gamma_\II,X)} \arrow[d, "\iota_*"'] \arrow[rr, "\pp_\II^\JJ\varepsilon"] &  & \pp_\II^\JJ X \\
			{B(\Gamma_\JJ,\Gamma_\JJ,\pp_\II^\JJ X)} \arrow[rru, "\varepsilon "']                                &  &              
		\end{tikzcd}
		\begin{tikzcd}
			{ B(\Gamma_\II,\Gamma_\II,X)} \arrow[d, "\uu_\II^\JJ\iota_*"'] \arrow[rr, "\varepsilon"] &  & X \\
			{\uu_\II^\JJ B(\Gamma_\JJ,\Gamma_\JJ,\pp_\II^\JJ X)} \arrow[rru, "\uu_\II^\JJ\varepsilon "']                                &  &              
		\end{tikzcd}
	\end{center}
	the second of which is obtained by applying $\uu_\II^\JJ$ to the first one and using that $\pp_\II^\JJ$ and $\uu_\II^\JJ$ form an equivalence of categories.
	
	In the first diagram, the diagonal arrow is a $\JJ$-level $G$-equivalence, but there is no reason to believe that $\iota$ or $\pp_\II^\JJ\varepsilon$ are also $\JJ$-level $G$-equivalences since that would imply that all three arrows in the second diagram are $\ff_\bullet^\JJ$-level equivalences which would contradict Warning 3.18 in \cite{MMO25}. In the second diagram $\uu_\II^\JJ\varepsilon$ is a $\ff_\bullet^\JJ$-level equivalence and the other two arrows are only $\ff_\bullet^\TT$-level equivalences.
\end{Rem}

We need the following result proven in \cite{MMO25}, \S 8.2.

\begin{Prop}[\cite{MMO25}, Proposition 2.30]\label{prop:MMO-prop.2.30}
	For $X$ is a $\TT$-special $\Gamma$-$G$-space $\smash{\sS_\TT^{G,\rr^\infty}X}$ is a positive $\Omega$-$G$-prespectrum whose bottom structure map is a group completion of $X_1$.
\end{Prop}

The following theorem for $\II=\CC$ is due to \cite{MMO25}, Proposition 3.23.

\begin{Prop}\label{prop:group-completions}
	Let $X$ be a $\TT$-special $\Gamma$-$G$-space. Then the positive naive $\Omega$-$G$-prespectra $\sS_\TT^{G,\rr^\infty}X$ or $\sS_\TT^{G,\rr^\infty}b_\TT X$ or $\sS_\TT^{G,\rr^\infty}\uu_\TT^\II b_\II \pp_\TT^\II X$ are level $G$-equivalent. Their bottom structural maps are compatible group completions of $G$-spaces equivalent to $X_1$.
\end{Prop}

\begin{proof}
	We prolong from $\Gamma_\TT$-$G$-spaces ($=\Gamma$-$G$-spaces) to $G\underline{\CW}_*^\TT$-$G$-spaces ($=\underline{\CW}_*$-$G$-spaces) via the functor $\PP_\TT$. This maps a $\Gamma_\TT$-$G$-space $X$ to the $G\underline{\CW}_*^\TT$-$G$-space $\PP_\TT X$ which evaluated at some $G$-trivial CW complex $A$ is given by $\PP_\TT X(A):=A^\bullet\otimes_\Gamma X$\footnote{In \cite{MMO25}, $A$ is taken to be a general $G$-CW complex, but we don't need this here since we restrict to spheres $S^n$ anyways.}. Applying $\PP_\TT(-)(A)=A^\bullet\otimes_\Gamma(-)$ to the second diagram in Remark \ref{rem:prolongation-and-bar-construction} with $\II=\TT$ and $\JJ=\II$ this yields the following commutative diagram
	\begin{center}
		\begin{tikzcd}
{ B(A^\bullet,\Gamma_\II,X)} \arrow[rr, "\cong"] \arrow[d]   &  & { A^\bullet\otimes_{\Gamma} B(\Gamma,\Gamma,X)} \arrow[rr, "A^\bullet\otimes_\Gamma\varepsilon"] \arrow[d, "A^\bullet\otimes_\Gamma\uu_\TT^\II\iota_*"'] &  & A^\bullet\otimes_\Gamma X \\
{B(A^\bullet,\Gamma_\II,\pp_\TT^\II X)} \arrow[rr, "\cong"'] &  & {A^\bullet\otimes_{\Gamma} \uu_\II^\JJ B(\Gamma_\JJ,\Gamma_\JJ,\pp_\II^\JJ X)} \arrow[rru, "A^\bullet\otimes_\Gamma\uu_\TT^\II\varepsilon"']             &  &                          
\end{tikzcd}
	\end{center}
	We use that $\pp_\TT^\II\uu_\TT^\II\cong \Id$ and $\PP_\TT\uu_\TT^\II\cong \UU_\TT^\II\PP_\II$. Using the comparison theorems we find that all spaces in the diagram above are $\TT$-special. Restricting to spheres $A=S^n$ we can then apply Proposition \ref{prop:MMO-prop.2.30}. For $A=S^1$ we get the group completion properties and for $A=S^0$ we see compatible weak $G$-equivalences. This implies that we have weak $G$-equivalences at level $1$ and therefore also on all levels $n\ge 1$ since we are comparing naive $\Omega$-$G$-prespectra.
\end{proof}

We conclude the proof of Theorem A by showing the following theorem which generalises \cite{MMO25}, Theorem 3.34.

\begin{Thm}
	Let $X$ be an $\II$-special $\Gamma_\II$-$G$-space and let $U$ be a compatible $G$-universe. Then the associated orthogonal $G$-spectrum $\sS^{G,U}_\II X$ is a positive $\Omega$-$G$-spectrum. Moreover, if $V^G\ne 0$, then the composite
	\[X_1\to B(\Gamma_\II,\Gamma_\II,X)_1=(\sS_\II^{G,U}X)(S^0)\to \Omega^V(\sS_\II^{G,U}X)(S^V)\]
	of $\eta$ and the structure $G$-map is a group completion.
\end{Thm}

\begin{proof}
	Since $X$ is $\II$-special $B(\Gamma_\II,\Gamma_\II,X)$ is $\II$-special using Proposition \ref{prop:bar-construction-approximation-of-X}. By the positive linearity theorem (Theorem \ref{thm:positive-linearity-theorem}) $b_\II X$ is positive linear and preserves connectivity. Hence, Theorem \ref{thm:special-CW-space->-omega-G-spectrum} implies the first part of the theorem. 
	
	Since the underlying $\Gamma$-$G$-space of $X$ is $\TT$-special, using the fourth comparison theorem (Theorem \ref{thm:4th-comparison-theorem}), we may apply Remark \ref{rem:group-completions-genuine-to-naive} and Proposition \ref{prop:group-completions} to obtain the above group completion.
\end{proof}


\section{Suspension and Eilenberg-MacLane spectra}\label{sec:suspension-E-ML-spectra}

The first two applications of our Segal machines are models for suspension and Eilenberg-MacLane spectra. Suspension $G$-spectra are almost never $\Omega$-$G$-spectra so that the $\Gamma_\II$-$G$-spaces modelling these are almost never special. We can still apply the functor $\hat{\sS}_\II^{G,U}$ to obtain orthogonal spectra which we can then compare with suspension $G$-spectra. In the case of Eilenberg-MacLane $G$-spectra we will only be able to model genuine Eilenberg-MacLane $G$-spectra for abelian $G$-groups. Hence, our incomplete machine gives no additional examples of Eilenberg-MacLane spectra over the machines constructed in \cite{MMO25} and \cite{GMMO19b} and only recovers those.

\subsection{Incomplete suspension spectra}\label{sec:suspension}

The key idea of this section is due to Segal (\cite{Seg74}, \S 2 and Proposition 3.6) but we will follow the modernised and equivariantly generalised exposition given in \cite{GMMO19b}, \S 6.1.

\begin{Def}
	Let $\Sigma_U^\infty\colon G\Topu_{\textup{wp}} \to \Sp^{G,U}$ be the $G\Top_*$ functor which maps a based $G$-space $A$ to the \textit{suspension spectrum} $\Sigma_U^\infty A$ of $A$ which is given by $\Sigma_U^\infty A(V)=\smash{\Sigma^V} A$ with evident structure maps.
\end{Def}

\begin{Rem}
	This functor has a right adjoint $\Omega_U^\infty\colon \Sp^{G,U}\to G\Topu_{\textup{wp}}$ called the \textit{$0$-th space functor} which maps an orthogonal $G$-spectrum $E$ indexed on $U$ to its zero space $E(0)$ (cf. \cite{MM02}, \S.I.6).
\end{Rem}

The following two constructions mimic and use notations from \cite{GMMO19b}, \S 6.1.

\begin{Constr}\label{constr:suspension-gamma-space}
	Let $X$ be a based $G$-space. Let $^\bullet X\colon\Gamma_\II\to G\Topu_{\textup{wp}}$ be the covariant $G\Top_*$-functor given by $^\bullet X(\textbf{n}_+^\alpha)=\textbf{n}_+^\alpha \land X$ on objects with evident induced maps. Note that there is an identification $\textbf{n}_+^\alpha \land X = (X^{\lor n})^\alpha$ by identifying a pair $[i,x]$ in the smash product with the element $x$ in the $i$-th copy of $X$ and noting that by definition the $G$-actions match. Evidently, we obtain a functorial assignment
	\[^\bullet(-)\colon G\Topu_*\to \Gamma_\II[G\Topu_*].\]
	
	Composing with the Segal machine functor, given a choice of compatible $G$-universe $U$, we obtain a functor
	\[G\Topu_* \overset{^\bullet(-)}{\longrightarrow} \Gamma_\II[G\Topu_*] \overset{\hat\sS_\II^{G,U}}{\longrightarrow} \Sp^{G,U}\]
	which we wish to compare to the suspension spectrum functor.
\end{Constr}

\begin{Constr}
	Let $Y$ be a based $G$-space. Define a contravariant $G\Top_*$-functor $Y^\bullet\colon\Gamma_\II\to G\Topu_{\textup{wp}}$ by setting $Y^\bullet(\textbf{n}_+^\alpha)= F(\textbf{n}_+^\alpha,Y)$ on objects with induced maps given by precomposition. This clearly defines a $G\Top_*$-functor. Note that we may identify $F(\textbf{n}_+^\alpha,Y)$ with $(Y^n)^\alpha$.
\end{Constr}

The following Lemma generalises Lemma 6.8 in \cite{GMMO19b} and uses essentially the same proof.

\begin{Lem}\label{lem:GMMO-6.8}
	Let $X$ and $Y$ be based $G$-spaces. Then there is a $G$-homeomorphism
	\[Y^\bullet \otimes_{\Gamma_\II} \text{}^\bullet X \cong Y\land X.\]
\end{Lem}

\begin{proof}
	The tensor product of functors is given by the quotient of the wedge of $G$-spaces indexed on finite $G$-sets $\textbf{n}_+^\alpha$ in $\Gamma_\II$ of the form
	\[F(\textbf{n}_+^\alpha,Y)\land (\textbf{n}_+^\alpha\land X).\tag{$*$}\]
	Note that the inclusions $\iota_j\colon \textbf{1}_+\to \textbf{n}_+^\alpha$ induce $G$-maps $\iota_j^*\colon F(\textbf{n}_+^\alpha,Y)\to F(\textbf{1}_+,Y)\cong Y$ and $\iota_{j*}\colon X\cong \textbf{1}_+\land X\to \textbf{n}_+^\alpha\land X$. For $\phi\in F(\textbf{n}_+^\alpha,Y)$ and $x\in X$ we obtain an identification of $[\phi,\iota_{j*}(x)]$ and $[\iota_j^*(\phi),x]$ in the quotient of the coequalizer. Hence, each term $(*)$ gets identified with $Y\land X$ after passing to the quotient showing the claim.
\end{proof}

Again, the following Lemma generalises Theorem 6.10 in \cite{GMMO19b} to work with $\Gamma_\II$ and uses essentially the same proof idea.

\begin{Lem}\label{lem:GMMO-6.10}
	Let $X$ and $Y$ be based $G$-spaces. Then there is a weak $G$-equivalence
	\[Y\land X \simeq B(Y^\bullet,\Gamma_\II,\text{}^\bullet X).\] 
\end{Lem}

\begin{proof}
	There is a natural inclusion of $Y\land X\to B_0(Y^\bullet,\Gamma_\II,\text{ }^\bullet X)=\bigvee_{\textbf{n}_+\in\Gamma_\II} F(\textbf{n}_+^\alpha, Y)\land \textbf{n}_+^\alpha \land X$ which after geometric realization induces a natural $G$-map 
	\[\eta\colon Y\land X \to B(Y^\bullet,\Gamma_\II,\text{}^\bullet X).\] Using Lemma 3.7 in \cite{MMO25} there is a natural $G$-map $B(Y^\bullet,\Gamma_\II,\text{ }^\bullet X) \to Y^\bullet \otimes_{\Gamma_\II} \text{}^\bullet X$ and after applying Lemma \ref{lem:GMMO-6.8} this yields a $G$-map
	\[\zeta\colon  B(Y^\bullet,\Gamma_\II,\text{}^\bullet X) \to Y\land X.\]
	Clearly, $\zeta\circ\eta=\id$ so it remains to check that $\eta\circ\zeta\simeq \id$. Considering $Y\land X$ as a constant simplicial $G$-space providing a homotopy $\eta_*\circ\zeta_*\simeq \id_{Z_*}$ with
	\[Z_*:=B_*(Y^\bullet,\Gamma_\II,\text{}^\bullet X)\cong B_*(Y^\bullet,\Gamma_\II,\Gamma_\II(\textbf{1}_+,-)\land X)\]
	is equivalent to providing an \textit{extra degeneracy} map on $Z_*$ (cf. \cite{Shu09}, proof of Lemma 9.9). The isomorphism above follows from the fact that $\textbf{n}_+^\alpha\land X \cong \Gamma(\textbf{1}_+,\textbf{n}_+^\alpha)\land X$. The extra degeneracy map is induced by the map $\id_{\textbf{1}_+}\colon \textbf{1}_+\to \textbf{1}_+$ in $\Gamma_\II$. Note that $Z_q$ is the wedge indexed over $q+1$-tuples $(a_q,\dots,a_0)$ of elements in $\Gamma_\II$ where in each wedge summand the elements are of the form $(f,\phi_q,\dots,\phi_1,\psi,x)$ for $f\colon a_q\to Y$, $\phi_i\colon a_{i-1}\to a_{i}$, $\psi\colon \textbf{1}_+\to a_0$ and $x\in X$. Define
	\[s_{-1}\colon Z_q\to Z_{q+1} \ , \ (f,\phi_q,\dots,\phi_0,\psi,x)\mapsto (f,\phi_q,\dots,\phi_1,\psi,\id_{\textbf{1}_+},x).\]
	One then needs to check that for all $q\ge 0$ it holds that $d_0s_{-1}=\id_{Z_q}$, $d_{i+1}s_{-1}=s_{-1}d_i$ and $s_{j+1}s_{-1}=s_{-1}s_j$, but this is a routine verification.
\end{proof}

Using this we can now compare suspension $G$-spectra with our Segal machine output. This works in parallel to \cite{GMMO19b}, Theorem 6.7.

\begin{Thm}
	For based $G$-spaces $X$ there is a weak equivalence of orthogonal $G$-spectra indexed on a $G$-universe $U$
	\[\mu\colon \Sigma_U^\infty X\to \sS_\II^{G,U}\text{}^\bullet X.\]
\end{Thm}

Note that a weak equivalence of orthogonal $G$-spectra (as defined in \cite{MM02}, Definition III.3.2) is the same thing as a level $G$-equivalence (\cite{MM02}, Lemma III.3.3 and Theorem III.3.4).

\begin{proof}
	Setting $Y=S^0$ in Lemma \ref{lem:GMMO-6.10} yields a $G$-map
	\[\eta\colon X\cong S^0\land X \to B((S^0)^\bullet,\Gamma_\II,\text{}^\bullet X)=(\sS_\II^{G,U}\text{}^\bullet X)_0.\]
	Let $\mu$ be the adjoint under the adjunction between $\Sigma_U^\infty$ and $\Omega_U^\infty$. By Lemma \ref{lem:GMMO-6.10} this is a levelwise $G$-homotopy equivalence.
\end{proof}

\subsection{Eilenberg-MacLane spectra}\label{sec:Eilenberg-Mac-Lane}

\begin{Def}
	An \textit{abelian $G$-group} is an abelian group $(A,+,0)$ together with a left $G$-action on $A$ such that $g(a+b)=ga+gb$ and $g0=0$ for all $g\in G, a\in A$.
\end{Def}

\begin{Ex}\label{ex:EMLS-counterex-1}
	Any abelian $G$-group $M$ has all transfer maps. Let $H/K$ be some $H$-orbit with a choice of coset representatives $r_1,\dots,r_n$ with $n=[H:K]$. Define the internal transfer map $t_K^H\colon M^K\to M^H$ by setting
	\[t_K^H(m):=r_1m+\dots+r_nm.\]
	Let $\rho\colon H\to \Sigma_n$ be the permutation representation so that for $h\in H$ it holds that $hr_i=r_{\rho(h)(i)}k_i$ for some $k_i\in K$. For some $m\in M^K$ we thus have that $t_K^H(m)$ is indeed an $H$-fixed point of $M$ $h(r_1m+\dots+r_nm) = r_{\rho(h)(1)}k_1m+\dots +r_{\rho(h)(n)}k_nm = r_1m+\dots+r_nm$ using that the $G$-action is strictly associative, strictly commutative and that $m\in M^K$. In particular, $M$ induces a \textit{complete} Mackey functor $\underline{M}^{\textup{fix}}$ given by $\underline{M}^{\textup{fix}}(G/H)=M^H$ with restriction maps being inclusions of fixed points and above transfer maps.
\end{Ex}

\begin{Def}
	A complete Mackey functor $\underline{M}$ is called a \textit{fixed point Mackey functor} if there exists an abelian $G$-group $N$ in $G\Top_*$ such that $\underline{N}^{\textup{fix}}\cong \underline{M}$.
\end{Def}

For fixed point Mackey functors we can define associated $\Gamma$-$G$-spaces. This works completely analogously to \cite{Seg74}.

\begin{Constr}\label{constr:Eilenberg-Mac-Lane-Gamma}
	Let $M$ be an abelian $G$-group in $G\Top_{\textup{wp}}$. Define a $\Gamma$-$G$-space $X_M:\Gamma\to G\Top_*$ as follows: On objects set $X_M(\textbf{n}_+)=M^n$ with diagonal $G$-action. On morphisms send a map $\phi\colon\textbf{m}_+\to \textbf{n}_+$ in $\Gamma$ to the map $X_M(\phi)$ which is given by 
	\[M^n\ni(a_1,\dots,a_n) \mapsto (b_1,\dots,b_m)\in M^m \textup{ with } b_i=\sum_{j\in \phi^{-1}(i)}a_j.\]
	If $\phi^{-1}(i)=\emptyset$ then we set $b_i=0$. Since $M$ is well pointed the cofibration condition of $\Gamma$-$G$-spaces is easily verified. The $\Sigma_n$-action on $M^n$ for $\sigma\in \Sigma_n$ is given by the induced map $M^n\to M^n$ via $(x_1,\dots,x_n)\mapsto (x_{\sigma^{-1}(1)},\dots,x_{\sigma^{-1}(n)})$. The Segal maps $\delta\colon M^n\to M^n$ are induced by the maps $\delta_i\colon M^n\to M$ which map $(x_1,\dots,x_n)$ to $x_i$ and are thus identities on $M^n$. In particular, $X_M$ is an $\ff_\bullet^\CC$-special $\Gamma$-$G$-space and by prolongation $\pp_\TT^\CC X_M$ is an $\CC$-special $\Gamma_G$-$G$-space.
\end{Constr}

In the complete case the Segal machine $\sS_\CC^{G,U}$ is the same as the machine constructed in \cite{MMO25} which in turn is equivalent to the machine constructed in \cite{GMMO19b}. There it was shown that the $\Gamma$-$G$-space $X_M$ yields an Eilenberg-MacLane $G$-spectrum implying the following proposition.

\begin{Prop}
	The associated genuine $G$-spectrum $\sS^{G,U_c}_\CC\pp_\TT^\CC X_M$ is an Eilenberg-MacLane $G$-spectrum. 
\end{Prop}

\begin{Rem}
	There is no evident way to nuance a set of incomplete transfers into a space $M$. Strict commutativity already enforces the existence of all transfer maps. The problem is intrinsic to the category of $G$-spaces. A possible solution is to consider $G$-coefficient systems instead of $G$-spaces which are equivalent. In \cite{HH16} it is shown that the category of $G$-coefficient systems is able to host incomplete monoid structures. Hence one would need to define a Segal machine where the input are not $\Gamma$-$G$-spaces but instead \textit{$\Gamma$-$G$-coefficient systems}, that is, functors $\Gamma\to\textsf{Coef}(\Top_*)$ with an appropriate Segal condition. The output category would then be a category of incomplete spectral Mackey functors. We intend to return to this issue in a subsequent paper.
	
	 It should be noted that \cite{BO15} define a machine taking Mackey functors of permutative categories to spectral Mackey functors, hence allowing to construct genuine Eilenberg-MacLane spectra from arbitrary complete Mackey functors. 
\end{Rem}


\section{\texorpdfstring{$K$}{K}-theory of normed permutative categories}\label{sec:K-theory}

In this section we will generalise algebraic $K$-theory of permutative categories as originally defined by \cite{Seg74} and \cite{May78} to the equivariant setting. Given a $G$-indexing system $\II$ and an $\II$-normed permutative category $\A$ in the sense of \cite{Rub25} internal to $G\Top_*$ we will define a $\Gamma_\II$-$G$-category $\overline{\A}$ which induces an $\II$-special $\Gamma_\II$-$G$-space $B\overline{\A}$. The associated positive connected $\Omega$-$G$-spectrum $\sS_\II^{G,U}|\overline{\A}|$ then gives rise to the \textit{$K$-theory $G$-spectrum} of $\A$.

\subsection{Normed permutative categories.}\label{sec:I-NPCs}

We define $\II$-normed permutative $G$-categories in the sense of \cite{Rub25} in the internal setting. Let $\V$ be a category with all finite limits and colimits. We follow notations and conventions of \textit{$\V$-internal categories} or \textit{categories internal to $\V$} in \cite{GMMO18}. The following definition is an equivariant generalisation of \cite{May78}, Definition 1, stated for example in \cite{GM17}, \S 4.1.

\begin{Def}
	A \textit{$G\V$-internal permutative $G$-category} is a $G\V$-internal category $\A$ together with a $G\V$-internal functor $\oplus\colon\A\times\A\to\A$, an object $0\in \A$, and a $G\V$-internal natural isomorphism $\beta\colon \oplus\Rightarrow\oplus$ with $\beta_{A,A'}\colon A\oplus A'\to A'\oplus A$ such that for all $A,A',A''\in\ob\A$
	\begin{enumerate}[(P1)]
		\item $(A\oplus A')\oplus A''=A\oplus (A'\oplus A'')$ (strict associativity) and $A \oplus 0=A=0\oplus A$ (strict unitality) as well as 
		\item the following diagrams commute:
		\begin{center}
			\begin{tikzcd}[column sep = 2mm]
				A\oplus 0 \arrow[rr, "\beta"] \arrow[rd, Rightarrow, no head] &   & 0\oplus A \arrow[ld, Rightarrow, no head] \\
                                                             & A &                     
			\end{tikzcd}
			\hspace{1cm}
			\begin{tikzcd}[column sep = 0mm]
				A\oplus A' \arrow[rd, "\beta_{A,A'}"'] \arrow[rr, Rightarrow, no head] &                               & A\oplus A' \\ & A'\oplus A \arrow[ru, "\beta_{A',A}"'] &          
			\end{tikzcd}
		\end{center}
		\begin{center}
			\begin{tikzcd}[column sep = 0mm]
				A\oplus A'\oplus A'' \arrow[rd, "{\id_A\oplus \beta_{A',A''}}"'] \arrow[rr, "{\beta_{A\oplus A',A''}}"] &                                                            & A''\oplus A\oplus A' \\
				& A\oplus A''\oplus A' \arrow[ru, "{\beta_{A,A''}\oplus \id_{A'}}"'] &                  
			\end{tikzcd}
		\end{center}
	\end{enumerate}
	A map between $G\V$-internal permutative categories is a $G\V$-internal functor $F\colon\A\to\B$ which is strict symmetric monoidal, that is, $F(0_\A)=0_\B$, $F\circ\oplus = \oplus \circ(F\times F)$, and $F(\beta_{A,A'})=\beta_{F(A),F(A')}$. This makes $G\V$-internal permutative $G$-categories into a category $G\textsf{PC}(\V)$.
\end{Def}

\begin{Rem}
	Note that in the case of $\V=G\Cat$, the category of small categories, by \cite{GM17}, Proposition 4.2 the category $G\textsf{PC}(\Cat)$ is equivalent to the category of $\mathscr{P}$-algebras where $\mathscr{P}$ denotes the categorial Barratt-Eccles operad.
\end{Rem}

Define the \textit{standard $n$-fold sum} in a $G\V$-internal permutative $G$-category $\A$ inductively by setting $\bigoplus_0()=0$, $\bigoplus_1(A)=A$, and $\bigoplus_{n+1}(A_1,\dots,A_{n+1})=\bigoplus_n(A_1,\dots,A_n)\oplus A_{n+1}$ as in \cite{Rub25}, Definition 2.1. Let $T\cong \textbf{n}^\sigma$ be a finite $H$-set in some indexing system $\II$ with permutation representation $\sigma\colon H\to \Sigma_n$. Any $G\V$-internal category $\C$ naturally forgets to an $H\V$-internal category for any subgroup $H\subset G$.

\begin{Def}
	 Let $\C$ be a $G\V$-internal category. Define an $H\V$-internal category $\C^{\times T}$. This is the $n$-fold cartesian power of the $H\V$-internal category $\C$ with twisted action given on objects by
\[h(C_1,\dots,C_n)=(hC_{\sigma(h)^{-1}(1)},\dots,hC_{\sigma(h)^{-1}(n)})\]
and similarly on morphisms (cf. \cite{Rub25}, Definition 2.2). This is again an $H\V$-internal category with $\ob (\C^{\times T})=(\ob \C)^{\times T}$ and $\mor (\C^{\times T})=(\mor \C)^{\times T}$. We may use both the notations $\C^{\times T}$ and $(\C^{\times n})^\sigma$ interchangeably. Let $\C$ be a $G\V$-internal category. An \textit{external $T$-norm} is an $H\V$-internal functor
	\[\bar{t}_T\colon \C^{\times T}\to \C.\]
\end{Def}

\begin{Rem}
	Similarly to how external transfers of spaces give rise to internal transfers also external $H/K$-norms give rise to internal norms which are $\V$-functors $t_H^K\colon \C^K\to \C^H$. The construction is virtually the same as for $G$-spaces as done in Lemma \ref{lem:external->internal-transfer}.
\end{Rem}

\begin{Def}[\cite{Rub25}, Definition 2.3]
	A \textit{set of exponents} is an indexed set $\II=(T_i)_{i\in I}$ such that for each $i\in I$, the element $T_i$ is a finite ordered $H_i$-set for some subgroup $H_i\subset G$, depending on $i\in I$.
\end{Def}

This is the most general starting point for the following definition. The case of most interest is when $\II$ is an indexing system with a choice of ordering. We can now define the internal version of $\II$-normed permutative categories defined originally in \cite{Rub25}, Definition 2.3, remark after Example 7.13.

\begin{Def}\label{def:I-NPC}
	Let $\II$ be a set of exponents. An \textit{$G\V$-internal $\II$-normed permutative category} is a $G\V$-internal permutative category $\A$ together with a set of \textit{external $T$-norms} $\bigoplus_T\colon \A^{\times T}\to \A$ for every finite $H$-set $T$ in $\II$ and \textit{$T$-untwistors} which are (not necessarily $H$-equivariant) $\V$-internal natural isomorphisms 
	\[v_T\colon \textstyle\bigoplus_T(A_1,\dots,A_{|T|})\to \textstyle\bigoplus_{|T|}(A_1,\dots,A_{|T|})\] 
	for every finite $H$-set $T$ in $\II$ and tuples $(A_1,\dots,A_n)$ in $\A^{\times T}$ such that
	\begin{enumerate}[(P1)]\setcounter{enumi}{2}
		\vspace{-1mm}\item for every $H\subset G$ and finite $H$-set $T\cong\textbf{n}^\sigma$ with $\sigma\colon H\to\Sigma_n$ in $\II(H)$ and $h\in H$ the following \textit{twisted equivariance diagram} commutes
		\begin{center}
			\begin{tikzcd}
				{h\bigoplus_T(A_1,\dots,A_{n})} \arrow[dd, "hv_{T}"'] \arrow[rr, "\id"] &  & {\bigoplus_T(hA_{\sigma(h)^{-1}(1)},\dots,A_{\sigma(h)^{-1}(n)})} \arrow[d, "v_T"]                \\
				&  & {\bigoplus_{n}(hA_{\sigma(h)^{-1}(1)},\dots,A_{\sigma(h)^{-1}(n)})} \arrow[d, "\sigma(h)^{-1}"] \\
				{h\bigoplus_{n}(A_1,\dots,A_{n})} \arrow[rr, "\id"'] &  & {\bigoplus_{n}(hA_1,\dots,hA_{n})}                                                            
			\end{tikzcd}
		\end{center}
		where $\sigma(h)^{-1}$ denotes the canonical isomorphism for the $G\V$-internal permutative category $\A$ which permutes the factors of $\bigoplus_{|T|}$ by $\sigma(h)^{-1}$, and
		\item for any object $C$ in $\C$ and finite $H$-set $T$ the untwistor 
		\[v_T\colon \textstyle\bigoplus_T(0,\dots,C,0,\dots,0)\to \textstyle\bigoplus_{|T|}(0,\dots,C,0,\dots,0)=C\]
		with $C$ in some spot $1\le i\le |T|$ is the identity map and similarly the untwistior $v_T\colon \textstyle\bigoplus_T(0,\dots,0) \to \textstyle\bigoplus_{|T|}(0,\dots,0)=0$ is also the identity map. For $|T|=1$, the first map is understood to be $\id_C=v_T\colon \bigoplus_T(C)\to C$, and for $|T|=0$, the second map is understood to be $\id_*=v_T\colon \bigoplus_T()\to 0$.
	\end{enumerate}
	
	A map of $G\V$-internal $\II$-normed permutative $G$-categories is a map $F\colon \A\to \B$ of underlying $G\V$-internal permutative $G$-categories such that for any finite $G$-set $T$ in $\II$ and $(A_1,\dots,A_n)$ in $\A^{\times T}$
\[F(\textstyle\bigoplus_T^\A(A_1,\dots,A_n)) \cong  \textstyle\bigoplus_T^\B(F(A_1),\dots,F(A_n)).\]

	Denote the category of $G\V$-internal $\II$-normed permutative $G$-categories by $G\textsf{NPC}^\II(\V)$.
\end{Def}

\begin{Rem}
	Note that in the case of $\V=G\Cat$, the category of small categories, by \cite{Rub25}, Theorem 7.11 and Example 7.13 the category $G\textsf{NPC}^\II(\Cat)$ is equivalent to the category of $\mathscr{P_\II}$-algebras where $\mathscr{P}_\II$ denotes the $\II$-permutativity operad as defined in \cite{Rub21b}, Definition 7.3. Since small $G$-categories are just $\Set$-internal categories they are also $\Top$-internal with discrete topology. Hence, $\mathscr{P}_\II$-algebras deliver a large class of examples.
\end{Rem}

By \cite{GMMO18} Remark 1.6 all of the above concepts also work in the based case.

\begin{Rem}
	Note that if $\II=\TT$ is the trivial indexing system, then $G\textsf{NPC}^\TT(\V)=G\textsf{PC}(\V)$. As before, $G\textsf{NPC}^\II(\V)$ naturally admits the structure of a $2$-category with respective notions of $2$-morphisms similarly to \cite{Rub25}, Definition 2.6.
\end{Rem}

\begin{Ex}\label{ex:I-NPC-from-abelian-G-group}
	Let $A$ be an abelian $G$-group. Consider $A$ as a discrete $G\Top_*$-internal category with $\ob A=A$ with base point being the unit $0$ and $\mor A\cong \ob A$ with only identity morphisms. First of all, this is permutative with respect to addition in $A$ and with twist being the identity. Furthermore, this is a $\CC$-NPC with external norms defined as follows: Define $+_T\colon A^{\times T}\to A$ by setting $t_T(x_1,\dots,x_n)=x_1+\dots+x_n$ and set $v_T=\id$. Since $A$ is abelian this map is $H$-equivariant and in this case the untwistors are also $H$-equivariant.
\end{Ex}

\subsection{\texorpdfstring{$\Gamma_\II$}{GammaI}-\texorpdfstring{$G$}{G}-categories}\label{sec:Gamma_I-G-cats}

\begin{Rem}\label{rem:V-enrichment-of-Cat(V)}
	If $\V$ is a finite limit complete cartesian closed category then so is the category of $\V$-internal categories $\Cat(\V)$ by \cite{Joh02}, remark after Lemma B.2.3.15. Hence, there is a self-enrichment of $\Cat(\V)$. Using the adjunction $i\colon \V\leftrightarrows \Cat(\V):\ob$ from Remark 1.8 in \cite{GMMO18} we can forget from an $\Cat(\V)$-enrichment of $\Cat(\V)$ to an $\V$-enrichment.
\end{Rem}

Henceforward we can consider $\Cat(G\Top_*)$ as a $G\Top_*$-enriched category. Denote the space of $G\Top_*$-internal functors $F\colon \C\to \D$ with conjugation $G$-action on object and morphism maps $\ob F\colon \ob\C\to\ob\D$ and $\mor F\colon \mor\C\to\mor\D$ by $\Fun(\C,\D)$. In the $G\Top_*$-enrichment of $\Cat(G\Top_*)$ the enriched functor $G$-spaces are precisely described by $\Fun(\C,\D)$. Denote by $\Fun_G(\C,\D)$ the subspace of $G$-equivariant internal functors.

\begin{Def}
	A \textit{$\Gamma_\II$-$G$-category} is a $G\Top_*$-enriched functor $\mathscr{X}\colon\Gamma_\II\to \Cat(G\Top_*)$. Unravelling definitions, a morphism of $\Gamma_\II$-$G$-categories is a $G\Top_*$-enriched natural transformation. Denote the category of $\Gamma_\II$-$G$-categories by $\Gamma_\II[\Cat(G\Top_*)]$.
\end{Def}

Let $B\colon\Cat(G\Top_*)\to G\Top_*$ be the realization functor. Then any $\Gamma_\II$-$G$-category $\mathscr{X}$ gives rise to a functor $B\circ\mathscr{X}\colon\Gamma_\II\to G\Top_*$. A priori this needs not to fulfil the additional cofibrancy condition endowed on $\Gamma_\II$-$G$-spaces. There are two ways around this. Either cofibrantly replace via bearding of functors as in Proposition 2.12 in \cite{MMO25}. Or we assume that the condition after realization must be fulfilled. Via bearding this results in no loss of generality.

\begin{Def}
	A $\Gamma_\II$-$G$-category $\mathscr{X}$ is \textit{nice} if $B\mathscr{X}$ is a $\Gamma_\II$-$G$-space.
\end{Def}

We are interested in $\Gamma_\II$-$G$-categories fulfilling a \textit{`specialness'} condition similar to $\Gamma_\II$-$G$-spaces which passes via $B$ to the usual specialness condition on $\Gamma_\II$-$G$-spaces. To do so recall that the Segal maps were simply defined via passing through adjunctions starting from the Kronecker delta map $T_+\land T_+\to \textbf{1}_+$. Since $\Cat(G\Top_*)$ is cartesian closed (see Remark \ref{rem:V-enrichment-of-Cat(V)}) we can do the same thing to obtain Segal $G$-functors
\[\delta\colon \mathscr{X}(T_+)\to \mathscr{X}(\textbf{1}_+)^{\times T}\]
for any finite $G$-set $T$ in $\II$.

\begin{Def}
	Let $\JJ\subset \II$ be an indexing system. A $\Gamma_\II$-$G$-category $\mathscr{X}$ is called \textit{$\JJ$-special} if the Segal $G$-functor $\delta\colon \mathscr{X}(T_+)\to \mathscr{X}(\textbf{1}_+)^{\times T}$ is an equivalence of internal $G$-categories for all $T_+\in \Gamma_\II$
\end{Def}

\begin{Ex}
	Let $A$ be an abelian $G$-group. By Example \ref{ex:I-NPC-from-abelian-G-group} $A$ is an $\II$-NPC. Define a $\Gamma_\CC$-$G$-category by setting $A_n=A^n$ on objects and mapping a morphism $\phi\colon \textbf{n}_+^\alpha\to \textbf{m}_+^\beta$ to $\phi_*\colon A^n\to A^m$ which maps $(a_1,\dots,a_n)$ to $(b_1,\dots,b_m)$ with $b_i=\sum_{j\in \phi^{-1}(i)}a_j$. By Construction \ref{constr:Eilenberg-Mac-Lane-Gamma} this is $\CC$-special since the Segal maps are identity maps.
\end{Ex}

Of course we can analogously define an $\ff_\bullet^\II$-special condition on an $\Gamma$-$G$-category by requiring that the Segal $G$-functors $\delta$ are $\Lambda_\alpha$-equivalences of $G\times\Sigma_n$-categories. Here the $\Sigma_n$-action on $\mathscr{X}(\textbf{n}_+)$ is analogously induced by maps $\sigma\colon \textbf{n}_+\to\textbf{n}_+$ in $\Gamma$.

\begin{Rem}\label{rem:eequiv-of-cats->homotopy-equiv}
	Note that any $G\Top_*$-internal natural transformation $\tau\colon F\Rightarrow F'$ gives rise to a homotopy $BF\simeq BF'$. The proof works just as in the classical case (cf. \cite{Ric20}, Theorem 11.2.4). In particular, a $G\Top_*$-internal equivalence of $G\Top_*$-internal categories $\C\simeq \D$ gives rise to a $G$-homotopy equivalence of $G$-spaces $B\C \simeq B\D$.
\end{Rem}
 
 \begin{Prop}\label{prop:Gamma-cat->Gamma-space}
 	Realizing a nice $\II$-special $\Gamma_\II$-$G$-category $\mathscr X$ yields an $\II$-special $\Gamma_\II$-$G$-space $B\mathscr{X}$.
 \end{Prop}
 
 \begin{proof}
 	Since $\mathscr{X}$ is nice $B\circ \mathscr{X}$ is a $\Gamma_\II$-$G$-space. By Remark \ref{rem:eequiv-of-cats->homotopy-equiv} the realization of $G\Top_*$-internal equivalences of categories yields $G$-homotopy equivalences of $G$-spaces so that $B\mathscr{X}(\delta)$ is a $G$-equivalence of based $G$-spaces for all $\delta\colon \mathscr{X}(T_+)\to \mathscr{X}_1^{\times T}$ with admissible $T$. In particular, $B\circ\mathscr{X}$ is $\II$-special.
 \end{proof}
 
 \begin{Thm}[Theorem B]\label{thm:theorem-B}
	Let $\II$ be a disk-like indexing system and let $\A$ be a $G\Top_*$-internal $\II$-normed permutative category. Then there exists an $\II$-special $\Gamma_\II$-$G$-category $\overline{\A}\colon \Gamma_\II\to \Cat(G\Top_*)$. If $\overline{\A}$ is nice this gives rise to the \textit{equivariant $K$-theory spectrum of the $\II$-normed permutative category $\A$}
	\[K^{G,U}_\II \A = \sS_\II^{G,U}(B\overline{\A}).\]
\end{Thm}

By construction, if $G=*$, $U=\rr^\infty$, and $\II=\TT$ then we recover the non-equivariant $K$-theory of a permutative category originally defined by \cite{Seg74} and \cite{May78}, though they did not use orthogonal spectra.

\subsection{Construction of the associated \texorpdfstring{$\Gamma_\II$}{Gamma-I}-\texorpdfstring{$G$}{G}-category}\label{sec:construction-K-theory}

Given a $G\Top_*$-internal $\II$-normed permutative category $\A$ we will first construct a $G\Top_*$-functor $\overline{\A}\colon\Gamma\to \Cat(G\Top_*)$ and then prolong to $\Gamma_\II$. The first step is done analogously to the non-equivariant case described in \cite{May78}. For convenience, we will follow the slightly simplified construction in \cite{EM06} \S 4. The prolongation step follows an idea originally due to Shimakawa \cite{Shi89}.

\begin{Constr}\label{constr:perm-G-cat-pt.1}
	Let $\A$ be a $G\Top_*$-internal permutative category. For any finite based set $\textbf{n}_+$ define a $G\Top_*$-internal category $\overline{\A}_n$ as follows: The objects are systems $\langle A_s,a_{s,t}\rangle$ indexed on subsets $s\subset \textbf{n}$ and pairs of subsets $s,t\subset \textbf{n}$ such that $s\cap t=\emptyset$. The morphisms are systems $\langle \alpha_s\rangle$ indexed on  subsets $s\subset \textbf{n}$. The $A_s$ are objects of $\A$ and $\A_\emptyset=*$ is the base point. The $a_{s,t}$ are $G$-isomorphisms
	\[a_{s,t}\colon A_{s}\oplus A_{t}\to A_{s\cup t}\]
	subject to the condition that $a_{s,\emptyset}=\id_{A_s}$ and that
	\begin{center}
		\begin{tikzcd}
			{A_s\oplus A_t} \arrow[rr, "{a_{s,t}}"] \arrow[d, "\textup{tw}"'] &  & A_{s\cup t} \arrow[d, Rightarrow, no head] \\
			{A_t\oplus A_s} \arrow[rr, "{a_{t,s}}"']                &  & A_{t\cup s}                               
		\end{tikzcd}
		and 
		\begin{tikzcd}
			A_r\oplus A_s\oplus A_t \arrow[d, "{a_{r,s}\oplus\id_{A_t}}"'] \arrow[rr, "{\id_{A_r}\oplus a_{s,t}}"] &  & A_r\oplus A_{s\cup t} \arrow[d, "{a_{r,s\cup t}}"] \\
			A_{r\cup s}\oplus A_t \arrow[rr, "{a_{r\cup s,t}}"']                                                   &  & A_{r\cup s\cup t}                                 
		\end{tikzcd}
	\end{center}
	are commutative diagrams for all pairwise disjoint subsets $r,s,t\subset\textbf{n}$. A morphism $\langle\alpha_s\rangle$ between systems $\langle A_s,a_{s,t}\rangle$ and $\langle A_s',a_{s,t}'\rangle$ consists of morphisms $\alpha_s\colon A_s\to A'_s$ in $\A$ such that the following diagram commutes
	\begin{center}
		\begin{tikzcd}
			A_s\oplus A_t \arrow[d, "\alpha_s\oplus \alpha_t"'] \arrow[rr, "{a_{s,t}}"] &  & A_{s\cup t} \arrow[d, "\alpha_{s\cup t}"] \\
			A_s'\oplus A_t' \arrow[rr, "{a_{s,t}'}"']                                   &  & A_{s\cup t}'                             
		\end{tikzcd}
	\end{center}
	Composition and identity are inherited from $\A$ so that $\overline{\A}_n$ is indeed a category. It is indeed a $G\Top_*$-internal category with object and morphisms $G$-spaces given by the subspace topology of the $G$-spaces
	\[\prod_s \ob \A \times \prod_{s,t}\mor\A \text{ and }\prod_{s}\mor\A\]
	Hence, the $G$-action is given diagonally by
	\[g\langle A_s,a_{s,t}\rangle = \langle gA_s,ga_{s,t}\rangle \text{ and } g\langle \alpha_s\rangle=\langle g\alpha_s\rangle.\]
	Note that this is actually a $G$-subspace since $\oplus$ is a $G$-functor. More precisely $gA_s\oplus gA_t=g(A_s\oplus A_t)$, so that all of the above diagrams parse with a $g$ in front of everything to fulfil the condition for $\langle gA_s,ga_{s,t}\rangle$ to be an object of $\overline{\A}_n$. The same holds for morphisms.
	
	The identity morphism maps an object $\langle A_s,a_{s,t} \rangle$ to the morphism $\langle \alpha_s\rangle$ which is given by $\alpha_s=\id_{A_s}$ for all $s$. This gives a continuous $G$-map $I\colon \ob\overline\A_n\to \mor\overline\A_n$ since 
	\[gI(\langle A_s,a_{s,t} \rangle)=g\langle \id_{A_s}\rangle=\langle g\id_{A_s}\rangle = \langle \id_{gA_s}\rangle =I(\langle gA_s,ga_{s,t} \rangle)\]
	noting that $ga_{s,t}\colon gA_{s\cup t}\to gA_{s}\oplus gA_t = g(A_s\oplus A_t)$ and that $I$ in $\A$ is continuous.
	
	The source and target maps also give continuous $G$-maps $S,T\colon \mor\overline\A_n\to \ob\overline\A_n$ which assign a morphism $\langle \alpha_s\rangle$ given by $\alpha_s\colon A_s\to A_s'$ to the object $\langle A_s,a_{s,t}\rangle$ or $\langle A_s',a_{s,t}'\rangle$. By the same remark this gives indeed a continuous $G$-map.
	
	Composition is induced by the composition in $\A$, i.e. we define $\langle \alpha_s\rangle \circ \langle \alpha_s'\rangle =\langle \alpha_s\circ\alpha_s'\rangle$. Since composition in $\A$ is a continuous $G$-map so is the composition map
	\[C\colon \overline\A_n(\langle A_s,a_{s,t}\rangle,\langle A_s',a_{s,t}'\rangle)\times \overline\A_n(\langle A_s',a_{s,t}'\rangle,\langle A_s'',a_{s,t}''\rangle)\to \overline\A_n(\langle A_s,a_{s,t}\rangle,\langle A_s'',a_{s,t}''\rangle)\]
	in $\overline{\A}_n$. It is also evident from this that unitality and associativity conditions hold.
\end{Constr}

\begin{Constr}\label{constr:perm-G-cat-pt.2}
	As before, let $\A$ be a $G\Top_*$-internal permutative category. Extend the assignment $\textbf{n}_+\mapsto \overline{\A}_n$ from Construction \ref{constr:perm-G-cat-pt.1} to a functor $\overline{\A}\colon \Gamma\to G\Cat(\Top_*)$ by mapping on objects $\textbf{n}_+$ to $\overline{\A}_n$ and for any based map $\phi\colon \textbf{m}_+\to\textbf{n}_+$ in $\Gamma$ define a morphism $\overline{\A}(\phi)\colon \overline{\A}_m \to \overline{\A}_n$ by mapping an object $\langle A_s,a_{s,t}\rangle$ to the object $\langle B_u,b_{u,v}\rangle$, where 
	\[B_u=A_{\phi^{-1}(u)} \text{ and } b_{u,v}=a_{\phi^{-1}(u), \phi^{-1}(v)}.\]
	This is well defined since $\phi$ is based so that $\phi^{-1}(s)$ never contains the base point for $s\subset\textbf{n}\subset\textbf{n}_+$. Functoriality follows just as in the non-equivariant case, see \cite{May78}. Since $\Gamma$ is discrete with trivial $G$-action $\overline{\A}$ therefore defines a $G\Top_*$-functor.
\end{Constr}

\begin{Rem}
	Note that $\overline{\A}_1$ can be identified with $\A$. The objects of $\overline{\A}_1$ are systems $\langle A_s,a_{s,t} \rangle$ with $s$ running through subsets of $\textbf{1}=\{1\}$. Hence, either $A_\emptyset$ or $A_{\{1\}}$ are possible, but $A_\emptyset$ is just the unit of $\A$ so no additional data is given. The $a_{s,t}$ are always identities since either $s$ or $t$ are required to be empty by the assumption that $s\cap t=\emptyset$.
\end{Rem}

Any $\Gamma$-$G$-category $\mathscr{X}$ gives rise to an $\Sigma_n$-action on $\mathscr{X}_n$ using functoriality. More precisely, a permutation $\sigma\in\Sigma_n$ gives rise to a morphism $\sigma\colon\textbf{n}_+\to \textbf{n}_+$ in $\Gamma$ which induces a morphism $\mathscr{X}(\sigma)\colon \mathscr{X}_n\to \mathscr{X}_n$. For $\overline{\A}_n$ we have an explicit description of the $\Sigma_n$-action. The induced map $\mathscr{X}(\sigma)$ sends objects $\langle A_s,a_{s,t}\rangle$ to $\langle B_u,b_{u,v}\rangle$ where 
\[B_u=A_{\sigma^{-1}(u)} \text{ and } b_{u,v}=a_{\sigma^{-1}(u),\sigma^{-1}(v)}\] 
and which assigns morphisms $\langle \alpha_s\rangle$ to $\langle\beta_{u}\rangle$ with $\beta_u=\alpha_{\sigma^{-1}(u)}$. Hence, $\overline{\A}_n$ is a $G\times\Sigma_n$-category with $G$-action as described in Construction \ref{constr:perm-G-cat-pt.1} and $\Sigma_n$-action as described above.

Similarly to how we defined a twisted $G$-action on $G\times\Sigma_n$-spaces we can define twisted $G$-actions on $G\times\Sigma_n$-categories.

\begin{Def}
	Let $\C$ be a $G\times\Sigma_n$-category and let $\sigma\colon G\to\Sigma_n$ be a homomorphism. Denote by $\C^\sigma$ the $G$-category with same underlying objects and morphisms as in $\C$ but with $G$-action on objects $C$ and morphisms $f\colon C\to C'$ given by $g._\sigma C=(g,\sigma(h)).C$ and $g._\sigma f=(g,\sigma(g)).f$. We call $\C^\sigma$ the \textit{$\sigma$-twisted category of $\C$}.
\end{Def}

A $G\times\Sigma_n$-category $\C$ is in particular a $G$-category. To emphasise the action we will either denote objects and morphisms in $\C^\sigma$ by $C^\sigma$ and $f^\sigma$ and write the $G$-action as $gC^\sigma$ and $gf^\sigma$ or write objects as before with action $g._\sigma C$ and $g._\sigma C$.

\begin{Rem}
	Of course the above definition generalises for internal categories. If $\C$ is $\V$ internal, then $G\times\Sigma_n$-objects in $\V$ are $(G\times\Sigma_n)\V$-internal categories. Defining a new twisted action defines twisted $G$-objects internal to $\V$.
\end{Rem}

\begin{Constr}
	Unravelling this definition for $\overline\A_n$ we obtain a category $(\overline\A_n)^\sigma$ with same underlying objects and morphisms, but with $G$-action given by
	\[g._\sigma\langle A_s,a_{s,t}\rangle =\langle gA_{\sigma(g)^{-1}(s)},ga_{\sigma(g)^{-1}(s),\sigma(g)^{-1}(t)}\rangle \text{ and }g._\sigma \langle \alpha_s\rangle =\langle g\alpha_{\sigma(g)^{-1}(s)}\rangle.\]
	Note that this is well defined since if $s\cap t=\emptyset$ then $\sigma(h)^{-1}(s)\cap \sigma(h)^{-1}(t)=\emptyset$.
\end{Constr}

Using this we can prolong the previously define $\Gamma$-$G$-category $\overline{\A}$ to $\Gamma_\II$.

\begin{Def}
	Let $\A$ be a permutative $G$-category. Define 
	\[\overline{\A}\colon \Gamma_\II \to \Cat(G\Top_*)\] 
	by setting $\overline{\A}(\textbf{n}_+^\sigma):=(\overline{\A}_n)^\sigma$ for any finite based $G$-set $\textbf{n}_+^\sigma$ in $\Gamma_\II$ and by setting for any based map $\phi$ in $\Gamma_\II$ the induced map $\overline{\A}(\phi)$ to be the underling map $\overline{\A}_n\to \overline{\A}_m$ on objects and morphisms. Here, $G$-acts on $\phi$ and on $\overline{\A}(\phi)$ via conjugation. It is not hard to see that using this definition $\overline{\A}$ is indeed a $G\Top_*$-functor
\end{Def}

\subsection{Proving the Segal condition}\label{sec:proof-K-theory}

To prove that given a $G\Top_*$-internal $\II$-normed permutative $G$-category that the associated $\Gamma_\II$-$G$-category $\overline{\A}$ is $\II$-special we need to provide for any finite $G$-set $T=\textbf{n}^\sigma$ with $\sigma\colon H\to\Sigma_n$ in $\II(G)$ a $G\Top_*$-internal equivalence of categories
\[\delta\colon (\overline{\A}_n)^\sigma = \overline{\A}(T_+)\to \A^{\times T} = (\A^{\times n})^\sigma.\]
Unraveling the construction of the induced maps we see that the underlying untwisted Segal functors $\delta\colon \overline{\A}_n\to \A^{\times n}$ map on objects $\langle A_s,a_{s,t} \rangle$ to the tuple with $i$-th entry $A_i=A_{\{i\}}$.

\begin{Rem}
	Note that most of the machinery in Chapter \ref{sec:equivariant-Gamma} can be translated to the categorical case. We are basically showing that the $\Gamma$-$G$-category $\overline{\A}$ is $\ff_\bullet^\II$-special, but we have not defined this in the categorical case. Compare our definition of prolongation here with the discussion around the fourth comparison theorem, Theorem \ref{thm:4th-comparison-theorem}.
\end{Rem}

\begin{Thm}\label{thm:I-NPC->I-special-Gamma-cat}
	Let $\A$ be an $\II$-permutative category internal to $G\Top_*$. Then for any $\textbf{n}^\sigma\in\II(G)$ the Segal $G$-functor $\delta\colon (\overline{\A}_n)^\sigma\to (\A^{\times n})^\sigma$ specifies a $G\Top_*$-internal equivalence of categories. In particular, $\overline{\A}$ is an $\II$-special $\Gamma_\II$-$G$-space.
\end{Thm}

\begin{proof}
	Write $T=\textbf{n}^\sigma$. 
	
	\underline{\smash{Step 1:}} Define a $G\Top_*$-internal functor $\nu\colon (\A^{\times n})^\alpha \to (\overline{\A}_n)^\alpha$ on objects as follows: It maps twisted tuples $(A_1,\dots,A_n)$ in $(\A^{\times n})^\alpha$ to systems $\langle A_s,a_{s,t}\rangle$ where for any subset $s\subset\textbf{n}$ we define $A_s$ to be the twisted sum
	\[A_s=\textstyle\bigoplus_T(A_{\kappa(s)(1)},\dots, A_{\kappa(s)(n)})\]
	where
	\[\kappa(s)(i)=\begin{cases} \{i\} &, \text{ if } i\in s \\ \emptyset &,\text{ if } i\notin s. \end{cases}\]
	We will write $A_s=\bigoplus_T(A_{\kappa(s)(i)})$ for short. The maps $a_{s,t}$ are isomorphisms given by
	\begin{center}
		\begin{tikzcd}
			\bigoplus_T(A_{\kappa(s)(i)})\oplus \bigoplus_T(A_{\kappa(t)(i)}) \arrow[d, "v_T\oplus v_T"'] \arrow[rr, "{a_{s,t}}"] &  & \bigoplus_T(A_{\kappa(s\cup t)(i)})                        \\
			\bigoplus_{|T|}(A_{\kappa(s)(i)})\oplus \bigoplus_{|T|}(A_{\kappa(t)(i)}) \arrow[rr, "S"']                               &  & \bigoplus_{|T|}(A_{\kappa(s\cup t)(i)}) \arrow[u, "v_T^{-1}"']
		\end{tikzcd}
	\end{center}
	where the $v_T$ are the untwistors of $\A$. The map $S$ is the shuffle $G$-isomorphism induced by the $G$-permutative structure of $\A$. We need to show that $\langle A_s,a_{s,t}\rangle$ is indeed an element in $(\overline{\A}_n)^\sigma$. First of all, we need to show that the $a_{s,t}$ are actually $G$-isomorphisms. This follows from a diagram chase in the following diagram
	\begin{center}
        \begin{tikzcd}[column sep = 8mm, row sep = 10 mm]& g\bigoplus_{|T|}\oplus g\bigoplus_{|T|} \arrow[ldd] \arrow[rr, "S_1"]                                                  &  & g\bigoplus_{|T|} \arrow[rdd]                          &                                              \\& g\bigoplus_T\oplus g\bigoplus_T \arrow[rr, "{g.a_{s,t}}"] \arrow[dd] \arrow[u, "gv_T\oplus gv_T" description]                         &  & g\bigoplus_T \arrow[dd] \arrow[u, "gv_T" description] &                                              \\
		\bigoplus_{|T|}(g-)\oplus \bigoplus_{|T|}(g-) \arrow[rdd, "\sigma(g)"'] &                                                                                                                                       &  &                                                       & \bigoplus_{|T|}(g-) \arrow[ldd, "\sigma(g)"] \\& \bigoplus_T(g-)\oplus \bigoplus_T(g-) \arrow[rr, "{ga_{\sigma(g)^{-1}(s),\sigma(g)^{-1}(t)}}"] \arrow[d, "v_T\oplus v_T" description] &  & \bigoplus_T(g-) \arrow[d, "v_T" description]          &                                              \\& \bigoplus_{|T|}(g._\sigma-)\oplus \bigoplus_{|T|}(g._\sigma-) \arrow[rr, "S_2"']                                       &  & \bigoplus_{|T|}(g._\sigma-)                           &                                             
\end{tikzcd}
	\end{center}
	Note that the left and right most pentagons commute using the $\II$-permutative structure of $\A$, the upper and lower square are the defining commutative squares of the maps $a_{s,t}$. The outer hexagon commutes since if $s\cap t=\emptyset$ then $\sigma(g)^{-1}(s)\cap\sigma(g)^{-1}(t)=\emptyset$ and first shuffling and then acting is the same thing as first acting and then shuffling. Since the untwistors are isomorphisms the inner square commutes. Hence, the $a_{s,t}$ are actually isomorphisms in $\A$. The diagram also makes apparent the action
	\[g\langle A_s,a_{s,t}\rangle = \langle gA_{\sigma(h)^{-1}(s)},ga_{\sigma(h)^{-1}(s),\sigma(h)^{-1}(t)}\rangle\]
	so that $\langle A_s,a_{s,t}\rangle$ is indeed an element of $(\overline{\A}_n)^\sigma$ if the required axioms on the maps $a_{s,t}$ are fulfilled. Since $\A$ is $\II$-permutative it holds that $\bigoplus_T(*,\dots,*)=*$. Hence, $a_{s,\emptyset}=\id$. Secondly, consider the following diagram
	\begin{center}
		\begin{tikzcd}
			\bigoplus_{|T|}(A_{\kappa(s)(i)})\oplus \bigoplus_{|T|}(A_{\kappa(t)(i)}) \arrow[rr, "S_1"] \arrow[ddd, "\beta"', bend right, shift right=25] &  & \bigoplus_{|T|}(A_{\kappa(s\cup t)(i)}) \arrow[d, "v_T^{-1}"] \arrow[ddd, Rightarrow, no head, bend left, shift left=10] \\
			\bigoplus_T(A_{\kappa(s)(i)})\oplus \bigoplus_T(A_{\kappa(t)(i)}) \arrow[rr, "{a_{s,t}}"] \arrow[d, "\beta"'] \arrow[u, "v_T\oplus v_T"]     &  & \bigoplus_T(A_{\kappa(s\cup t)(i)}) \arrow[d, Rightarrow, no head]                                                      \\
			\bigoplus_T(A_{\kappa(t)(i)})\oplus \bigoplus_T(A_{\kappa(s)(i)}) \arrow[rr, "{a_{t,s}}"'] \arrow[d, "v_T\oplus v_T"']                       &  & \bigoplus_T(A_{\kappa(s\cup t)(i)})                                                                                     \\
			\bigoplus_{|T|}(A_{\kappa(t)(i)})\oplus \bigoplus_{|T|}(A_{\kappa(s)(i)}) \arrow[rr, "S_2"']                                                 &  & \bigoplus_{|T|}(A_{\kappa(s\cup t)(i)}) \arrow[u, "v_T^{-1}"']                                                         
		\end{tikzcd}
	\end{center}
	The upper and lower squares are the defining commuting squares of the maps $a_{s,t}$ and $a_{t,s}$ and the left and right most outer squares trivially commute. Choosing the shuffle isomorphisms so that $S_2\circ \beta=S_1$ implies that the outer part of the diagram commutes. Since the untwistors are isomorphisms also the inner square commutes. Hence, the maps $a_{s,t}$ fulfil the commutativity axiom.
	
	Similarly, one finds that the following associativity diagram commutes
	\begin{center}
		\begin{tikzcd}
			\bigoplus_T(A_{\kappa(s)(i)})\oplus \bigoplus_T(A_{\kappa(t)(i)})\oplus \bigoplus_T(A_{\kappa(r)(i)}) \arrow[d, "{\id\oplus a_{t,r}}"'] \arrow[rr, "{a_{s,t}\oplus \id}"] &  & \bigoplus_T(A_{\kappa(s\cup t)(i)})\oplus \bigoplus_T(A_{\kappa(r)(i)}) \arrow[d, "{a_{s\cup t,r}}"] \\
			\bigoplus_T(A_{\kappa(s)(i)})\oplus \bigoplus_T(A_{\kappa(t\cup r)(i)}) \arrow[rr, "{a_{s,t\cup r}}"']                                                                   &  & \bigoplus_T(A_{\kappa(s\cup t\cup r)(i)})                                                           
		\end{tikzcd}
	\end{center}
	To see this one pastes the defining commutative squares of the maps $a_{s,t}$ on each side of the diagram and then checks that on the non-twisted parts the associativity holds wich then implies the twisted case.
	
	To show that $\nu$ is $G\Top_*$-internal on objects we need to show that the assignment $(A_i)_{1\le i\le n}\mapsto\langle A_s,a_{s,t}\rangle$ is continuous and equivariant. Continuity is not hard to check and denoting objects of $\A^{\times T}$ by $(A_i)_{1\le i\le n}$ we see that $\nu$ is $G$-equivariant
	\begin{align*}
		\nu(g(A_i)_{1\le i\le n})
		=\nu((gA_{\sigma(g)^{-1}(i)})_{1\le i\le n})
		=\langle gA_{\sigma(g)^{-1}(s)},ga_{\sigma(g)^{-1}(s),\sigma(g)^{-1}(t)}\rangle 
		&=g._\sigma\langle A_s,a_{s,t}\rangle \\
		&=g._\sigma\nu((A_i))
	\end{align*}
	using the definition of the $G$-action in $\overline{\A}_n^\sigma$.
	
	\underline{\smash{Step 2:}} Define a $G\Top_*$-internal functor $\nu\colon (\A^{\times n})^\alpha \to (\overline{\A}_n)^\alpha$ on morphisms as follows: It maps a morphism $(f_i)_{1\le i\le n}\colon (A_i)_{1\le i\le n}\to (A_i')_{1\le i\le n}$ to the morphism $\langle \bigoplus_T(f_{\kappa(s)(i)}) \rangle $ where $f_\emptyset\colon A_\emptyset\to A'_\emptyset$ is the identity on $*$. It is not hard to see that this assignment is continuous and is is $G$-equivariant since
	\begin{align*}
		\nu(g(f_i)_{1\le i\le n})
		=\nu((gf_{\sigma(g)^{-1}(i)}))
		=\langle \textstyle\bigoplus_n^\sigma(gf_{\kappa(s)(\sigma(g)^{-1}(i)}))\rangle
		&= g._\sigma \langle \textstyle\bigoplus_n^\sigma(f_{\kappa(s)(i)})\rangle \\
		&= g._\sigma \nu((f_i)_{1\le i\le n})
	\end{align*}
	for any $g\in G$.
	
	We also have to show that for any $(f_i)_{1\le i\le n}\colon (A_i)_{1\le i\le n}\to (A_i')_{1\le i\le n}$ the maps $a_{s,t}$ as defined above are compatible with $\nu((f_i))$. This follows from the following diagram
	\begin{center}
		\begin{tikzcd}
			\bigoplus_{|T|}(A_{\kappa(s)(i)})\oplus \bigoplus_{|T|}(A_{\kappa(t)(i)}) \arrow[rr, "S_1"] \arrow[ddd, "|f_s|\oplus |f_t|"', bend right, shift right=25] &  & \bigoplus_{|T|}(A_{\kappa(s\cup t)(i)}) \arrow[d, "v_T^{-1}"] \arrow[ddd, "|f_{s\cup t}|", bend left, shift left=10] \\
			\bigoplus_T(A_{\kappa(s)(i)})\oplus \bigoplus_T(A_{\kappa(t)(i)}) \arrow[rr, "{a_{s,t}}"] \arrow[d, "f_s\oplus f_t"'] \arrow[u, "v_T\oplus v_T"]         &  & \bigoplus_T(A_{\kappa(s\cup t)(i)}) \arrow[d, "f_{s\cup t}"]                                                        \\
			\bigoplus_T(A'_{\kappa(s)(i)})\oplus \bigoplus_T(A'_{\kappa(t)(i)}) \arrow[rr, "{a_{s,t}'}"'] \arrow[d, "v_T\oplus v_T"']                                &  & \bigoplus_T(A'_{\kappa(s\cup t)(i)})                                                                                \\
			\bigoplus_{|T|}(A'_{\kappa(t)(i)})\oplus \bigoplus_{|T|}(A_{\kappa(s)(i)}) \arrow[rr, "S_2"']                                                            &  & \bigoplus_{|T|}(A'_{\kappa(s\cup t)(i)}) \arrow[u, "v_T^{-1}"']                                                    
		\end{tikzcd}
	\end{center}
	We need to show that the inner square commutes. The upper and lower one commute by definition. We denote by $|f_s|$ etc. the underlying induced maps on the non-twisted direct sum. The outer squares commute by definition of the untwistor maps. Again, since $s\cap t=\emptyset$ and thus $\kappa(s)(i)\ne \kappa(t)(j)$ for all $1\le i,j\le n$ the outer part commutes. Using that the untwistor maps are isomorphisms it follows that the center commutes.
	
	\underline{\smash{Step 3:}} To show that $\nu$ is a $G\Top_*$-internal functor it remains to show that it is compatible with source, target, identity and composition, but this is a straightforward check and follows from our construction. For instance, $\nu(S((f_i)))=\nu(A_1,\dots,A_n)=\langle A_s,a_{s,t}\rangle=S(\nu((f_i)))$ and similarly for $T$. The identity maps $\id_{(A_i)_{1\le i\le n}}$ to the identity $\langle \id_s\rangle$ with $\id_s\colon \bigoplus_T(A_\kappa(s)(i))\to \bigoplus_T(A_\kappa(s)(i))$. Composition is induced by composition in $\A$ so that it is not hard to see that $\nu$ is compatible with composition. Hence, $\nu$ is indeed a $G\Top_*$-internal functor.
	
	\underline{\smash{Step 4:}} Note that $\delta\nu$ is the identity functor on $\A^{\times T}$. To show that $\nu\delta$ is $G\Top_*$-naturally isomorphic to the identity on $(\overline{\A}_n)^\sigma$ we need to define a $G\Top_*$-internal natural transformation $\eta\colon\Id\Rightarrow \delta\nu$. Hence, we must define for each object $\langle A_s,a_{s,t}\rangle^\sigma$ a map $\langle \eta_s\rangle$ with components $\eta_s\colon A_s \to \bigoplus_T(A_{\kappa(s)(i)})$. This will be done inductively.
	
	For $|s|=0$, $\eta_\emptyset\colon A_\emptyset\to *$ is equal to the identity using axiom (P4) on $\II$-NPCs. For $|s|=1$ with $A_s=A_i$ for $s=\{i\}$, $\eta_s\colon A_i^\sigma \to \bigoplus_T(*,\dots,*,A_i,*,\dots,*)$ is also the identity using (P4). Hence we can identify $A_i^\sigma$ with $\bigoplus_T(*,\dots,*,A_i,*,\dots,*)$. Let us next consider the case $|s|=2$ with $s=\{i,j\}$ and $i<j$. Using the permutative structure of $\A$ we can rewrite $A_i^\sigma\oplus A_j^\sigma$ as the untwisted direct sum
	\[*\oplus \dots\oplus * \oplus A_i^\sigma\oplus *\oplus \dots\oplus *\oplus A_j^\sigma \oplus \dots\oplus * = \textstyle\bigoplus_n(*,\dots,*, A_i^\sigma,*,\dots,*,A_j^\sigma,*,\dots,*)\]
	with $A_i$ and $A_j$ in the $i$-th and $j$-th spot. There are two $\Sigma_n$-action on these $G$-spaces. There is a $\Sigma_n$-action by twisting factors and a diagonal action on each summand. In this case the two actions agree so that one can identify $A_i^\sigma\oplus A_j^\sigma$ with $\bigoplus_T(*,\dots,*, A_i,*,\dots,*,A_j,*,\dots,*)$ using the untwistor of $\A$. More precisely, the untwistor gives a not-necessarily $G$-equivariant natural isomorphism $\nu_T\colon \oplus_T\Rightarrow \oplus_n$. Recalling the twisted equivariance diagram in Definition \ref{def:I-NPC} the non-equivariance of $\nu_T$ stems from a discrepancy of $\Sigma_n$-actions on $\oplus_T$ and $\oplus_n$. In this case this discrepancy vanishes for objects $A_i^\sigma$ so that there is a natural $G$-isomorphism
	\[\textstyle\bigoplus_T(*,\dots,*, A_i,*,\dots,*,A_j,*,\dots,*) \to A_i^\sigma\oplus A_j^\sigma.\]
	Since $\nu_T$ is part of the data of $\A$ we can treat this as an identification. Define
	\[\eta_s\colon A_{\{i,j\}}^\sigma \overset{(a_{i,j}^\sigma)^{-1}}{\longrightarrow} A_i^\sigma\oplus A_j^\sigma =\bigoplus_T(0,\dots,0,A_i,0,\dots,0,A_j,0,\dots,0)\] 
	Inductively, for $|s|<n$ with $i\notin s$ and given map $\eta_s$, define
	\begin{center}
		\begin{tikzcd}
			A_{s\cup\{i\}}^\sigma \arrow[r, "{(a_{s,\{i\}}^\sigma)^{-1}}"] \arrow[rrd, "\eta_{s\cup\{i\}}"'] & A_s^\sigma \oplus A_i^\sigma \arrow[r, "\eta_s\oplus\id_{A_i^\sigma}"] & \bigoplus_T(A_{\kappa(s)(j)})\oplus A_i^\sigma \arrow[d, "\cong"] \\ & & \bigoplus_T(A_{\kappa(s\cup\{i\})(j)})                           
		\end{tikzcd}
	\end{center}
	The vertical isomorphism stems from identifying $A_i^\sigma$ with $\bigoplus_{T}(*,\dots,*,A_i,*,\dots,*)$ and, using that $i\notin s$, untwisting shuffling $A_i$ into the $i$-th spot replacing $*$ and then twisting again. As mentioned above, we will treat this as an identification. Using associativity and commutativity of the maps $a_{s,t}^\sigma$ the maps $\eta_s$ are determined uniquely by the data of $\langle A_s,a_{s,t}\rangle^\sigma$.
	
	It remains to show that this $\eta$ is a continuous $G$-map $\eta\colon \ob (\overline{\A}_n)^\sigma \to \mor (\overline{\A}_n)^\sigma$ fulfilling the two commutativity diagrams of Definition 1.9 in \cite{GMMO18}. For $g\in G$ and $\langle A_s,a_{s,t}\rangle^\sigma\in (\overline{\A}_n)^\sigma$ it holds that
	\[g\eta_{\langle A_s,a_{s,t}\rangle^\sigma}=g\langle \eta_s\rangle^\sigma = \langle g\eta_{\sigma(g)^{-1}(s)}\rangle^\sigma = \eta_{\langle gA_{\sigma(g)^{-1}(s)},a_{\sigma(g)^{-1}(s),\sigma(g)^{-1}(t)}\rangle^\sigma}=\eta_{g \langle A_s,a_{s,t}\rangle^\sigma}.\]
	Hence, $\eta$ is a $G$-map. Continuity of $\eta$ is not hard to see using the inductive definition of $\eta_s$, using the subspace topology of the product topology on $\ob(\A_n)^\sigma$ and $\mor(\A_n)^\sigma$ as defined in \ref{constr:perm-G-cat-pt.1}, and the $G\Top_*$-internal structure of $\A$. For example, $\eta_\emptyset$ and $\eta_{\{i\}}$ are induced by $I$ in $\A$, $\eta_{\{i,j\}}$ is a projection composed with an isomorphism and inductively one uses that the composition map $C$ of $\A$ is continuous.
	
	It remains to show naturality. First, note that
	\[(S,T)(\eta_{\langle A_s,a_{s,t}\rangle^\sigma})=(\langle A_s,a_{s,t}\rangle^\sigma,\langle \textstyle\bigoplus_T(A_{\kappa(s)(i)}),\hat a_{s,t}\rangle)=(\Id,\nu\delta)(\eta_{\langle A_s,a_{s,t}\rangle^\sigma})\]
	where $\hat a_{s,t}$ denote the maps $a_{s,t}$ defined in Step 1 of this proof. This implies commutativity of the first naturality diagram. The second naturality diagram boils down to showing that for any morphism $\langle\alpha_s\rangle^\sigma\colon \langle A_s,a_{s,t}\rangle^\sigma\to \langle A_s',a_{s,t}'\rangle^\sigma$ in $(\overline{\A}_n)^\sigma$ the following diagram commutes
	\begin{center}
		\begin{tikzcd}
			A_s^\sigma \arrow[d, "\eta_s"'] \arrow[rr, "\alpha_s"]         &  & A_s'^\sigma \arrow[d, "\eta_s'"] \\
		\bigoplus_T(A_{\kappa(s)(i)}) \arrow[rr, "\delta\nu\alpha_s"'] &  & \bigoplus_T(A_{\kappa(s)(i)}')  
		\end{tikzcd}
	\end{center}
	This follows from the fact that $\eta_s$ is built inductively from maps $a_{s,t}$ and the comparison isomorphisms using compatibility of $a_{s,t}$ with maps $\langle\alpha_s\rangle$. For $s=\emptyset$ and $s=\{i\}$ the diagram trivially commutes. For $s=\{i,j\}$ with $i< j$ this follows directly form Construction \ref{constr:perm-G-cat-pt.1} and the fact that $a_{\{i,j\}}$ is compatible with $\alpha_{\{i,j\}}$ and using the identification above. One proceeds inductively to show that the above diagram commutes utilising the defining diagram of $\eta_s$ above and compatibility of $a_{s,t}$ with $\alpha_{s}$.
	\begin{center}
		\begin{tikzcd}[scale = 0.8, column sep = 5mm]
A^\sigma_s\oplus A^\sigma_{\{i\}} \arrow[rrrr, "\alpha_{s}\oplus \alpha_{\{i\}}"] \arrow[ddd, "\eta_s\oplus \id_{{A_i^\sigma}}"']      &                                                                                                                                   &  &                                                                                 & A'^\sigma_s\oplus A'^\sigma_{\{i\}} \arrow[ddd, "\eta_s'\oplus\id_{A_i'^\sigma}"] \\& A_{s\cup\{i\}}^\sigma \arrow[d, "\eta_{s\cup\{i\}}"] \arrow[rr, "\alpha_{s\cup\{i\}}^\sigma"] \arrow[lu, "{a_{s,\{i\}}^\sigma}"'] &  & A_{s\cup\{i\}}'^\sigma \arrow[d, "\eta_s'"] \arrow[ru, "{a_{s,\{i\}}'^\sigma}"] &                                                                                 \\& \bigoplus_T(A_{\kappa(s\cup\{i\})(i)}) \arrow[rr, "\delta\nu\alpha_s"']                                                           &  & \bigoplus_T(A_{\kappa(s)(i)}') \arrow[rd, Rightarrow, no head,]                  &                                                                                 \\ \bigoplus_T(A_{\kappa(s)(i)})\oplus A_i^\sigma \arrow[ru, Rightarrow, no head] \arrow[rrrr, "\delta\nu\alpha_s\oplus\id_{A_i^\sigma}"] &                                                                                                                                   &  &                                                                                 & \bigoplus_T(A'_{\kappa(s)(i)})\oplus A_i'^\sigma                               
		\end{tikzcd}
	\end{center}
	The left and right squares are the defining squares of $\eta$. The top square commutes by Construction \ref{constr:perm-G-cat-pt.1} and the lower square commutes using the identification. Since the inner square is connected to the outer square through isomorphisms the inner square commutes. This concludes the proof that $\overline{\A}$ is an $\II$-special $\Gamma_\II$-$G$-category.
\end{proof}

\begin{proof}[Proof of Theorem \ref{thm:theorem-B}]
	By Theorem \ref{thm:I-NPC->I-special-Gamma-cat} the $\Gamma_\II$-$G$-category $\overline{\A}$ is $\II$-special. Since by assumption $\A$ is nice we obtain by Proposition \ref{prop:Gamma-cat->Gamma-space} an $\II$-special $\Gamma_\II$-$G$-space $B\overline{\A}$. Hence, $\sS_\II^{G,U}B\overline{\A}$ is a positive connective $\Omega$-$G$-spectrum indexed on a choice of compatible $G$-universe $U$ with group completion $\smash{B\overline{\A}_1=B\A \to \sS_\II^{G,U}B\overline{\A}_1}$ and thus the is Segal $K$-theory of the $\II$-NPC $\A$.
\end{proof}


\begin{appendices}

\section{Proofs of technical results}\label{appendix}

\subsection{The incomplete invariance theorem}\label{app:invariance-theorem}

To prove the wedge lemmawe need a preliminary result, which in the case $\II=\CC$ is due to \cite{GMMO19b}, Theorem 2.5/ 2.6\footnote{Implicitly, \cite{MMO25} also use a version of the invariance for $\II=\TT$, but I could not find a statement or proof anywhere.}. Let $\II\subset\JJ$. Note that the property of a $G\Top_*$-functor $\Gamma_\II\to G\Topu_*$ or $G\underline{\CW}_*^\II\to G\Top_*$ to be a $\Gamma_\II$ or $G\underline{\CW}_*^\II$-$G$-space solely relies on the underlying $\Gamma$-$G$-space. Hence, Proposition \ref{prop:commutative-diagram-prolongation} specialises to give a commutative diagram:
\begin{center}
	\begin{tikzcd}
		{\Gamma_\JJ[G\Topu_*]} \arrow[rr, "\PP_\JJ"]                          &  & {\underline{\CW}_*^\JJ[G\Topu_*]}                          \\
		{\Gamma_\II[G\Topu_*]} \arrow[u, "\pp_\II^\JJ"] \arrow[rr, "\PP_\II"'] &  & {\underline{\CW}_*^\II[G\Topu_*]} \arrow[u, "\PP_\II^\JJ"']
	\end{tikzcd}
\end{center}

The invariance theorem basically states that under suitable cofibrancy conditions any $\II$-level $G$-equivalence of $\Gamma_\II$-$G$-spaces induces under prolongation a level equivalence of $G\underline{\CW}_*^\II$-$G$-spaces. We will follow the nomenclature in \cite{GMMO19b}, Definition 2.3. The original notion of this cofibrancy condition was defined for $\Gamma$ and $\Gamma_G$-$G$-spaces with genuine structures in mind. In the incomplete setting we cannot always expect cofibrancy for finite $G$-sets outside of $\II$.

\begin{Def}
	Let $\II\subset\JJ$ be indexing systems. A $\Gamma_\II$-$G$-space $X$ is called \textit{$\JJ$-proper} if for any based simplicial $G$-set $A_*$ with isotropy in $\FF_\JJ$ the simplicial $G$-space $\PP_\JJ\pp_\II^\JJ X(A_*)$ is Reedy cofibrant.
\end{Def}

We defined $\PP_\JJ\pp_\II^\JJ X$ to be a $G\Top_*$-functor $G\underline{\CW}_*^\II\to G\Topu_*$ but we can extend this to a simplicial based $G$-spaces by setting on $q$-simplices $(\PP_\JJ\pp_\II^\JJ X(A))_q=\PP_\JJ\pp_\II^\JJ X(A_q)$.

\begin{Rem}
	For $\II\subset\JJ\subset\KK$ a $\Gamma_\II$-$G$-space $X$ is $\KK$-proper if and only if $\pp_\II^\JJ X$ is $\KK$-proper. We recover the two notions of properness in  \cite{GMMO19b}, Definition 2.3 by considering $\CC$-proper $\Gamma_\TT$-$G$-spaces and $\CC$-proper $\Gamma_\CC$-$G$-spaces. 
	
	Note that a $\TT$-proper $\Gamma_\II$-$G$-space $X$ is proper if and only if $\PP_\II(A_*)$ is Reedy cofibrant for any simplicial set with trivial $G$-action.
\end{Rem}

\begin{Ex}\label{ex:bar-constructions-are-proper}
	The $\Gamma_\II$-$G$-space $\bar{b}_\II X$ is proper for any $\Gamma_\II$-$G$-space $X$. To prove this, we need to show for any based simplicial set $A_*$ with isotropy in $\FF_\II$ that $\PP_\II \bar{b}_\II X(A_*)$ is Reedy cofibrant. As we have defined $b_\II :=\PP_\II\circ \bar{b}_\II$, it remains to show that $b_\II X(A_*)$ is Reedy cofibrant. In Lemma \ref{lem:cofibrancy-lemma} we will show that $b_\II$ preserves Reedy cofibrancy, that is, if $A_*$ is Reedy cofibrant, so is $b_\II(A_*)$. Recall that any simplicial $G$-set $A_*$ is Reedy cofibrant (cf. \cite{GMMO19b}, remark after Definition 2.3). Hence, $\bar{b}_\II$ is $\II$-proper. The same argument shows that $\bar{b}_\II^\times$ is $\II$-proper.
\end{Ex}

\begin{Thm}[Invariance theorem]\label{thm:invariance-theorem}
	Let $f\colon X\to Y$ be an $\II$-level equivalence of $\II$-proper $\Gamma_\II$-$G$-spaces. Then the induced map
	\[\PP_\II(f)(A)\colon \PP_\II(X)(A)\to \PP_\II(Y)(A)\]
	is a weak $G$-equivalence for all based $\II$-$G$-CW complexes $A$.
\end{Thm}

\begin{Rem}\label{rem:I-G-CW-cx-vs-simplicial-I-sets}
	Note that any finite based $G$-CW complex is weakly $G$-homotopy equivalent to the geometric realization of a finite simplicial $G$-set. A full proof this statement can be found in \cite{Sch18}, Proposition B.46. Observe that if $A$ is a finite $G$-CW-complex and $K_*$ is a finite simplicial $G$-set so that $|K_*|\simeq A$, then $A$ has a cell of type $G/H\land D^n$ for each non-degenerate $n$-simplex with isotropy group $H$ (cf. \cite{GMMO19b}, \S2.2). This implies that $A$ is a based $\II$-$G$-CW complex if and only if $K_*$ is a simplicial $G$-set with isotropy type in $\FF_\II$.
\end{Rem}

Following \cite{GMMO19b}, \S2.2, consider the adjunction $|-| \vdash\Sing$. Let $A$ be a based $G$-space. Then $\Sing(A)$ is a based simplicial $G$-space with $G$ acting on a simplex $\Delta^n\to A$ via the trivial action on $\Delta^n$ and standard action on $A$. Clearly, $\Sing(A)^H=\Sing(A^H)$ since a simplex $f\colon \Delta^n\to A$ is $H$-fixed if and only if it takes values in $A^H$. Similarly, for any finite $G$-set $K_*$, $|K_*^H|=|K_*|^H$ using that fixed points commute with realization. From this it follows that the natural $G$-map $\varepsilon\colon |\Sing(A)|\to A$ restricts on $H$-fixed points to weak equivalences $\varepsilon^H\colon |\Sing(A^H)|\to A^H$. Therefore, if $A$ is a based $\II$-$G$-CW complex then $\varepsilon$ is a weak $G$-equivalence. This implies the following lemma.

\begin{Lem}\label{lem:replacing-G-CW-with-GR}
	For any $\Gamma_\II$-$G$-space $X$ and weak $G$-equivalence $\varepsilon\colon |\Sing(A)|\to A$ as above the induced map
	\[\PP_\II X(\varepsilon)\colon |\Sing(A)|^\bullet\otimes_{\Gamma_\II}X\to A^\bullet \otimes_{\Gamma_\II} X\]
	is a weak $G$-homotopy equivalence.
\end{Lem}

\begin{Lem}\label{lem:prolongation-restricted-to-finite-sets}
	Let $f\colon X\to Y$ be an $\II$-level $G$-equivalence of $\Gamma_\II$-$G$-spaces. Then 
	\[\PP_\II X(f)(A)\colon A^\bullet\otimes_{\Gamma_\II} X\to A^\bullet\otimes_{\Gamma_\II} Y\]
	is a weak $G$-equivalence for any finite based $G$-set $A$ in $G\Fin_*^\II$.
\end{Lem}

\begin{proof}
	Since $A$ is a finite based $G$-set with isotropy in $\FF_\II$ it is $G$-homeomorphic to some $T_+$. The proof then follows from the Yoneda lemma for coends (cf. \cite{Lor21}, Lemma 2.2.1, Remark 4.3.5) since
	\[\PP_\II X(T_+) =\int^{S_+\in \Gamma_\II}F(S_+,T_+)\land X(S_+)=\int^{S_+\in \Gamma_\II}\Gamma_\II(S_+,T_+)\land X(S_+)\cong X(T_+).\]
	Hence, $\PP_\II X(A)^H\cong \PP_\II X(T_+)^H\cong  X(T_+)^H$ and therefore $f(T_+)$ is a weak $G$-equivalence if and only if the map $\PP_\II X(f)(A)$ is one.
\end{proof}

The following proposition will then imply the invariance theorem. The proof idea follows directly \cite{GMMO19b}, Proposition 2.10.

\begin{Prop}
	Let $f\colon X\to Y$ be an $\II$-level equivalence of $\II$-proper $\Gamma_\II$-$G$-spaces. Then for any finite based $\II$-$G$-CW complex $A$, the induced map 
	\[\PP_\II(f)(A)\colon A^\bullet\otimes_{\Gamma_\II}X\to A^\bullet\otimes_{\Gamma_\II} Y\] 
	is a weak $G$-equivalence.
\end{Prop}

\begin{proof}
	By Lemma \ref{lem:replacing-G-CW-with-GR} we can replace $A$ in the proposition with $|K_*|$ so that $K_n$ is a based $G$-set with isotropy in $\FF_\II$ for all $n\ge 0$. Therefore, $K_*$ can be seen to be a simplicial object $G\Fin_*^\II$. Since products commute with realization and since the coends that define prolongation commute with the coends defining geometric realization (cf. proof of \cite{GMMO19b}, Proposition 2.10), we see that
	\[|K_*|^\bullet \otimes_{\Gamma_\II} X \cong |K_*\otimes_{\Gamma_\II}  X|\]
	and similarly for $Y$. Hence, $\PP_\II(f)(A)$ is the realization of the induced simplicial $G$-map $f_*'\colon K_*\otimes_{\Gamma_\II}X\to K_*\otimes_{\Gamma_\II} Y$. On $q$-simplices we have that $K_q$ is a finite based $G$-set with isotropy in $\FF_\II$ and thus by Lemma \ref{lem:prolongation-restricted-to-finite-sets} we know that $f_n'$ is a weak $G$-equivalence. Hence, $f'_*$ is a level weak $G$-equivalence. Since $X$ and $Y$ are $\II$-proper we know that $f_*'$ is a map between Reedy cofibrant $G$-spaces. Theorem \ref{thm:MMO-thm-1.12} then implies that $|f_*'|=\PP_\II f(A)$ is a weak $G$-equivalence.
\end{proof}


\subsection{Proof of the wedge lemma}\label{app:wedge-lemma}

The following proposition expands on a passing remark in \cite{MMO25}, proof of Proposition 8.14/ beginning of \S 8.4. It uses Lemma \ref{lem:MMO-lem.1.17} where one can interchange $G\Top_*$-functors $\Gamma_\II\to G\Topu_*$ with reduced $G\Top$-functors $\Gamma_\II\to G\Topu$, i.e. functors $X$ such that $X_0=*$. The main problem with $b_\II^\times X$ is that the condition $X_0=*$ does not in general hold (cf. \cite{MMO25}, Remark 3.11).

\begin{Prop}
	For any finite based $\II$-$G$-CW complex $A$, there is a weak $G$-equivalence
	\[b_\II^\times X(A) \xrightarrow{\simeq} b_\II X(A).\]
\end{Prop}

\begin{proof} Note that we have the following commutative diagram $G\Top$-functors $\Gamma_\II\to G\Topu$, akin to the diagram labelled (3.20) in \cite{GMMO19b}	
	\begin{center}
		\begin{tikzcd}
			b_\II^\times X \arrow[rrd, "\varepsilon^\times"'] \arrow[rr, "q"] &  & b_\II^\times X/b_\II^\times X(*) \arrow[rr, "p"] \arrow[d, "\varepsilon'" description] &  & b_\II X \arrow[lld, "\varepsilon"] \\                                                                  &  & X                                                                                      &  &                                      
		\end{tikzcd}
	\end{center}
	The maps $\varepsilon^\times$ and $\varepsilon$ come from Proposition 3.13 in \cite{MMO25}, the maps $q$ and $p$ are quotient maps and $\varepsilon'$ is the induced map of $\varepsilon^\times$ on the quotient. By Proposition \ref{prop:bar-construction-approximation-of-X} we find that $\varepsilon^\times$ is a level $G$-equivalence of $G\Top$-functors and that $\varepsilon$ is a level $G$-equivalence of $\Gamma_\II$-$G$-spaces. Under the assumption of non-degenerate base points of $X$ we can use \cite{MMO25}, Remark 3.11 to see that the map $q$ is a level $G$-equivalence. Therefore, also $\varepsilon'$ is a level $G$-equivalence which implies that $p$ is a level $G$-equivalence. We then use the fact that
	\[(\PP_\II b_\II X)(A)=A^\bullet\otimes_{\Gamma_\II}B(\Gamma_\II,\Gamma_\II,X)\cong B(A^\bullet,\Gamma_\II, X)=b_\II X(A)\]
	follwoing Remark 3.8 in \cite{MMO25}. By Lemma \ref{lem:realization-lemma} we know that $b_\II X$ and $b_\II^\times X$ preserve Reedy cofibrancy. This fact together with Example \ref{ex:bar-constructions-are-proper} implies that $p\circ q$ on the underlying map $\bar{b}_\II^\times X\to \bar{b}_\II X$ is an $\II$-level $G$-equivalence between $\II$-proper $G\Top$-functors\footnote{We have defined this notion only for $G\Top_*$-functors but the definition is the same for $G\Top$-functors.}. As pointed out in \cite{MMO25}, Footnote 27 the non-reducedness of $b_\II X$ is irrelevant here. In particular the invariance theorem (Theorem \ref{thm:invariance-theorem}) holds which then implies that $p\circ q (A)$ is a weak $G$-equivalence.
\end{proof}

Being able to use $b^\times_\II X$ we can now use its description as the realization of the nerve of an \textit{Grothendieck category of elements} as in \cite{MMO25}, Remark 3.2. We want to construct a weak $G$-equivalence
\[b_\II^\times X(A\lor B)\to b_\II^\times X(A)\times b_\II^\times X(B).\]
To do so we use \cite{MMO25}, Remark 3.2 to construct an equivalence of categories whose realization yields the desired weak homotopy equivalence.

\begin{Constr} Let $A,B$ be finite based $\II$-$G$-CW complexes and let $X$ be an $\II$-special $\Gamma_\II$-$G$-space. Analogously to \cite{MMO25}, Remark 3.2 and \S 8.4 define the following categories:
	\begin{enumerate}
		\item Let $\D(A,X)$ be the  $G\Top_*$-internal category whose object $G$-space is
		\[\coprod_{a}A^a\times X(a)\]
		and whose morphism $G$-space is
		\[\coprod_{a,c} A^c\times\Gamma_\II(a,c)\times X(a)\]
		where $a,b$ are finite $G$-sets in $\Gamma_\II$. The source and target maps are induced from the evaluation maps of the contravariant $G\Top_*$-functor $A^\bullet\colon\Gamma_\II\to G\Topu_*$. The identity map is induced by identity morphisms and composition is induced by composition in $\Gamma_\II$. Repeat the same construction for $B$ to obtain $\D(B,X)$. By Remark 3.2 \cite{MMO25} we then have that
		\[|\D(A,X)|\cong b_\II^\times X(A) \text{ and } |\D(B,X)|\cong b_\II^\times X(B).\]
		\item Note that at this point it is crucial to work with a nice category of topological $G$-spaces (i.e. compactly generated weak Hausdorff $G$-spaces) so that the product commutes with geometric realization and we can therefore write
		\[|\D(A,X)\times \D(B,X)|\cong |\D(A,X)|\times |\D(B,X)|\cong b_\II^\times X(A) \times b_\II^\times X(B).\]
		\item Define a $G\Top_*$-internal category $\D(A,B;X)$ with object $G$-space
	\[\coprod_{a,b}(A^a\times B^b)\times X(a\lor b)\]
	where $a,b$ range over finite $G$-sets in $\Gamma_\II$. Note that since $\II$ is an indexing system and is thus closed under coproducts (I6) the expressions $X(a\lor b)$ are well-defined. We define the morphism $G$-space by
	\[\coprod_{a,b,c,d}(A^c\times B^d)\times \Gamma_\II(a,c)\times\Gamma_\II(b,d)\times X(a\lor b)\]
	where the source and target maps $S,T$ are induced from the evaluation maps of $A^\bullet\times B^\star$ and the composite of $\lor\colon \Gamma_\II\times\Gamma_\II\to\Gamma_\II$ with the evaluation maps of $X$. More precisely, $T(\mu,\phi,\psi,x)=(\mu,X(\phi\lor\psi)(x))$ and $S(\mu,\phi,\psi,x)=(\mu\circ(\phi\lor \psi),x)$. Again, the identity map is induced by identity morphisms, where $(\mu,x)$ is sent to $(\mu,\id_a,\id_b,x)$, and composition is induced from composition in $\Gamma_\II$. 
	\end{enumerate}	
\end{Constr}

The proof of the wedge lemma then reduces to the following proposition. The statement is the same as in \cite{MMO25}, \S 8.4  and proof works almost exactly the same. But we need to be careful that the construction is contained within $\II$.

\begin{Rem}
	Note that the category $\D(A,B;X)$ is also internal to $G\Top_*$ and realizes to give a bar construction of the form
	\[B_*^\times (A^\bullet\times B^\star, \Gamma_\II\times\Gamma_\II,X(\bullet\lor \star))\]
	as in \cite{MMO25}, \S 8.4.
\end{Rem}

\begin{Prop}\label{prop:lemma-to-wedge-lemma}
	There are functors
	\begin{center}
		\begin{tikzcd}
             & {\D(A,B;X)} \arrow[ld, "F"'] \arrow[rd, "Q"] &                        \\
			\D(A\lor B;X) \arrow[rr, "P"'] &                                               & \D(A;X)\times\D(B;X)
		\end{tikzcd}
	\end{center}
	so that the diagram commutes up to natural transformations. Moreover, the functor $F$ is part of an equivalence of categories.
\end{Prop}

\begin{proof}
	We follow the proof in \cite{MMO25}, \S 8.4 and modify where needed. Write elements in $A^a$ etc. as based maps $\mu\colon a\to A$ with $G$ acting by conjugation. The functor $F$ is given on objects by $F(\mu,\nu,x)=(\mu\lor \nu,x)$ and on morphisms by $F(\mu,\nu,\phi,\psi,x)=(\mu\lor \nu,\phi\lor\psi,x)$. This is all well defined using that indexing systems are closed under coproducts (I6) and isomorphisms (I2). Similarly, we define on objects $Q(\mu,\nu,x)=(\mu,x_a)\times (\nu,x_b)$, where $x_a$ and $x_b$ are obtained from $X(a\lor b)$ via the projection maps $\pi_a\colon X(a\lor b)\to X(a)$ induced from the projection maps $a\lor b\to a$. On morphisms we define $Q(\mu,\nu,\phi,\psi,x)=(\mu,\phi,x_a)\times (\nu,\psi,x_b)$. Lastly, $P$ is induced by the projection maps $\pr_A\colon A\lor B\to A$ and $\pr_B$. More precisely, define on objects $P(\mu,x)=(\mu_A,x)\times (\mu_B,x)$. Here, $\mu_A=\pr_A\circ\mu\colon a\to A\lor B\to A$ and similarly for $\mu_B$. On morphisms we define $P(\mu,\phi,x)=(\mu_A,\phi,x)\times (\mu_B,\phi,x)$.
	
	We have to show that $F,P$ and $Q$ are $G\Top_*$-internal functors and we need to construct a $G\Top_*$-internal functor $F^{-1}\colon \D(A\lor B;X)\to \D(A,B;X)$ together with $G\Top_*$-internal natural transformations $\Id\Rightarrow F^{-1}F$ and $FF^{-1}\Rightarrow\Id$. This part of the proof is exactly the same as in \cite{MMO25}, \S 8.4 repeatedly noting that $\II$ is closed under coproducts, subobjects and isomorphisms. For self-containment we will repeat the argument here. It is easy to see that $P$ and $Q$ are $G\Top_*$-internal functors.
	
	For $\omega\colon f\to A\lor B$ with $f$ a finite $G$-set with isotropy in $\FF_\II$, define $\omega_A=\pi_A\circ \omega$ and $\omega_B=\pi_B\circ\omega$. Define $\sigma_\omega\colon f\to f\lor f$, called the \textit{splitting} of $\omega$, by
	\[\sigma_\omega(j)=\begin{cases}
		j &\text{in the first copy of $f$ if }\omega(j)\in A\backslash * \\
		j &\text{in the second copy of $f$ if }\omega(j)\in B\backslash * \\
		* &\textit{if }\omega(j)=*.
	\end{cases}\]
	Note that $f\lor f$ also has isotopy in $\FF_\II$ using closure under coproducts (I6) so that $\sigma_\omega$ is a map in $\Gamma_\II$. Observe that $\omega$ factors as the composite
	\[f\overset{\sigma_\omega}{\longrightarrow}f\lor f \overset{\omega_A\lor \omega_B}{\longrightarrow} A\lor B.\]
	Define $F^{-1}$ on objects by
	\[F^{-1}(\omega,x)=(\omega_A,\omega_B,X(\sigma_\omega)(x))\]
	where $X(\sigma_\omega)\colon X(f)\to X(f\lor f)$ is well defined by the above remark. For a morphism $(\omega,\phi,x)$, $\omega\colon f\to A\lor B$, $\phi\colon e\to f$, and $x\in X(e)$, observe that
	\[(\omega\circ\phi)_A=\omega_A\circ \phi \text{ and } (\omega\circ\phi)_B=\omega_B\circ \phi.\] 
	Define $F^{-1}$ on morphisms by
	\[F^{-1}(\omega,\phi,x)=(\omega_A,\omega_B,\phi,\phi,X(\sigma_{\omega\circ\phi})(x)).\]
	A check of definitions using the commutative diagram
	\begin{center}
		\begin{tikzcd}
			e \arrow[d, "\sigma_{\omega\circ\phi}"'] \arrow[r, "\phi"] & f \arrow[d, "\sigma_\omega"'] \arrow[r, "\omega"] & A\lor B \\
			e\lor e \arrow[r, "\phi\lor\phi"']                         & f\lor f \arrow[ru, "\omega_A\lor \omega_B"']      &        
		\end{tikzcd}
	\end{center}
	shows that $S\circ F^{-1}=F^{-1}\circ S$ and $T\circ F^{-1}=F^{-1}\circ T$, for $S,T$ the source and target maps of the respective categories. Note that $F^{-1}$ is clearly compatible with composition and identities. It is easily checked that $F$ is continuous on object and morphism $G$-spaces. Equivariance is not immediate but is shown analogously to the genuine case. We defer to \cite{MMO25}, \S 8.4.
	
	\underline{Claim:} There is a natural transformation $\Id\Rightarrow F\circ F^{-1}$.
	
	The composition $F\circ F^{-1}$ maps an object $(\omega,x)$ in $\D(A\lor B;X)$ to $(\omega_A\lor \omega_B,X(\sigma_\omega)(x))$. Here, $(\omega_A\lor \omega_B,\sigma_\omega,x)$ is a morphism from $(\omega,x)$ to $(\omega_A\lor \omega_B,X(\sigma_\omega)(x))$. We claim that this morphism is the component at $(\omega,x)$ of a natural transformation $\Id\Rightarrow F\circ F^{-1}$. These morphisms clearly give a continuous map from the object $G$-space of $\D(A\lor B;X)$ to the morphism $G$-space, and it is not hard to see naturality using the diagram above. To show equivariance, for $g\in G$ note that
	\[g.(\omega_A\lor \omega_B,\sigma_\omega,x)=(g.(\omega_A\lor \omega_B),g.\sigma_\omega,gx)=((g.\omega)_A\lor (g.\omega)_B),\sigma_{g.\omega},gx),\]
	which is precisely the component at $g.(\omega,x)=(g.\omega,gx)$.
	
	\underline{Claim:} There is a natural transformation $\Id\Rightarrow F^{-1}\circ F$.
	
	The composite $F^{-1}\circ F$ sends an object $(\mu,\nu,x)$ in $\D(A,B;X)$ to the object
	\[(\pi_A\circ(\mu\lor \nu),\pi_B\circ (\mu\lor \nu),X(\sigma_{\mu\lor \nu})(x))=(\mu\circ\pi_a,\nu\circ\pi_b,X(\sigma_{\mu\lor \nu})(x)).\]
	Note that $\sigma_{\nu\lor\mu}$ is given by $\tilde\sigma_\nu\lor\tilde\sigma_\mu$, where $\tilde\sigma_\nu\colon a\to a\lor b$ is the inclusion, except that if $\mu(j)=*$, then $\tilde\sigma_\mu(j)=*$, and similarly for $\tilde\sigma_\mu\colon b\to a\lor b$. Here, $(\pi_A\circ(\mu\lor \nu),\pi_B\circ (\mu\lor \nu),\tilde\sigma_\nu,\tilde\sigma_\mu,x)$ is a morphism in $\D(A,B;X)$ from $(\mu,\nu,x)$ to $(\pi_A\circ(\mu\lor \nu),\pi_B\circ (\mu\lor \nu),X(\sigma_{\mu\lor \nu})(x))$. To see that the source of this morphism is as claimed, observe that $\mu\circ\pi_a\circ \tilde\sigma_\mu=\mu$ since $\pi_a\circ\tilde\sigma_\mu=\id$ except on those $j$ such that $\mu(j)=*$, and similarly for $\nu$. The naturality follows from the equations
	\[\tilde\sigma_\mu\circ\phi=(\phi\lor\psi)\circ\tilde\sigma_{\mu\circ\phi} \text{ and } \tilde\sigma_\nu\circ\psi=(\phi\lor\psi)\circ\tilde\sigma_{\nu\circ\psi},\]
	which are easily checked. The continuity of the assignment is also easily verified. For $g\in G$, a verification similar to the previous claim shows that $g$ acting on the component of our natural transformation at $(\mu,\nu,x)$ is the component of the natural transformation at $g.(\mu,\nu,x)=(g.\mu,g.\nu,gx)$. This shows the claim and the proposition.
\end{proof}

\begin{proof}[Proof of Lemma \ref{lem:wedge-lemma}]
	The diagram in Proposition \ref{prop:lemma-to-wedge-lemma} passes after nerve and realization to a diagram 
	\begin{center}
		\begin{tikzcd}
                                         & {|\D(A,B;X)|} \arrow[ld, "f"'] \arrow[rd, "q"] &                                           \\
b_\II^\times X(A\lor B) \arrow[rr, "p"'] &                                             & b_\II^\times X(A)\times b_\II^\times X(B)
\end{tikzcd}
	\end{center}
	which commutes up to homotopy and so that $f$ is a weak $G$-equivalence. The tricky part is to show that $q$ is a weak $G$-equivalence. As noted in \cite{MMO25}, \S 8.4, there is no reason to believe that $Q$ is an equivalence of categories. Instead, we show that the nerve of $q_*=N(Q)$ gives a level $G$-equivalence of simplicial $G$-spaces.
	
	On $0$-simplices, the map $q_0=\ob(Q)$
	\[\bigvee_{(a,b)\in\Gamma_\II^2}(A^a\lor B^b) \land X(a\lor b) \to \Bigg(\bigvee_{a\in\Gamma_\II}A^a\land X(a) \Bigg)\times \Bigg( \bigvee_{b\in\Gamma_\II} B^b\land X(b)\Bigg)\]
	maps $(\mu,\nu,x)$ to $(\mu,x_a)\times (\mu,x_b)$, where $x_a=\pi_a(x)$. Note that since $X$ is an $\II$ -special $\Gamma_\II$-$G$-space, we have a weak equivalence $X(a\lor b)\to X(a)\times X(b)$. Since homotopy groups and fixed points commute with wedges and since 
	\[(A^a\lor B^b)\land X(a\lor b)\cong (A^a\land X(a\lor b)) \lor (B^b\land X(a\lor b))\]
	we see that passage to $\pi_n^H$ gives an isomorphism on homotopy groups for all subgroups $H\subset G$. Hence $q_0$ is a weak $G$-homotopy equivalence.
	
	On $n$-simplices $q_n$ is the same as $q_0$ but with inserted discrete mapping spaces. Using that $(\Gamma_\II\times\Gamma_\II)((a,a'),(b,b'))\cong \Gamma_\II(a,b)\times\Gamma_\II(a',b')$ we find by the same argument above that $q_n$ is a weak $G$-homotopy equivalence. Hence, $q_*$ is a level weak $G$-homotopy equivalence of simplicial $G$-spaces. Using that bar constructions are Reedy cofibrant we get upon passage to realization a $G$-homotopy equivalence of classifying $G$-spaces.
\end{proof}


\subsection{Proof of the delooping theorem}\label{app:proof-delooping-theorem}

We introduce some preliminary notations and constructions in line with \cite{MMO25} and \cite{Shi89}. Recall that a $G$-representation is a finite dimensional real inner product space with linear and isometric $G$-action. Denote by $\langle-,-\rangle$ the inner product in $V$ which induces a $G$-invariant norm denoted by $|-|$.

\begin{Constr}[\cite{MMO25}, \S 8.5]\label{constr:deloop-1}
	Fix a $G$-representation $V$ lying in a fixed $G$-universe $U$ compatible with $\II$. Let $M\subset V$ be a subspace. 
	\begin{enumerate}
		\item Let $\varepsilon>0$. Denote by $M_\varepsilon$ the open $\varepsilon$-neighbourhood of $M$ in $V$, that is,
		\[M_\varepsilon:=\{v\in V\mid \text{there exists $m\in M$ such that }|v-m|<\varepsilon\}.\]
		Note that if $M$ is a $G$-subspace, then $M_\varepsilon$ is also a $G$-subspace.
		\item Denote by $(-)^c$ the one point compactification. This is functorial on proper maps between locally compact spaces. Denote the point at infinity by $\infty$. Recall that $(-)^c$ converts cartesian products into smash products and that $A^c=A_+$ if $A$ is already compact.
		\item Let $\varepsilon>0$. Denote by $0_\varepsilon$ the open $\varepsilon$-neighbourhood of the origin in $V$, i.e. $0_\varepsilon=\{0\}_\varepsilon$. Note that $0_\varepsilon^c$ is a sub $G$-space and there is a $G$-homeomorphism $0_\varepsilon^c \cong S^V$. As in \cite{MMO25}, \S 8.5 we will treat this $G$-homeomorphism as an identification from now on.
		\item For any compact (unbased) $G$-subspace $M$ of $V$ we have a $G$-map $M\to F(0_\varepsilon,M_\varepsilon)$ given by $\ x\mapsto[v\mapsto x+v]$. Applying a Pontryagin-Thom construction gives a based $G$-map
		\[\mu\colon M_+\to F(M_\varepsilon^c,0_\varepsilon^c)\]
		which for $x\in M$ and $y\in M_\varepsilon$, is given by $\mu(x)(y):=y-x$ if $|x-y|<\varepsilon$ and $\mu(x)(y)=\infty$ otherwise. This has an adjoint map
		\[\hat\mu_M\colon Z(M_\varepsilon^c)\to F(M_+,Z(S^V)).\]
		\item For $M=\{0\}$, the map $\mu\colon S^0=M_+\to F(0_\varepsilon^c,0_\varepsilon^c)$ sends $0$ to the identity map and thus $\hat\mu_{\{0\}}$ can be identified with the identity map of $Z(S^V)$.
	\end{enumerate}
\end{Constr}

As usual, denote by $S(V)$ the unit sphere in $V$ and by $D(V)$ the unit disk, that is $v\in S(V)$ if and only if $|v|=1$ and $v\in D(V)$ if and only if $|v|\le 1$. Using $G$-invariance and linearity of the inner product on $V$ these are sub $G$-spaces of $V$.

We continue to introduce notation.

\begin{Constr}[\cite{MMO25}, \S 8.5]\label{constr:deloop-2}
	Let $\pi\colon S(V)_\varepsilon \to S(V)$ be the radial projection given by $\pi(v)=v/|v|$. For a subset $X$ of $S(V)$ let
	\[\widehat{X}:=\pi^{-1}(X)=\{v\in V\mid 1-\varepsilon<|v|<1+\varepsilon \text{ and }\pi(v)\in X\}\subset S_\varepsilon.\]
	Thus $\widehat{X}$ is the subset $X$ on the sphere, thickened by $\varepsilon$ in the radial direction. Note that $\widehat{X}\subset X_\varepsilon$, but these are not equal in general. But they are equal in the case $X=S(V)$. We also define
	\[\widetilde{X}:=X-(S(V)-X)_{2\varepsilon}=\{x\in X\mid |x-v|\ge 2\varepsilon \text{ for all }v\in S(V)-X\} \subset X\subset S(V).\]
	Thus $\widetilde{X}$ is the subset $X$ on the sphere, with a $2\varepsilon$-neighbourhood cropped out of in the tangent direction.
\end{Constr}

\begin{Constr}[\cite{MMO25}, \S 8.5]
	Using the definitions from Construction \ref{constr:deloop-1} and \ref{constr:deloop-2}, for any subset $X$ of $S(V)$ the $G$-map of Construction \ref{constr:deloop-1} (4) with $M=S(V)$ yields a map $S(V)\to F(0_\varepsilon,S(V)_\varepsilon)$. One can show that this restricts to a $G$-map
	\[\alpha_X\colon \widetilde{X}\to F(0_\varepsilon,\widehat{X}).\]
	Applying the Pontryagin-Thom construction yields a based $G$-map
	\[\mu_X\colon \widetilde{X}_+\to F(\widehat{X}^c,0_\varepsilon^c)=F(\widehat{X}^c_+,S^V),\]
	which again gives rise to a based $G$-map
	\[\hat{\mu}_X\colon Z(\widetilde{X}_+)\to F(\widehat{X}^c_+,Z(S^V))\]
	for any subset $X$ of $S(V)$. When $X=S(V)$, this is the map $\hat{\mu}_{S(V)}$ that we constructed before by setting $M=S(V)$ in Construction \ref{constr:deloop-1} (4).
\end{Constr}

To proof the delooping theorem we need to study triangulations of $S(V)$. Since $S(V)$ is a smooth compact $G$-manifold there exists a finite triangulation $K$ together with a $G$-homeomorphism $h\colon |K|\to S(V)$. In particular, $S(V)$ admits the structure of a finite based $G$-CW complex. Since $h$ is a $G$-homeomorphism it preserves isotropy by Proposition \ref{prop:isotropy-homeomorphism} so that $S(V)$ is a finite based $\II$-$G$-CW complex. Recall the following general terminology which in the equivariant setting is due to \cite{Ill78}.

\begin{Def}
	Given some simplicial $G$-complex $K$ and some simplex $s\in K$, the \textit{star} of $s$ in $K$, denoted $\textup{st}(s,K)$, is the union of simplices $t\in K$ with face $s$. The \textit{open star} of $s$ in $K$ is the interior of $\textup{st}(s,K)$.
\end{Def}

The following observation and remark are due to \cite{MMO25}, proof of Theorem 8.32.

\begin{Obs}\label{obs:disjoint-open-stars}
	Note that the set of all open stars in $K$ forms an open covering of $K$. As we have chosen $K$ to be finite, the covering is also finite. As remarked in \cite{MMO25}, proof of Theorem 8.32, after possibly refining $K$ to some $K'$, we can assume that for any simplex $s\in K'$ either 
	\[g\textup{st}^\circ(s,K)=\textup{st}^\circ(s,K) \text{ or }g\textup{st}^\circ(s,K)\cap \textup{st}^\circ(s,K)=\emptyset.\]
	We will from now on assume that $K$ fulfils this property.
\end{Obs}

\begin{Rem}\label{rem:indexing-on-barycenters}
	Each simplex $s$ has a barycentre denoted $\lambda_s$. The collection of open stars can therefore be indexed on barycentres $\lambda_s$ of simplices in $s$ in $K$. In particular, we can index the open covering of open stars on these $\lambda_s$. From now on write $\lambda$ for the indexing elements and $T$ for the indexing set and following \cite{MMO25}, Theorem 8.32, we write $C_\lambda$ for the open star of the simplex with barycentre $\lambda$.
\end{Rem}

\begin{Lem}\label{lem:MMO-Thm.8.32}
	Let $Z$ be an $\II$-special $\underline{\CW}^\II_*$-$G$-space. Then the map
	\[\hat{\mu}_S\colon Z(S(V)^c_\varepsilon)\to F(S(V)_+,Z(S^V))\]
	is a weak $G$-equivalence.
\end{Lem}

\begin{proof}
	Using Observation \ref{obs:disjoint-open-stars} and Remark \ref{rem:indexing-on-barycenters} we may choose a finite triangulation $K$ of $S(V)$ such that for any $s\in K$ either $g\textup{st}^\circ(s,K)=\textup{st}^\circ(s,K)$ or $g\textup{st}^\circ(s,K)\cap \textup{st}^\circ(s,K)=\emptyset$. Let $T$ be an indexing set of open stars in $K$. Let $X$ be a $G$-stable union of open stars $C_\lambda$, that is, $gX=X$. Let $\Phi(X)=\{H\subset G\mid X^H\ne \emptyset\}$ be the isotropy of $X$. Since $X\subset S(V)$ also $\Phi(X)\subset \Phi(S(V))$. Then for each $\lambda\in T$ we know that $gC_\lambda=C_\lambda$ or they are disjoint. Hence, $g\in H$ for some $H\in \Phi(X)$. Since $X$ is $G$-stable, we know that $G/H\subset T$. Since $T$ is finite $T\cong G/H_1\sqcup \dots G/H_n$ with $H_i\in\Phi(X)$. Since $\Phi(X)\subset\Phi(S(V))\subset\Phi(V)$, $T$ is a finite $G$-set with isotropy in $\FF_\II$.
	
	Let $2\varepsilon>0$ small enough so that it is smaller than all of the radii of all simplices in the triangulation. For any $G$-stable union $X$ of $C_\lambda$'s this implies that $X$ deformation retracts onto $\widetilde{X}$ (recall that $\widetilde{X}=X-(S(V)-X)_{2\varepsilon}$). We claim that for any such $X$ the map $\hat\mu_X$ is a weak $G$-equivalence. In particular for $X=S(V)$. One proceeds by induction on the number of orbits contained in the indexing $G$-set $T$ with isotropy in $\FF_\II$.
	
	\underline{Base case:} Let $X=\bigcup_{\lambda\in T}C_\lambda$, where $T=G/H\in \II(G)$. By assumption, the inclusion $T\to X$ induces $G$-equivalences $\smash{T_+\to \widetilde{X}_+}$ and $\smash{T_+\land S^V\to \widehat{X}^c}$. Via these equivalences, $\hat{\mu}_X$ can be identified with
	\[Z(T_+\land S^V)\to F(T_+,Z(S^V)),\]
	which is a $G$-equivalence since $Z[S^V]$ is an $\II$-special $G\underline{\CW}_*^\II$-$G$-space using the structure theorem (Theorem \ref{thm:structure-theorem}).
	
	\underline{\smash{Inductive Step}}: Assume that the conclusion holds for $X$, $Y$ and $X\cap Y$ and prove that the conclusion holds for $X\cup Y$. Since the open stars of a triangulation form a covering in the sense that the intersection of any two open stars is a star, so that $X\cap Y$ is indeed also a $G$-stable union of stars, possibly empty. Note that if $T_X$ and $T_Y$ are indexing $\II$-sets of $X$ and $Y$, then $T_{X\cap Y}\subset T_X$ and $T_Y$. Since $\II$ is closed under subobjects (I5) this is also a finite $G$-set with isotropy in $\FF_\II$ and the union $T_{X\cup Y}\subset T_X\sqcup T_Y$ is also a finite $G$-set with isotropy in $\FF_\II$ using (I6) and (I5). The rest of the proof is the same as in \cite{MMO25}, Theorem 8.32. By assumption $\hat{\mu}_X,\hat{\mu}_Y$ and $\hat{\mu}_{X\cap Y}$ are weak $G$-equivalences. If we show that the front and back squares of the following diagram are homotopy cartesian then $\mu_{X\cup Y}$ is also a weak equivalence
	\begin{center}
		\begin{tikzcd}
			Z((\widehat{X\cup Y})^c) \arrow[dd] \arrow[rd, "\hat\mu_{X\cup Y}"] \arrow[rrr] &                                                          &  & Z(\widehat{X}^c) \arrow[rd, "\hat\mu_X"] \arrow[dd]      &                                        \\& {F(\widetilde{X\cup Y}_+,Z(S^V))} \arrow[dd] \arrow[rrr] &  &                                                          & {F(\widetilde{X}_+,Z(S^V))} \arrow[dd] \\
			Z(\widehat{Y}^c) \arrow[rd, "\hat\mu_Y"] \arrow[rrr]                            &                                                          &  & Z((\widehat{X\cap Y})^c) \arrow[rd, "\hat\mu_{X\cap Y}"] &                                        \\& {F(\widetilde{Y}_+,Z(S^V))} \arrow[rrr]                  &  &                                                          & {F(\widetilde{X\cap Y}_+,Z(S^V))}     
		\end{tikzcd}
	\end{center}
	The proposition then follows inductively for $S(V)$.
	
	The front square is homotopy cartesian since the square
	\begin{center}
		\begin{tikzcd}
			\widetilde{X\cap Y} \arrow[d] \arrow[r] & \widetilde{X} \arrow[d] \\
			\widetilde{Y} \arrow[r]                 & \widetilde{X\cup Y}    
		\end{tikzcd}
	\end{center}
	is homotopy cocartesian, since $X$, $Y$ are open in $X\cup Y$, and $\varepsilon$ was chosen small enough so that there is a $G$-deformation retract $\widetilde{W}\simeq W$ for all $G$-stable unions of open stars $W$ in the triangulation. Hence, application with $F(-,Z(S^V))$ yields a cartesian square.
	
	To see that the back square is homotopy cocartesian note that
	\[(\widehat{X\cup Y}-\widehat{X})^c\to \widehat{X\cup Y})^c \text{ and } (\widehat{Y}-\widehat{X\cap Y})^c \to \widehat{Y}^c\]
	are $G$-cofibrations, just as in \cite{MMO25}. Moreover, we have that $\widehat{X\cup Y}-\widehat{X} = \widehat{Y}-\widehat{X\cap Y}$ since both sets consist of $v\in S(V)_\varepsilon$ such that $\pi(v)$ is in $Y$ but not in $X$. After identifying the cofibres, we then get a diagram of homotopy cofibre sequences
	\begin{center}
		\begin{tikzcd}
			(\widehat{X\cup Y}-\widehat{X})^c \arrow[r] \arrow[d, Rightarrow, no head] & (\widehat{X\cup Y})^c \arrow[r] \arrow[d] & \widehat{X}^c \arrow[d] \\
			(\widehat{Y}-\widehat{X\cap Y})^c \arrow[r]                                & \widehat{Y}^c \arrow[r]                   & (\widehat{X\cap Y})^c  
		\end{tikzcd}
	\end{center}
	Note that alls of these spaces in the diagram have isotropy type in $\FF_\II$ since before compactifying they are bounded subsets of $V$. Therefore the compactifications admit structures based of $\II$-$G$-CW complexes so that we may apply $Z$ to obtain a diagram of fibre sequences
	\begin{center}
		\begin{tikzcd}
			Z((\widehat{X\cup Y}-\widehat{X})^c) \arrow[r] \arrow[d,  Rightarrow, no head] & Z((\widehat{X\cup Y})^c) \arrow[r] \arrow[d] & Z(\widehat{X}^c) \arrow[d] \\
			Z((\widehat{Y}-\widehat{X\cap Y})^c) \arrow[r]                                & Z(\widehat{Y}^c) \arrow[r]                   & Z((\widehat{X\cap Y})^c)  
		\end{tikzcd}\qedhere
	\end{center}
\end{proof}

\begin{proof}[Proof of Theorem \ref{thm:delooping-theorem}]
	Let $S(V)$ and $D(V)$ be the unit sphere and unit disk in $V$, and let $i\colon S(V)\to D(V)$ be the inclusion. Define $p\colon D(V)_\varepsilon^c\to S(V)_\varepsilon^c$ by $p(x)=x$ if $x\in S(V)_\varepsilon$ and $p(x)=\infty$ otherwise. The rest of the proof follows \cite{MMO25}, Theorem 8.27 verbatim adding some remarks on isotropy. The following diagram commutes
	\begin{center}
		\begin{tikzcd}
			S(V)_+ \arrow[d, "i"'] \arrow[r, "\mu"] & {F(S(V)_\varepsilon^c,0_\varepsilon^c)} \arrow[d, "p^*"] \\
			D(V)_+ \arrow[r, "\mu"']                & {F(D(V)_\varepsilon^c,0_\varepsilon^c)}                 
		\end{tikzcd}
	\end{center}
	Note that $D(V)_\varepsilon\backslash S(V)_\varepsilon \to D(V)_\varepsilon$ is proper. Let $j\colon (D(V)_\varepsilon\backslash S(V)_\varepsilon)^c\to D(V)_\varepsilon^c$ be the inclusion and observe that the quotient of $j$ is the $G$-space $S_\varepsilon^c$ with quotient map $p$. Further define
	\[\nu\colon S^V\cong D(V)/S(V) \to F((D(V)_\varepsilon\backslash S(V)_\varepsilon)^c,0_\varepsilon^c)\]
	by applying a relative Pontryagin-Thom construction. Explicitly, $\nu$ is the based $G$-map such that if $x\in D(V)\backslash S(V)$ and $y\in D(V)_\varepsilon\backslash S(V)_\varepsilon$, then $\nu(x)(y)=y-x$ if $|y-x|<\varepsilon$ and $v(x)(y)=\infty$ otherwise. Then the following diagram commutes, where $q$ is the quotient map.
	\begin{center}
		\begin{tikzcd}
			D(V)_+ \arrow[r, "\mu"] \arrow[d, "q"'] & {F(D(V)_\varepsilon^c,0_\varepsilon^c)} \arrow[d, "j^*"]             \\
			S^V\cong D(V)/S(V) \arrow[r, "\nu"]     & {F((D(V)_\varepsilon\backslash S(V)_\varepsilon)^c,0_\varepsilon^c)}
		\end{tikzcd}
	\end{center}
	Just as $\mu$ gives rise to a map $\hat{\mu}$, $\nu$ gives rise to a map $\hat{\nu}$ in the following diagram. By easy diagram chases, writing $S^V$ for $0_\varepsilon^c$, the squares above imply that the squares in the diagram commute
	\begin{center}
		\begin{tikzcd}
			Z((D(V)_\varepsilon\backslash S(V)_\varepsilon)^c) \arrow[d, "Z(i)"'] \arrow[r, "\hat\nu"] & {F(D(V)/S(V),Z(S^V))} \arrow[d, "q^*"] \\
			Z(D(V)_\varepsilon^c) \arrow[d, "Z(p)"'] \arrow[r, "\hat\mu"]                              & {F(D(V)_+,Z(S^V))} \arrow[d, "j^*"]    \\
			Z(S(V)^c_\varepsilon) \arrow[r, "\hat\mu"]                                                 & {F(S(V)_+,Z(S^V))}                    
		\end{tikzcd}
	\end{center}
	The right column is clearly a fibre sequence. We claim that
	\[(D(V)_\varepsilon\backslash S(V)_\varepsilon)^c \xrightarrow{j}D(V)_\varepsilon^c \xrightarrow{p}S(V)^c_\varepsilon \tag{$*$}\]
	is a cofibre sequence so that we can use the assumption that $Z$ is a positive linear $G\underline{\CW}_*^\II$-$G$-space to conclude the proof. For this we again remark that before compactifying all spaces are subsets of $V$ so that isotropy is preserved and applying compactification actually yield based $\II$-G-CW complexes. Since $Z$ is $\II$-linear, under this assumption, this would imply that the left column is also a fibre sequence.
	
	Observe that $D(V)_\varepsilon\backslash S(V)_\varepsilon$ is the closed disk $D(1-\varepsilon)$ of radius $1-\varepsilon$ in $V$. Applying compactification adds a disjoint base point. Therefore the inclusion $S^0=\{0\}\sqcup\{\infty\}\to (D(V)_\varepsilon\backslash S(V)_\varepsilon)^c\cong D(1-\varepsilon)_+$ is a $G$-homotopy equivalence. The inclusion $j\colon D(1-\varepsilon)_+\to D_\varepsilon^c\cong S^V$ which sends the disjoint base point to the north pole (the point at $\infty$) and sends the disk of radius $1-\varepsilon$ to its image around the south pole (the point $0\in V$), is a $G$-cofibration. We conclude that the sequence ($*$) is $G$-equivalent to the cofibre sequence
	\[S^0\to S^V\to S^V/S^0.\]
	
	Hence the two vertical sequences in the diagram above are fibre sequences. The bottom horizontal map $\hat{\mu}$ is a weak $G$-equivalence by Lemma \ref{lem:MMO-Thm.8.32}. The middle horizontal map $\hat{\mu}$ is easily seen to be a weak equivalence since $D(V)_\varepsilon^c\cong S^V$ and $D(V)_+\simeq S^0$. Hence, $\hat{\nu}$ is also a weak $G$-equivalence. Its domain is weakly $G$-equivalent to $Z(S^0)$ and its target is weakly $G$-equivalent to $\Omega^V Z(S^V)$. Under these equivalences $\hat{\nu}$ agrees with $\tilde{\sigma}$ at least up to sign as demonstrated in \cite{MMO25}\footnote{As remarked in the source, the standard choice of identification $S^V\cong D(V)/S(V)$ would introduce a sign, but this is irrelevant for the conclusion that $\tilde{\sigma}$ is a weak $G$-equivalence if $\hat{\nu}$ is a weak $G$-equivalence.}.
\end{proof}

\end{appendices}


\bibliographystyle{alpha}
\bibliography{references}

\begin{thebibliography}{GMMO19}

\bibitem[Bal25]{Bal25}
S.~Balchin.
\newblock {Ninfty: A software package for homotopical combinatorics}, 2025.
\newblock Available at arXiv:2504.01003v1.

\bibitem[BH15]{BH15}
A.J. Blumberg and M.A. Hill.
\newblock {Operadic multiplications in equivariant spectra, norms and transfers}.
\newblock {\em Advances in Mathematics}, 285:658--708, 2015.

\bibitem[BH21]{BH21}
A.J. Blumberg and M.A. Hill.
\newblock {Bi-Incomplete Tambara functors}, 2021.
\newblock Available at arXiv:2104.10521.

\bibitem[BHK24]{BHK24}
D.~Barnes, M.A. Hill, and M.~Kedziorek.
\newblock {Splitting rational incomplete Mackey functors}, 2024.
\newblock Available at arXiv:2410.10962.

\bibitem[Blu17]{Blu17}
A.J. Blumberg.
\newblock {Lecture Notes on Equivariant Stable Homotopy Theory}, 2017.
\newblock Available at \url{https://github.com/adebray/equivariant_homotopy_theory}. Last accessed August 29, 2025.

\bibitem[BO15]{BO15}
A.M. Bohmann and A.~Osorno.
\newblock {Constructing equivariant spectra via categorical Mackey functors}.
\newblock {\em Algebraic \& Geometric Topology}, 15:537--563, 2015.

\bibitem[EM06]{EM06}
A.D. Elmendorf and M.A. Mandell.
\newblock {Rings, modules, and algebras in infinite loop space theory}.
\newblock {\em Advances in Mathematics}, 205(1):163--228, 2006.

\bibitem[GM17]{GM17}
B.J Guillou and J.P. May.
\newblock {Equivariant iterated loop space theory and permutative $G$-categories}.
\newblock {\em Algebraic \& Geometric Topology}, 17(6):3259--3339, 2017.

\bibitem[GMMO18]{GMMO18}
B.J Guillou, J.P. May, M.~Merling, and A.M. Osorno.
\newblock {Symmetric monoidal $G$-categories and their strictification}, 2018.
\newblock Available at arXiv:1809.03017.

\bibitem[GMMO19]{GMMO19b}
B.J Guillou, J.P. May, M.~Merling, and A.M. Osorno.
\newblock {A symmetric monoidal and equivariant Segal infinite loop space machine}.
\newblock {\em Journal of Pure and Applied Algebra}, 223(6):2425--2454, 2019.

\bibitem[HH16]{HH16}
M.A. Hill and M.J. Hopkins.
\newblock {Equivariant symmetric monoidal structures}, 2016.
\newblock Available at arXiv:1610.03114.

\bibitem[Ill78]{Ill78}
S.~Illman.
\newblock {Smooth Equivariant Triangulations of $G$-manifolds for $G$ a Finite Group}.
\newblock {\em Mathematische Annalen}, 233:199--220, 1978.

\bibitem[Joh02]{Joh02}
P.~Johnstone.
\newblock {\em {Sketches of an Elephant – A Topos Theory Compendium}}.
\newblock Oxford University Press, 2002.

\bibitem[Lor21]{Lor21}
F.~Loregian.
\newblock {\em {(Co)end calculus}}.
\newblock London Mathematical Society Lecture Note Series. Cambridge University Press, 2021.

\bibitem[Mat71]{Mat71}
T.~Matumoto.
\newblock {Equivariant K-theory and Fredholm operators}.
\newblock {\em J. Fac. Sci. Tokyo}, 18:109--112, 1971.

\bibitem[May75]{May75}
J.P. May.
\newblock {\em {Classifying spaces and fibrations}}, volume~1 of {\em Memoirs of the American Mathematical Society}.
\newblock American Mathematical Society, 1975.

\bibitem[May78]{May78}
J.P. May.
\newblock {The spectra associated to permutative categories}.
\newblock {\em Topology}, 17(3):225--228, 1978.

\bibitem[May96]{May96}
J.P. May.
\newblock {\em {Equivariant Homotopy and Cohomology Theory}}, volume~91 of {\em Conference Board of the Mathematical Sciences}.
\newblock American Mathematical Society, 1996.

\bibitem[MM02]{MM02}
M.A. Mandell and J.P. May.
\newblock {\em {Equivariant orthogonal spectra and $S$-modules}}, volume 159 of {\em Memoirs of the American Mathematical Society}.
\newblock American Mathematical Society, 2002.

\bibitem[MMO25]{MMO25}
J.P. May, M.~Merling, and A.M. Osorno.
\newblock {\em {Equivariant Infinite Loop Space Theory: The Space Level Story}}, volume 305 of {\em Memoirs of the American Mathematical Society}.
\newblock American Mathematical Society, 2025.

\bibitem[MMSS01]{MMSS01}
M.A. Mandell, J.P. May, S.~Schwede, and B.~Shipley.
\newblock {Model Categories of Diagram Spectra}.
\newblock {\em Proc. London Math. Soc.}, 82(2):441--512, 2001.

\bibitem[MS76]{DS76}
D.~McDuff and G.~Segal.
\newblock {Homology Fibrations and the "Group-Completion" Theorem}.
\newblock {\em Inventiones mathematicae}, 31(3):279--284, 1976.

\bibitem[Ric20]{Ric20}
B.~Richter.
\newblock {\em {From Categories to Homotopy Theory}}.
\newblock Cambridge Studies in Advanced Mathematics. Cambridge university Press, 2020.

\bibitem[Rub21a]{Rub21b}
J.~Rubin.
\newblock {Combinatorial $N_\infty$ operads}.
\newblock {\em Algebraic \& Geometric Topology}, 21:3513--3568, 2021.

\bibitem[Rub21b]{Rub21a}
J.~Rubin.
\newblock {Detecting Steiner and Linear Isometries Operads}.
\newblock {\em Glasgow Mathematical Journal}, 63(2):307--342, 2021.

\bibitem[Rub25]{Rub25}
J.~Rubin.
\newblock {Normed Symmetric Monoidal Categories}.
\newblock {\em Journal of Homotopy and Related Structures}, 20(2):195--250, 2025.

\bibitem[Sch18]{Sch18}
S.~Schwede.
\newblock {\em {Global Homotopy Theory}}, volume~34 of {\em New Mathematical Monographs}.
\newblock Cambridge University Press, 2018.

\bibitem[Sch23]{Sch23}
S.~Schwede.
\newblock {Lecture notes on equivariant stable homotopy theory}, 2023.
\newblock Available at \url{https://www.math.uni-bonn.de/~schwede/equivariant.pdf}. Last accessed October 7, 2025.

\bibitem[Seg74]{Seg74}
G.~Segal.
\newblock {Categories and Cohomology Theories}.
\newblock {\em Topology}, 13:293--312, 1974.

\bibitem[Shi89]{Shi89}
K.~Shimakawa.
\newblock {Infinite Loop $G$-spaces Associated to Monoidal $G$-Graded Categories}.
\newblock {\em Publ. Res. Inst. Math. Sci.}, 25(2):239--262, 1989.

\bibitem[Shi91]{Shi91}
K.~Shimakawa.
\newblock {A note on $\Gamma_G$-spaces}.
\newblock {\em Osaka Journal of Mathematics}, 28(2):223--228, 1991.

\bibitem[Shu09]{Shu09}
M.~Shulman.
\newblock {Homotopy limits and colimits and enriched homotopy theory}, 2009.
\newblock Available at arXiv:0610194v3.

\bibitem[tD87]{Die87}
T.~tom Dieck.
\newblock {\em {Transformation groups}}, volume~8 of {\em De Gruyter studies in Mathematics}.
\newblock De Gruyter, 1987.

\end{thebibliography}

\end{document}